\documentclass[11pt]{article}

\usepackage[margin=1in]{geometry}

\usepackage{amsmath,amsthm} 
\usepackage{amssymb}
\usepackage{mathrsfs}
\usepackage[all]{xypic} %For some commutative diagrams
\usepackage{amsbsy}
\usepackage{graphicx}  %To import figures, rotate   %and for scalebox
\usepackage{subfig} %To format figures side by side
\usepackage[all,cmtip]{xy} %For matrices in xymatrix
\usepackage{multirow}  %For table labels
\usepackage{multicol}
\usepackage{rotating} %for sideways tables
\usepackage{comment} %to comment large sections
\usepackage{relsize} %For making math symbols smaller
\usepackage{hyperref} %For email at the end
\usepackage{xcolor} %For colored text
\usepackage{lscape} %To turn some pages landscape
\usepackage{manfnt} %for the dbend  
\usepackage{fancyvrb}  %for fancy verbatim
 \usepackage{url}

\usepackage{tikz}
\usetikzlibrary{arrows, arrows.meta, decorations.markings, decorations.pathmorphing, shapes} 
\usetikzlibrary{calc}
\usepackage{graphicx}
\definecolor{vertexred}{RGB}{235,70,50}
\definecolor{stubblue}{RGB}{35,110,230}
\definecolor{partpurple}{RGB}{200,130,230}
\definecolor{partgreen}{RGB}{145,190,165}
\definecolor{pathorange}{RGB}{255,216,158}
\tikzset{
 graph edge/.style={draw=black,line width=1.2pt,line cap=round,line join=round},
 graph vertex/.style={circle,fill=vertexred,draw=none,minimum size=2.5mm,inner sep=0pt},
 purple vertex/.style={graph vertex,fill=partpurple},
 green vertex/.style={graph vertex,fill=partgreen},
 grey vertex/.style={graph vertex,fill=gray!60},
 grey edge/.style={draw=gray!60,line width=1.2pt,line cap=round},
 external stub/.style={draw=stubblue,line width=1.4pt,line cap=round},
 path mark/.style={draw=pathorange,line width=9pt,line cap=round,line join=round},
}
\newcommand{\GOuterCoords}{%
  \coordinate (TL) at (0.000,3.500);
  \coordinate (TM) at (1.300,3.500);
  \coordinate (TR) at (2.600,3.500);
  \coordinate (RM) at (2.600,1.750);
  \coordinate (BR) at (2.600,0.000);
  \coordinate (BM) at (1.300,0.000);
  \coordinate (BL) at (0.000,0.000);
  \coordinate (LM) at (0.000,1.750);
}
\newcommand{\GCoordsH}{%
  \GOuterCoords
  \coordinate (u) at (1.300,2.550);
  \coordinate (v) at (1.300,0.970);
}
\newcommand{\GCoordsHone}{%
  \GOuterCoords
  \coordinate (u) at (1.300,2.580);
  \coordinate (v) at (1.300,1.700);
  \coordinate (l) at (0.660,1.100);
  \coordinate (r) at (1.940,1.100);
}
\newcommand{\GCoordsHtwo}{%
  \GOuterCoords
  \coordinate (a) at (0.900,2.660);
  \coordinate (b) at (1.900,2.660);
  \coordinate (c) at (1.500,2.050);
  \coordinate (d) at (0.920,1.400);
  \coordinate (e) at (1.720,1.200);
  \coordinate (f) at (0.600,0.670);
}
\newcommand{\GCoordsHthree}{%
  \GOuterCoords
  \coordinate (a) at (0.760,2.600);
  \coordinate (b) at (1.860,2.570);
  \coordinate (c) at (1.100,2.030);
  \coordinate (d) at (1.550,1.480);
  \coordinate (e) at (0.770,1.060);
  \coordinate (f) at (1.910,1.060);
}
\newcommand{\GOuterEdges}{%
 \draw[graph edge] (TL)--(TM)--(TR)--(RM)--(BR)--(BM)--(BL)--(LM)--cycle;
 \draw[external stub] (TL)--++(-.45,0) (TL)--++(0,.45);
 \draw[external stub] (TR)--++(.45,0) (TR)--++(0,.45);
 \draw[external stub] (BL)--++(-.45,0) (BL)--++(0,-.45);
 \draw[external stub] (BR)--++(.45,0) (BR)--++(0,-.45);
}
\newcommand{\GEdgesH}{%
  \GOuterEdges
  \draw[graph edge]
    (TM)--(u)
    (LM)--(u)
    (RM)--(u)
    (LM)--(v)
    (RM)--(v)
    (BM)--(v)
  ;
}
\newcommand{\GNodesH}{%
  \foreach \n in {TL,TM,TR,RM,BR,BM,BL,LM,u,v}{\node[graph vertex] at (\n) {};}
}
\newcommand{\GEdgesHone}{%
  \GOuterEdges
  \draw[graph edge]
    (TM)--(u)
    (LM)--(u)
    (RM)--(u)
    (u)--(v)
    (LM)--(l)
    (l)--(v)
    (v)--(r)
    (r)--(RM)
    (l)--(BM)
    (r)--(BM)
  ;
}
\newcommand{\GNodesHone}{%
  \foreach \n in {TL,TM,TR,RM,BR,BM,BL,LM,u,v,l,r}{\node[graph vertex] at (\n) {};}
}
\newcommand{\GEdgesHtwo}{%
  \GOuterEdges
  \draw[graph edge]
    (TM)--(b)
    (b)--(RM)
    (a)--(b)
    (a)--(c)
    (a)--(d)
    (c)--(RM)
    (c)--(e)
    (d)--(e)
    (LM)--(d)
    (LM)--(f)
    (f)--(BM)
  ;
  \draw[graph edge,line width=1.6pt] (LM) .. controls (0.15,3.08) and (1.05,3.26) .. (b);
  \draw[graph edge,line width=1.6pt] (f) .. controls (0.84,0.88) and (1.3,0.96) .. (e);
  \draw[graph edge,line width=1.6pt] (f) .. controls (1.05,0.38) and (2.4,0.42) .. (RM);
}
\newcommand{\GNodesHtwo}{%
  \foreach \n in {TL,TM,TR,RM,BR,BM,BL,LM,a,b,c,d,e,f}{\node[graph vertex] at (\n) {};}
}
\newcommand{\GEdgesHthree}{%
  \GOuterEdges
  \draw[graph edge]
    (TM)--(a)
    (TM)--(b)
    (LM)--(a)
    (a)--(c)
    (c)--(b)
    (b)--(RM)
    (b)--(d)
    (c)--(e)
    (e)--(d)
    (d)--(f)
    (LM)--(e)
    (e)--(BM)
    (RM)--(f)
    (f)--(BM)
  ;
}
\newcommand{\GNodesHthree}{%
  \foreach \n in {TL,TM,TR,RM,BR,BM,BL,LM,a,b,c,d,e,f}{\node[graph vertex] at (\n) {};}
}
\newcommand{\GMarks}[1]{%
 \foreach \aaa/\bbb in {#1}{\draw[path mark] (\aaa)--(\bbb);}%
}
\newcommand{\GGrayEdges}[1]{%
 \foreach \aaa/\bbb in {#1}{\draw[grey edge] (\aaa)--(\bbb);}%
}
\newcommand{\GGrayNodes}[1]{%
 \foreach \nn in {#1}{\node[grey vertex] at (\nn) {};}%
}
\newcommand{\GBipartiteH}{%
 \foreach \v in {TM,LM,RM,BM}{\node[purple vertex] at (\v) {};}%
 \foreach \v in {TL,TR,BL,BR,u,v}{\node[green vertex] at (\v) {};}%
}
\newcommand{\GBipartiteHone}{%
 \foreach \v in {TM,LM,RM,BM,v}{\node[purple vertex] at (\v) {};}%
 \foreach \v in {TL,TR,BL,BR,u,l,r}{\node[green vertex] at (\v) {};}%
}
\newcommand{\GBipartiteHtwo}{%
 \foreach \v in {TM,LM,RM,BM,a,e}{\node[purple vertex] at (\v) {};}%
 \foreach \v in {TL,TR,BL,BR,b,c,d,f}{\node[green vertex] at (\v) {};}%
}
\newcommand{\GBipartiteHthree}{%
 \foreach \v in {TM,LM,RM,BM,c,d}{\node[purple vertex] at (\v) {};}%
 \foreach \v in {TL,TR,BL,BR,a,b,e,f}{\node[green vertex] at (\v) {};}%
}

\numberwithin{equation}{section}

\newtheorem{thm}{Theorem}[section]
\newtheorem{cor}[thm]{Corollary}
\newtheorem{conj}[thm]{Conjecture}
\newtheorem{lem}[thm]{Lemma}
\newtheorem{prop}[thm]{Proposition}
\theoremstyle{definition} 
\newtheorem{defn}[thm]{Definition}
\newtheorem{defn*}{Definition}
\newtheorem{rem}[thm]{Remark}

\newtheorem{exam*}{Example}

\newtheorem{exercise*}{Solution to Exercise}

\newcommand{\bC}{\mathbb{C}} 
\newcommand{\bF}{\mathbb{F}} 
\newcommand{\bN}{\mathbb{N}} 
\newcommand{\bQ}{\mathbb{Q}} 
\newcommand{\bR}{\mathbb{R}}
\newcommand{\bZ}{\mathbb{Z}}
\newcommand{\nz}[1]{\mathbb{Z} \! \hspace{.03cm} / \! \hspace*{.03cm}  #1\mathbb{Z}}
\newcommand{\Hom}{\text{Hom}}
\newcommand{\Sym}{\text{Sym}}
\newcommand{\Rep}{\text{Rep}}
\newcommand{\Mat}{\text{Mat}}
\newcommand{\GL}{\text{GL}}
\newcommand{\SL}{\text{SL}}
\newcommand{\Aut}{\text{Aut}}
\newcommand{\Diag}{\text{Diag}}
\newcommand{\tr}{\text{tr}}

\begin{document}

\title{Local modifications in planar hypohamiltonian graphs}
\author{Samantha Durias, Sabrina Guerrero, Briana Jaiyen,
Daniel Kaplan, \\ Austin Lapine, Caleb Peralta, Akul Shrivastav
}%In mathematics names are listed alphabetically, not based on seniority nor contribution size.
\date{October 2026}
\maketitle

\begin{abstract}

We prove that for $n \geq 36$ there exists a planar hypohamiltonian graph of order $n$. As a corollary, we prove that for $n \geq 33$ there exists an almost hypohamiltonian planar graph of order $n$. Further, we exhibit over 100 planar hypohamiltonian graphs of order less than 40, including new graphs of orders 36, 37, 38, and 39. 

The proof utilizes three local modifications that alter a graph $G$ by removing a subgraph $H$ and replacing it with $H'$ to form a new graph $G'$: a plus two modification, a plus three modification, and Thomassen's iterable plus four modification. The plus two modification preserves hypohamiltonicity in general, and the plus three modification preserves hypohamiltonicity in our cases of interest. 
This establishes the main result beginning with Tsai's order 34 examples. 
%We prove that if $G$ is hypohamiltonian then so is $G'$, except possibly at . Modifying a certain 4-gon configuration increases the order by two and modifying a certain 5-gon configuration increases the order by three. Together with Thomassen's iterable modification that increases the order by four, this proves the result beginning with Tsai's order 34 examples. 

Further, we conjecture that among planar hypohamiltonian graphs, Tsai's  34 vertex examples are the smallest order. The conjecture is based on previous results, constraints on the degree modulo 3, and experimentation with planar graphs consisting entirely of 4-gons and 5-gons. 

\end{abstract}

\section{Introduction}
Let $G =(V, E)$ be a graph with vertex set $V$ and edge set $E$.
A \emph{Hamiltonian cycle} $C$ in $G$ is a cycle that visits each vertex exactly once. A graph is \emph{Hamiltonian} if it has a Hamiltonian cycle. Write $G-v$ for the graph obtained from $G$ by removing the vertex $v$ and all edges adjacent to $v$. 

\begin{defn}
A graph $G =(V, E)$ is \emph{hypohamiltonian} if
\begin{itemize}
    \item[(I)] $G$ is not Hamiltonian, and
    \item[(II)] $G - v$ is Hamiltonian for all $v \in V$.
\end{itemize}
\end{defn}

In 1973, Chv\'{a}tal asked whether planar hypohamiltonian graphs exist. 
Thomassen answered the question in the affirmative, producing infinitely many examples, the smallest of which has order 105 \cite{T76}

Following this work, there have been several successful attempts to find smaller planar hypohamiltonian graphs:
\begin{itemize}
    \item Hatzel in 1979 found a graph of order 57 \cite{H79}.
    \item Araya and Wiener in 2011 found a graph of order 42 \cite{AW11}.
    \item Jooyandeh, McKay, \"{O}stergard, Pettersson, and Zamfirescu in 2017 found 25 such graphs of order 40 \cite{JMOPZ17}.
    \item Tsai in 2025 found two graphs of order 34 and six of order 37 \cite{Tsai25}.
\end{itemize}

Additionally, Tsai, building on \cite{JMOPZ17} and a construction of Thomassen, proved there exists a planar hypohamiltonian graph for all orders $n \geq 40$. Tsai finishes his paper saying: ``Whether there exist planar hypohamiltonian graphs on 35, 36, and 39 vertices remains an open question." \cite[page 806]{Tsai25}

\begin{thm}
    There exists a planar hypohamiltonian graph for all orders $n \geq 36$.
\end{thm}

We prove the result by modifying patches of existing hypohamiltonian graphs while preserving the hypohamiltonian property. In the process, we construct the first such planar graphs of order 36 and 39 and find new graphs of order 37 and 38. In total, we construct over 100 planar hypohamiltonian graphs of order less than 40. 

%Let $G$ be a graph with subgraph $H$. Let $H'$ be a graph with the same boundary as $H$. Write $G'$ for the graph obtained from $G$ by removing $H$ and replacing it with $H'$. We refer to this construction as a modification of $G$ at $H$. We discover modifications with $|G'| - |G| = 2$ and $3$ and combine it with Thomassen's known modification with $|G'| - |G| = 4$.

Let $G$ be a graph with subgraph $H$. We apply a sequence of elementary modifications that removes the subgraph $H$ and replaces it with a different graph $H'$ with the same boundary. The elementary modifications include:
\begin{itemize}
    \item[(1)] the \emph{plus two} modification if $H$ is on the left and $H'$ is on the right here:
    \[
\centering
\begin{tikzpicture}[scale=0.65]

%========================================
% Figure 1, left panel
%========================================
\begin{scope}[xshift=0cm]

  % vertices
  \node[graph vertex] (a) at (0,4.0) {};
  \node[graph vertex] (b) at (1.8,4.0) {};
  \node[graph vertex] (c) at (3.6,4.0) {};

  \node[graph vertex] (d) at (0,2.25) {};
  \node[graph vertex] (e) at (1.8,3.1) {};
  \node[graph vertex] (f) at (3.6,2.05) {};

  \node[graph vertex] (j) at (1.8,1.1) {};

  \node[graph vertex] (g) at (0,0.0) {};
  \node[graph vertex] (h) at (1.8,0.0) {};
  \node[graph vertex] (i) at (3.6,0.0) {};

  % edges
  \draw[graph edge] (a) -- (b);
  \draw[graph edge] (b) -- (c);

  \draw[graph edge] (a) -- (d);
  \draw[graph edge] (c) -- (f);

  \draw[graph edge] (b) -- (e);

  \draw[graph edge] (d) -- (e);
  \draw[graph edge] (e) -- (f);

  \draw[graph edge] (d) -- (j);
  \draw[graph edge] (f) -- (j);

  \draw[graph edge] (d) -- (g);
  \draw[graph edge] (f) -- (i);

  \draw[graph edge] (g) -- (h);
  \draw[graph edge] (h) -- (i);

  \draw[graph edge] (j) -- (h);

  % external stubs
  \draw[external stub] (a) -- ++(-0.7,0);
  \draw[external stub] (a) -- ++(0,0.7);

  \draw[external stub] (c) -- ++(0.7,0);
  \draw[external stub] (c) -- ++(0,0.7);

  \draw[external stub] (g) -- ++(-0.7,0);
  \draw[external stub] (g) -- ++(0,-0.7);

  \draw[external stub] (h) -- ++(0,-0.7);

  \draw[external stub] (i) -- ++(0.7,0);
  \draw[external stub] (i) -- ++(0,-0.7);

\end{scope}

%========================================
% Figure 1, right panel
%========================================
\begin{scope}[xshift=7.2cm]

  % vertices
  \node[graph vertex] (a) at (0.0,4.0) {};
  \node[graph vertex] (b) at (1.9,4.0) {};
  \node[graph vertex] (c) at (3.8,4.0) {};

  \node[graph vertex] (d) at (0.3,2.55) {};
  \node[graph vertex] (e) at (1.9,2.95) {};
  \node[graph vertex] (f) at (3.5,2.65) {};

  \node[graph vertex] (g) at (1.9,1.85) {};
  \node[graph vertex] (h) at (1.1,1.45) {};
  \node[graph vertex] (i) at (2.8,1.40) {};

  \node[graph vertex] (j) at (1.9,0.40) {};
  \node[graph vertex] (k) at (0.0,0.25) {};
  \node[graph vertex] (l) at (4.2,0.35) {};

  % edges
  \draw[graph edge] (a) to[out=0,in=180] (b);
  \draw[graph edge] (b) -- (c);

  \draw[graph edge] (a) -- (d);
  \draw[graph edge] (c) -- (f);

  \draw[graph edge] (b) -- (e);
  \draw[graph edge] (e) -- (g);

  \draw[graph edge] (d) -- (e);
  \draw[graph edge] (e) -- (f);

  \draw[graph edge] (d) -- (h);
  \draw[graph edge] (h) -- (g);
  \draw[graph edge] (g) -- (i);
  \draw[graph edge] (i) -- (f);

  \draw[graph edge] (h) -- (j);
  \draw[graph edge] (i) -- (j);

  \draw[graph edge] (d) -- (k);
  \draw[graph edge] (f) -- (l);

  \draw[graph edge] (k) -- (j);
  \draw[graph edge] (j) -- (l);

  % external stubs
  \draw[external stub] (a) -- ++(-0.7,0);
  \draw[external stub] (a) -- ++(0,0.7);

  \draw[external stub] (c) -- ++(0.7,0);
  \draw[external stub] (c) -- ++(0,0.7);

  \draw[external stub] (k) -- ++(-0.7,0);
  \draw[external stub] (k) -- ++(0,-0.7);

  \draw[external stub] (l) -- ++(0.7,0);
  \draw[external stub] (l) -- ++(0,-0.7);

\end{scope}

\end{tikzpicture}
    \]
(2) the \emph{plus three} modification if $H$ is on the left and $H'$ is on the right here:
\[
\centering
\begin{tikzpicture}[scale=0.75]

%========================================
% Figure 2, left panel
%========================================
\begin{scope}[xshift=0cm]

  % boundary vertices, clockwise
  \node[graph vertex] (v1)  at (1.8,4.25) {};  % top
  \node[graph vertex] (v2)  at (3.2,4.20) {};
  \node[graph vertex] (v3)  at (4.6,3.75) {};
  \node[graph vertex] (v4)  at (5.25,2.75) {};
  \node[graph vertex] (v5)  at (5.00,1.45) {};
  \node[graph vertex] (v6)  at (4.05,0.45) {};
  \node[graph vertex] (v7)  at (2.60,0.20) {};
  \node[graph vertex] (v8)  at (1.05,0.65) {};
  \node[graph vertex] (v9)  at (0.15,1.70) {};
  \node[graph vertex] (v10) at (0.05,2.95) {};
  \node[graph vertex] (v11) at (0.75,3.85) {};

  % boundary edges
  \draw[graph edge] (v1) -- (v2);
  \draw[graph edge] (v2) -- (v3);
  \draw[graph edge] (v3) -- (v4);
  \draw[graph edge] (v4) -- (v5);
  \draw[graph edge] (v5) -- (v6);
  \draw[graph edge] (v6) -- (v7);
  \draw[graph edge] (v7) -- (v8);
  \draw[graph edge] (v8) -- (v9);
  \draw[graph edge] (v9) -- (v10);
  \draw[graph edge] (v10) -- (v11);
  \draw[graph edge] (v11) -- (v1);

  % chords
  \draw[graph edge] (v1) -- (v8);
  \draw[graph edge] (v2) -- (v6);

  % external stubs
  \draw[external stub] (v1) -- ++(0,0.7);
  \draw[external stub] (v3) -- ++(0.55,0.45);
  \draw[external stub] (v4) -- ++(0.7,0);
  \draw[external stub] (v5) -- ++(0.55,-0.45);
  \draw[external stub] (v7) -- ++(0,-0.7);
  \draw[external stub] (v9) -- ++(-0.55,-0.45);
  \draw[external stub] (v10) -- ++(-0.7,0);
  \draw[external stub] (v11) -- ++(-0.55,0.45);

\end{scope}

%========================================
% Figure 2, right panel
%========================================
\begin{scope}[xshift=8.0cm]

  % same boundary vertices
  \node[graph vertex] (v1)  at (1.8,4.25) {};
  \node[graph vertex] (v2)  at (3.2,4.20) {};
  \node[graph vertex] (v3)  at (4.6,3.75) {};
  \node[graph vertex] (v4)  at (5.25,2.75) {};
  \node[graph vertex] (v5)  at (5.00,1.45) {};
  \node[graph vertex] (v6)  at (4.05,0.45) {};
  \node[graph vertex] (v7)  at (2.60,0.20) {};
  \node[graph vertex] (v8)  at (1.05,0.65) {};
  \node[graph vertex] (v9)  at (0.15,1.70) {};
  \node[graph vertex] (v10) at (0.05,2.95) {};
  \node[graph vertex] (v11) at (0.75,3.85) {};

  % boundary edges
  \draw[graph edge] (v1) -- (v2);
  \draw[graph edge] (v2) -- (v3);
  \draw[graph edge] (v3) -- (v4);
  \draw[graph edge] (v4) -- (v5);
  \draw[graph edge] (v5) -- (v6);
  \draw[graph edge] (v6) -- (v7);
  \draw[graph edge] (v7) -- (v8);
  \draw[graph edge] (v8) -- (v9);
  \draw[graph edge] (v9) -- (v10);
  \draw[graph edge] (v10) -- (v11);
  \draw[graph edge] (v11) -- (v1);

  % interior vertices
  \node[graph vertex] (x) at (1.95,2.95) {};
  \node[graph vertex] (y) at (3.00,2.65) {};
  \node[graph vertex] (z) at (4.00,2.35) {};

  % interior edges
  \draw[graph edge] (v10) -- (x);
  \draw[graph edge] (x) -- (y);
  \draw[graph edge] (y) -- (z);

  \draw[graph edge] (x) -- (v2);
  \draw[graph edge] (y) -- (v8);
  \draw[graph edge] (z) -- (v3);
  \draw[graph edge] (z) -- (v6);

  % external stubs
  \draw[external stub] (v1) -- ++(0,0.7);
  \draw[external stub] (v3) -- ++(0.55,0.45);
  \draw[external stub] (v4) -- ++(0.7,0);
  \draw[external stub] (v5) -- ++(0.55,-0.45);
  \draw[external stub] (v7) -- ++(0,-0.7);
  \draw[external stub] (v9) -- ++(-0.55,-0.45);
  \draw[external stub] (v10) -- ++(-0.7,0);
  \draw[external stub] (v11) -- ++(-0.55,0.45);

\end{scope}

\end{tikzpicture}
    \]
(3) Thomassen's \emph{plus four} modification if $H$ is on the left and $H'$ is on the right here:
    \[
\centering
\begin{tikzpicture}[
  graph edge/.style={
    draw=black,
    line width=1.2pt,
    line cap=round
  },
  graph vertex/.style={
    circle,
    fill={rgb,255:red,235;green,70;blue,50},
    draw=none,
    minimum size=2.5mm,
    inner sep=0pt
  },
  external stub/.style={
    draw={rgb,255:red,35;green,110;blue,230},
    line width=1.4pt,
    line cap=round
  }
][scale = 0.9]

% Panel separation (in cm).
\def\panelshift{5.2}

% Identical outer coordinates in both panels.
\foreach \shift/\p in {0/L,\panelshift/R}{
  \begin{scope}[xshift=\shift cm]
    \coordinate (\p a) at (0,2);
    \coordinate (\p b) at (2.4,2);
    \coordinate (\p c) at (2.4,0);
    \coordinate (\p d) at (0,0);

    % Blue external stubs.
    \draw[external stub] (\p a) -- ++(-0.35,0.35);
    \draw[external stub] (\p b) -- ++(0.35,0.35);
    \draw[external stub] (\p c) -- ++(0.35,-0.35);
    \draw[external stub] (\p d) -- ++(-0.35,-0.35);
  \end{scope}
}

% Left panel: original four-cycle.
\draw[graph edge]
  (La)--(Lb)--(Lc)--(Ld)--cycle;

% Right panel: replacement.
\draw[graph edge] (Ra)--(Rd);
\draw[graph edge] (Rb)--(Rc);

% Inner vertices of the replacement.
\begin{scope}[xshift=\panelshift cm]
  \coordinate (ap) at (0.77,1.35);
  \coordinate (bp) at (1.63,1.35);
  \coordinate (cp) at (1.63,0.65);
  \coordinate (dp) at (0.77,0.65);

  % Inner four-cycle.
  \draw[graph edge]
    (ap)--(bp)--(cp)--(dp)--cycle;

  % Connections to the outer vertices.
  \draw[graph edge]
    (Ra)--(ap)
    (Rb)--(bp)
    (Rc)--(cp)
    (Rd)--(dp);
\end{scope}

% Draw all vertices last, above the edges.
\foreach \p in {L,R}{
  \foreach \v in {a,b,c,d}{
    \node[graph vertex] at (\p\v) {};
  }
}

\foreach \v in {ap,bp,cp,dp}{
  \node[graph vertex] at (\v) {};
}

\end{tikzpicture}
\]
\end{itemize}

Beginning with Wiener's almost hypohamiltonian graph of order 31 from \cite{W18}, the plus three modification can be used to reconstruct Tsai's graphs of order 34 and 37 in \cite{Tsai25} and two new graphs of order 37. From the 34 and 37 our plus two modification gives graphs of order 36 and 39. And Thomassen's plus four modification in \cite{T76} (or our variant below) gives graphs of order 38. 

As an application, we prove the following corollaries:

\begin{cor}
   There exists an almost hypohamiltonian graph of order $n$ for all $n \geq 33$. 
\end{cor}

\begin{cor}
    There exists a hypotraceable graph of order $n$ for all $n \geq 132$. 
\end{cor}

The main theory goes into proving that if $G$ is hypohamiltonian, then $G'$ is hypohamiltonian in the plus two modification and a variant with one fewer edge. For this, we consider all possible intersections that a hypothetical Hamiltonian cycle in $G$ (or $G-v$) could have with $H$ (or $H-v$). We call these intersections cycle fragments and show that they have restricted endpoints if $H$ is planar and bipartite. Then we demonstrate that $H'$ has fragments with the same endpoints. Therefore one such fragment can necessarily be extended to all of $G'$, since $G - H = G - H'$ and some fragment in $H$ extends to all of $G$. The arguments are elementary and self-contained.\\
\\
For the plus three modification, $G-v$ Hamiltonian only implies $G'-v$ Hamiltonian for $v$ not in $H'$, in general. For $v \in H'$, we prove that $G'-v$ is Hamiltonian by appealing to specific information about $G$, namely the existence of certain paths in the exterior $G-H$.  \\
\\
From a theoretical standpoint, we link the Hamiltonian condition for (I) two vertex removals within the same graph and (II) for two vertex removals in different graphs. %Hence most of our proofs that a given graph is hypohamiltonian do not require exhibiting a single cycle 
 
% We hope that the theory here can aid the search for additional hypohamiltonian graphs as well as the proof of their non-existence for low orders. 

\begin{rem}
    The majority of our results do not require a computer. The exceptions include a plantri search in Proposition \ref{prop:Tsai_min}, and verifications of the hypohamiltonian (resp. almost hypohamiltonian) property in Corollary \ref{cor:order37} (resp. Corollary \ref{cor:almost}).  
    %building on (I) an almost hypohamiltonian graph of order 34, and (II) a hypohamiltonian graph of order 37.
    %
    %and (III) hypohamiltonian graphs of order 39 built from the one of order 37, our hypohamiltonian results do not require a computer. It suffices to exhibit four Hamiltonian paths, rather than thousands of Hamiltonian cycles.  
\end{rem}

Finally, we give evidence towards the following conjecture:

\begin{conj}
    The smallest planar hypohamiltonian graph has order 34. 
\end{conj}

The conjecture is based on a series of simplifying heuristics, previous results including in \cite{JMOPZ17}, and an analysis of planar graphs consisting solely 4-gons and 5-gons. In particular, we explain why such a graph likely contains no 3-gons and a configuration of 4-gons. Then we prove the following:

\begin{prop}
    The 4-gon configuration in Tsai's 34 vertex examples cannot appear in a smaller order planar hypohamiltonian graph. 
\end{prop}

\subsection*{Acknowledgements}
Daniel Kaplan thanks Tom Kelly for numerous discussions about graph theory and for suggesting several improvements to an earlier draft of this manuscript. The House of Graphs was integral to this work and we are particularly grateful to Cheng-Chen Tsai for adding many relevant graphs to the database. We thank Gunnar Brinkmann and Brendan D. McKay whose plantri program allowed for exhaustive search techniques.

Finally, we thank California State University Long Beach for ideal working conditions.

\subsection*{AI Disclosure}
The main approach was devised by the authors and the examples were generated by hand. Further, we wrote the paper in a series of slowly improving drafts, without tools beyond spellcheck. 

We prompted ChatGPT 5.5 with questions regarding the existing literature. We used Claude for programming a Hamiltonian cycle finder, to test graphs after they were generated. We asked ChatGPT 5.6 to generate code in plantri and to create TikZ versions of our hand-drawn graphs.   

The majority of the new graphs were discovered by undergraduate students at California State University Long Beach, as part of a summer research project, led by the fourth author.

\section{Local Modifications}

We give three constructions to modify a graph $G$ formed by taking a subgraph $H$ and replacing it with $H'$. The first construction is due to Thomassen, and well-known, but we include it both for completeness and as an illustration of our thought process.  

\subsection{Thomassen's Plus Four Modification}

Let $G$ be a graph with subgraph $H$, shown in Figure \ref{fig:Thomassen}, where the blue arrow represents an edge in $G$ but not $H$. Note that each vertex in $\{ a, b, c, d \}$ has degree 3 in $G$. Let $G'$ be a modification of $G$ formed by replacing $H$ with $H'$.

%\begin{figure} \label{fig:Thomassen}
%    \centering
%    \includegraphics[width=0.5\linewidth]{Fig04-Thomassen.jpg}
%    \caption{The graph $H$ on the left is modified to form $H'$ on the right.}
%    \label{fig:Thomassen}
%\end{figure}

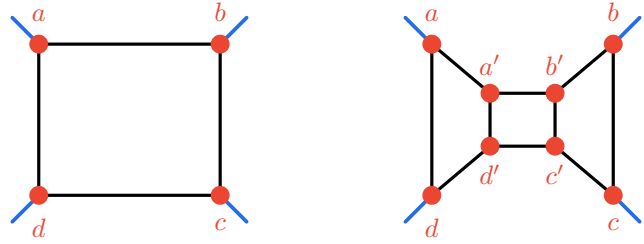
\begin{figure}[ht]
\centering
\begin{tikzpicture}

% Outer vertices
\coordinate (La) at (0,2);
\coordinate (Lb) at (2.4,2);
\coordinate (Lc) at (2.4,0);
\coordinate (Ld) at (0,0);

\coordinate (Ra) at (5.2,2);
\coordinate (Rb) at (7.6,2);
\coordinate (Rc) at (7.6,0);
\coordinate (Rd) at (5.2,0);

% Inner square
\coordinate (ap) at (5.97,1.35);
\coordinate (bp) at (6.83,1.35);
\coordinate (cp) at (6.83,0.65);
\coordinate (dp) at (5.97,0.65);

% External blue stubs
\foreach \v in {La,Ra}{
  \draw[external stub] (\v) -- ++(-0.35,0.35);
}
\foreach \v in {Lb,Rb}{
  \draw[external stub] (\v) -- ++(0.35,0.35);
}
\foreach \v in {Lc,Rc}{
  \draw[external stub] (\v) -- ++(0.35,-0.35);
}
\foreach \v in {Ld,Rd}{
  \draw[external stub] (\v) -- ++(-0.35,-0.35);
}

% Left panel: original four-cycle
\draw[graph edge] (La)--(Lb)--(Lc)--(Ld)--cycle;

% Right panel: replacement
\draw[graph edge] (Ra)--(Rd) (Rb)--(Rc);
\draw[graph edge] (ap)--(bp)--(cp)--(dp)--cycle;
\draw[graph edge]
  (Ra)--(ap) (Rb)--(bp)
  (Rc)--(cp) (Rd)--(dp);

% Draw all vertices last
\foreach \v in {
  La,Lb,Lc,Ld,Ra,Rb,Rc,Rd,ap,bp,cp,dp}{
  \node[graph vertex] at (\v) {};
}

% Outer vertex labels
\foreach \v in {La,Ra}{
  \node[above=5pt,font=\small,text=vertexred]
    at (\v) {$a$};
}
\foreach \v in {Lb,Rb}{
  \node[above=5pt,font=\small,text=vertexred]
    at (\v) {$b$};
}
\foreach \v in {Lc,Rc}{
  \node[below=5pt,font=\small,text=vertexred]
    at (\v) {$c$};
}
\foreach \v in {Ld,Rd}{
  \node[below=5pt,font=\small,text=vertexred]
    at (\v) {$d$};
}

% Inner labels
\node[above=3pt,font=\small,text=vertexred]
  at (ap) {$a'$};
\node[above=3pt,font=\small,text=vertexred]
  at (bp) {$b'$};
\node[below=3pt,font=\small,text=vertexred]
  at (cp) {$c'$};
\node[below=3pt,font=\small,text=vertexred]
  at (dp) {$d'$};

\end{tikzpicture}
   \caption{The graph $H$ on the left is modified to form $H'$ on the right.}
    \label{fig:Thomassen}
\end{figure}

Thomassen realized that $G$ and $G'$ are either both Hamiltonian or neither is. Up to symmetry, there are two ways that a Hamiltonian cycle can traverse $H$: $abcd$ or $ab$ and then later $cd$. Both ways extend to a traversing of $H'$: simply travel from $a$ to $b$ using $aa'b'b$ and from $c$ to $d$ using $cc'd'd$.

Conversely, any Hamiltonian cycle of $G'$ can traverse $H'$, up to symmetry, using the path $aa'b'bcc'd'd$ or the union of two paths $aa'd'd \cup bb'c'c$ or $ad \cup bb'a'd'c'c$. In each case, this gives rise to a unique path of $H$ by omitting the vertices in $H'$ but not in $H$. 

In fact, more is true. $G$ is hypohamiltonian if and only if $G'$ is hypohamiltonian. To see this, notice that:
\begin{itemize}
    \item If $G$ is not Hamiltonian then neither is $G'$, and vice versa. 
    \item If $G - v$ is Hamiltonian for $v \not \in H$ then so is $G' - v$, and vice versa. 
    \item If $G - a$ is Hamiltonian then the cycle must contain $bcd$. So the corresponding path $bcc'b'a'd'd$ can be completed to a Hamiltonian cycle in $G'$. 
    \item If $G' - a'$ is Hamiltonian then the cycle must contain $add'c'b'bc$, in which case $adcb$ can be completed to a Hamiltonian cycle in $G$. 
\end{itemize}

From the perspective of the Hamiltonian cycle, which of $H$ or $H'$ connects to the four external vertices is irrelevant. 

We apply a version of this local modification argument below, where the enumeration of paths becomes more complicated. To remedy this, we only consider the possible endpoints of paths, which is far more manageable. 

This approach is already fruitful here, where one can see that, e.g., adding the edge $ab$ in $H'$ increases the possible paths through $H'$, but not the enumeration of their endpoints. One can similarly, add the edge $cd$ in $H'$ or both edges simultaneously. In this way, Thomassen's construction produces four hypohamiltonian graphs of order $n+4$, given one of order $n$. 

%Similar arguments imply that one could replace $H$ with any of the graphs in Figure \ref{fig:square_modified}

\subsection{The Plus Two Modification}
Next let $H$, $H_1$, $H_2$, and $H_3$ be the local configurations in Figure \ref{fig:local_config}.

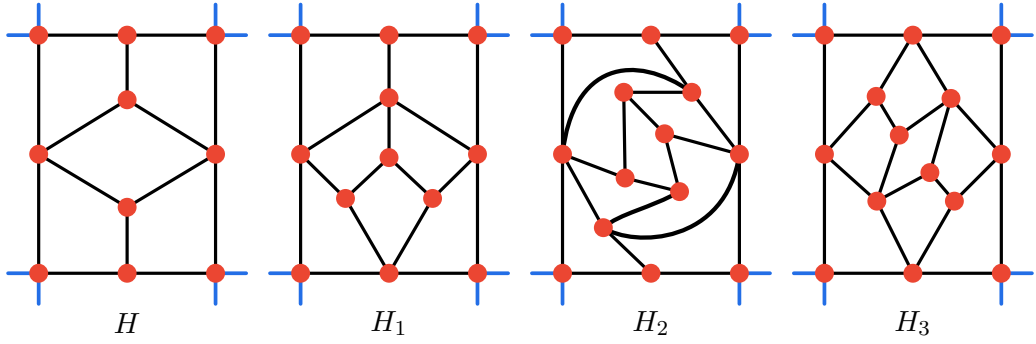
\begin{figure}[ht]
\centering
\begin{tikzpicture}[scale=.9]

% Identical outer boundary for all panels
\def\outerframe{%
  \coordinate (TL) at (0,3.5);
  \coordinate (TM) at (1.3,3.5);
  \coordinate (TR) at (2.6,3.5);
  \coordinate (RM) at (2.6,1.75);
  \coordinate (BR) at (2.6,0);
  \coordinate (BM) at (1.3,0);
  \coordinate (BL) at (0,0);
  \coordinate (LM) at (0,1.75);

  \draw[graph edge]
    (TL)--(TM)--(TR)--(RM)--(BR)--(BM)--(BL)--(LM)--cycle;

  \draw[external stub]
    (TL)--++(-0.45,0) (TL)--++(0,0.45);
  \draw[external stub]
    (TR)--++(0.45,0) (TR)--++(0,0.45);
  \draw[external stub]
    (BL)--++(-0.45,0) (BL)--++(0,-0.45);
  \draw[external stub]
    (BR)--++(0.45,0) (BR)--++(0,-0.45);
}

\def\outervertices{%
  \foreach \v in {TL,TM,TR,RM,BR,BM,BL,LM}{
    \node[graph vertex] at (\v) {};
  }
}

% ======================================
% H
% ======================================
\begin{scope}[xshift=0cm,name prefix=fiveH-]
  \outerframe

  \coordinate (u) at (1.30,2.55);
  \coordinate (v) at (1.30,0.97);

  \draw[graph edge]
    (TM)--(u) (LM)--(u) (RM)--(u);
  \draw[graph edge]
    (LM)--(v) (RM)--(v) (BM)--(v);

  \outervertices
  \foreach \v in {u,v}{
    \node[graph vertex] at (\v) {};
  }
  \node at (1.30,-0.73) {$H$};
\end{scope}

% ======================================
% H_1
% ======================================
\begin{scope}[xshift=3.85cm,name prefix=fiveOne-]
  \outerframe

  \coordinate (u) at (1.30,2.58);
  \coordinate (v) at (1.30,1.70);
  \coordinate (l) at (0.66,1.10);
  \coordinate (r) at (1.94,1.10);

  \draw[graph edge]
    (TM)--(u) (LM)--(u) (RM)--(u) (u)--(v);
  \draw[graph edge]
    (LM)--(l) (l)--(v) (v)--(r) (r)--(RM);
  \draw[graph edge]
    (l)--(BM) (r)--(BM);

  \outervertices
  \foreach \v in {u,v,l,r}{
    \node[graph vertex] at (\v) {};
  }
  \node at (1.30,-0.73) {$H_1$};
\end{scope}

% ======================================
% H_2
% ======================================
\begin{scope}[xshift=7.70cm,name prefix=fiveTwo-]
  \outerframe

  \coordinate (a) at (0.90,2.66);
  \coordinate (b) at (1.90,2.66);
  \coordinate (c) at (1.50,2.05);
  \coordinate (d) at (0.92,1.40);
  \coordinate (e) at (1.72,1.20);
  \coordinate (f) at (0.60,0.67);

  % Ordinary edges
  \draw[graph edge]
    (TM)--(b) (b)--(RM)
    (a)--(b) (a)--(c) (a)--(d)
    (c)--(RM) (c)--(e) (d)--(e)
    (LM)--(d) (LM)--(f) (f)--(BM);

  % Three distinguished curved edges
  \draw[graph edge,line width=1.6pt]
    (LM) .. controls (0.15,3.08)
                 and (1.05,3.26) .. (b);

  \draw[graph edge,line width=1.6pt]
    (f) .. controls (0.84,0.88)
                and (1.30,0.96) .. (e);

  \draw[graph edge,line width=1.6pt]
    (f) .. controls (1.05,0.38)
                and (2.40,0.42) .. (RM);

  \outervertices
  \foreach \v in {a,b,c,d,e,f}{
    \node[graph vertex] at (\v) {};
  }
  \node at (1.30,-0.73) {$H_2$};
\end{scope}

% ======================================
% H_3
% ======================================
\begin{scope}[xshift=11.55cm,name prefix=fiveThree-]
  \outerframe

  \coordinate (a) at (0.76,2.60);
  \coordinate (b) at (1.86,2.57);
  \coordinate (c) at (1.10,2.03);
  \coordinate (d) at (1.55,1.48);
  \coordinate (e) at (0.77,1.06);
  \coordinate (f) at (1.91,1.06);

  \draw[graph edge]
    (TM)--(a) (TM)--(b)
    (LM)--(a) (a)--(c) (c)--(b)
    (b)--(RM) (b)--(d)
    (c)--(e) (e)--(d) (d)--(f)
    (LM)--(e) (e)--(BM)
    (RM)--(f) (f)--(BM);

  \outervertices
  \foreach \v in {a,b,c,d,e,f}{
    \node[graph vertex] at (\v) {};
  }
  \node at (1.30,-0.73) {$H_3$};
\end{scope}

\end{tikzpicture}
  \caption{Four local configurations with $H$ consisting of five 4-gons, $H_1$ consisting of seven 4-gons, and $H_2$ and $H_3$ consisting of nine 4-gons.}
    \label{fig:local_config}
\end{figure}

%\begin{figure}
%    \centering
%    \includegraphics[width=0.75\linewidth]{Fig05-LocConfig2.jpg}
%    \caption{Four local configurations with $H$ consisting of five 4-gons, $H_1$ consisting of seven 4-gons, and $H_2$ and $H_3$ consisting of nine 4-gons.}
%    \label{fig:local_config}
%\end{figure}

Let $G$ be any graph containing $H$, with the prescribed blue edges. Let $G_i$ be the local modification formed by replacing $H$ with $H_i$ for $i = 1, 2, 3$.

\begin{rem} \label{rem:loc_config}
    The phrase local configuration is not synonymous with subgraph. Rather, $G$ contains $H$ as a subgraph \emph{and} the $G$ has the prescribed blue edges. Moreover, $G$ does not contain any additional edges to vertices in $H$. 
\end{rem}

Notice that $H$ and each $H_i$ are bipartite with external vertices colored green and two more green vertices than purple vertices, see Figure \ref{fig:bipartite} 

% Figure 06 
\begin{figure}[h]
\centering
\resizebox{0.7\linewidth}{!}{%
\begin{tikzpicture}
  % Panel 1 (H), row 1, col 1
  \begin{scope}[shift={(0.000,0.000)}]
    \GCoordsH
    \GEdgesH 
    \GNodesH
    \GBipartiteH
  \end{scope}
  % Panel 2 (H1), row 1, col 2
  \begin{scope}[shift={(4.300,0.000)}]
    \GCoordsHone
    \GEdgesHone
    \GNodesHone
    \GBipartiteHone
  \end{scope}
  % Panel 3 (H2), row 1, col 3
  \begin{scope}[shift={(8.600,0.000)}]
    \GCoordsHtwo
    \GEdgesHtwo
    \GNodesHtwo
    \GBipartiteHtwo
  \end{scope}
  % Panel 4 (H3), row 1, col 4
  \begin{scope}[shift={(12.900,0.000)}]
    \GCoordsHthree
    \GEdgesHthree
    \GNodesHthree
    \GBipartiteHthree
  \end{scope}
\end{tikzpicture}%
}
\caption{Each configuration is bipartite, with the external vertices colored green.}
\label{fig:bipartite}
\end{figure}
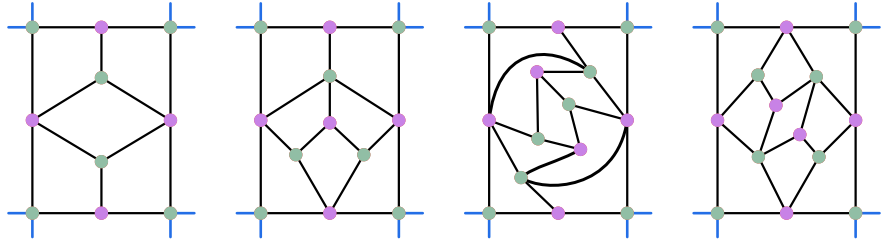

%\begin{figure}
%    \centering
%    \includegraphics[width=0.75\linewidth]{Fig06-Bipartite3.jpg}
%    \caption{Each configuration is bipartite, with the external vertices colored green.}
%    \label{fig:bipartite}
%\end{figure}

\begin{prop} \label{prop:Ham}
    For each $i$, $G$ is Hamiltonian if and only if $G_i$ is Hamiltonian. 
\end{prop}

The idea of the proof is to dissect the Hamiltonian cycle for $G$ into a piece within $H$ and a piece outside of $H$. Then modify the piece inside $H$ to work in $H_i$.

\begin{defn}
    Let $C$ denote a Hamiltonian cycle in the graph $G$. Let $H \subset G$ be a subgraph. Then we call the intersection $H \cap C$ a \emph{cycle fragment}. 
    
    The cycle fragment is necessarily a union of paths $\gamma_i$ and we refer to its collection of endpoints as an \emph{endpoint decomposition}. For example, a fragment consisting of a path starting at 1 visiting 2 and ending at 3 and a path only visiting 4 has the endpoint decomposition $1-3 \ 4$. Note that this data is slightly more than the set of endpoints since we regard $1-2 \ \ 3 \ \ 4$ as different from $1 \ \ 2 \ \ 3-4$.
\end{defn}

%Heuristically, the key component of each graph is the existence of ``nearly independent" cycles around each connecting vertex. Since the local configurations have a small number of vertices, rather than make this heuristic precise, we give a brute force argument. 

\begin{proof}
Assume that $G$ has a Hamiltonian cycle, $C$. Notice that $H$ is bipartite with two more green vertices than purple, and every path through $H$ has one more green than purple vertex. We conclude that the cycle fragment $C \cap H$ is a union of exactly two paths. In the appendix, Figure \ref{fig:H-Ham} depicts 14 possible hypothetical subgraphs $C \cap H$.
%\begin{figure} 
%    \centering
%    \includegraphics[width=0.9\linewidth]{Fig07-H-Ham.jpg}
%    \caption{A Hamiltonian cycle in $G$ intersects $H$ along one of the following orange paths.} 
%    \label{fig:H-Ham}
%\end{figure}

Each possibility for $C \cap H$ determines an endpoint decomposition of the four vertices $\{ 1, 2, 3, 4 \}.$ There are 14 possible endpoint decompositions:
\[
1-2 \ \  4, \quad 1-3 \  \ 4, \quad 2-3 \ \ 4, \quad 1-2 \  \ 3, \quad 1-4 \  \ 3, \quad 2-4 \ \ 3, 
\]
\[
1-3 \  \ 2, \quad 1-4 \  \ 2, \quad 3-4 \  \ 2, \quad 2-3 \  \ 1, \quad 2-4 \  \ 1, 3-4 \  \ 1,
\]
\[
 1-2 \  \ 3-4, \quad 1-4 \ \ 2-3.
\]
Since $H$ is planar the endpoint configuration $1-3 \ \ 2-4$ is not possible. 

Each endpoint decomposition is achievable by some cycle fragment in Figure \ref{fig:H-Ham}. 

By Figures \ref{fig:H1-Ham}, \ref{fig:H2-Ham}, and \ref{fig:H3-Ham} in the appendix, each cycle fragment in $H$ corresponds to one in $H_i$ with the same endpoint decomposition, for $i = 1, 2, 3$. Note that we use the horizontal symmetry of $H_1$, the $\pi$ radians rotational symmetry of $H_2$, and the diagonal reflectional symmetry to achieve all endpoint decompositions in $H_3$.

%\begin{figure}
%    \centering
%    \includegraphics[width=0.6\linewidth]{Fig08-H1-Ham.jpg}
%    \caption{For each fragment in $H$, there is one in $H_1$ with the same endpoints} 
%    \label{fig:H1-Ham}
%\end{figure}

Therefore, the cycle fragment corresponding to $C \cap H$ extends to a Hamiltonian cycle in $G_i$ using the complement $C - H$. We conclude that $G_i$ is Hamiltonian. 
 
Conversely, suppose that $G_i$ has a Hamiltonian cycle $C_i$. The intersection $C_i \cap H_i$ gives some endpoint decomposition. Since $H_i$ is planar and bipartite with two more green vertices than purple, $C_i \cap H_i$ is a cycle fragment consisting of a union of two paths. Hence each $H_i$ has at most the 14 endpoint decompositions. Figure \ref{fig:H-Ham} proves that $H$ has a cycle fragment with the same endpoints. And that fragment can be extended to $G$ using the rest of the cycle, $C_i - H_i$.  

%Lastly, none of the fragments considered in $H_2$ utilize either of the curved edges. Therefore, each one gives fragment in $H_3$. While the additional edges in $H_3$ could in theory create a Hamiltonian cycle in $G_3$ but not $G_2$, the bipartite condition on $H_3$ ensures that these edges do not yield additional endpoint decompositions. 
\end{proof}

%Rather than compute each possible intersection $C \cap H_i$, we simply observe that each planar-permitting endpoint combination is possible. For example, for there is a path going from vertex 1 to 3 to 4 and then separately passing through 2. But there is no path from vertex 1 to 3 and then from 2 to 4 since that would violate planarity. 

%For each endpoint combination (e.g., 1-2-3 4, 1-2-4 3, 1-3-2 4), one can use the symmetry in $H$ to find a corresponding path in Figure \ref{fig:pathsH}. We conclude that some fragment in $H$ extends to a Hamiltonian cycle in $G$. 

%\begin{defn}
%Let $G$ be a graph with $W \subset V$ a subset of vertices. We say $G$ is \emph{hypohamiltonian except at $W$} if $G$ is not Hamiltonian and $G - v$ is Hamiltonian for all $v \not \in W$. 
%\end{defn}

%One should think of this as except \emph{possibly} at $W$ as the statement makes no claim about the Hamiltonian property for $G - w$ for $w \in W$.  

%\begin{thm} \label{thm:hypoham}
%   If $G$ hypohamiltonian then:
%   \begin{itemize}
%       \item[(1)] $G_1$ is hypohamiltonian except at $\{ u_1, v_1, w_1 \}$.
%       \item[(2)] $G_2$ is hypohamiltonian except at $\{ u_2, v_2 \}$.
%       \item[(3)] $G_3$ is hypohamiltonian except at $\{ u_3, v_3, w_3 \}$.
%   \end{itemize}
%   Moreover, the potential failure of $G_i - z$ to be Hamiltonian for $z \in \{ u_1, v_1, w_1, u_3, v_3, w_3\}$ are related, depending on the existence of a single spanning subgraph of the complement of $H_i$. 
%\end{thm}

\begin{thm} \label{thm:hypoham}
   If $G$ hypohamiltonian then $G_i$ is hypohamiltonian for all $i$. 
\end{thm}

%%Note that the proof below is likely more tedious than necessary. Rather than an ad hoc mapping between cycle fragments, a holistic relationship should possible. 

%be one should be able to In summary, each Hamiltonian cycle in $G - v$ can again be decomposed into a cycle fragment within $H$ and a complementary piece outside of $H$. Each fragment in $H$ corresponds to one in $H_i$ with the same endpoint decomposition. Unfortunately, our proof relies on an ad-hoc, element-wise bijection between cycle fragments, rather than a holistic formula for the bijection.

\begin{proof}
Assume that $G$ is hypohamiltonian. Then, by definition, $G$ is not Hamiltonian. By Proposition \ref{prop:Ham} neither is $G_i$.

Additionally, since $G - v$ is Hamiltonian for $v \in G - H$, Proposition \ref{prop:Ham} applied to the graphs $G -v$\footnote{If $v$ is adjacent to $H$, then not all endpoint decompositions remain viable after removing $v$. However, the proof above continues to hold because when an endpoint decomposition does not appear among fragments of $H$, the corresponding endpoint decomposition in $H_i$ is similarly excluded.} and $G_i-v$ imply that $G_i - v$ is Hamiltonian. So it remains to consider the removal of vertices in $H_i$, where we build Hamiltonian cycles using the removal of corresponding vertices in $H$.  

We have already seen that if $v$ is one of the four endpoints, i.e., $v \in \{ 1, 2, 3, 4 \}$ then all fragments in $H-v$ and $H_i - v$ are Hamiltonian paths, with all endpoint decompositions possible. 

So it remains to remove an internal vertex $v \in H$ and we separately consider the cases $v$ is green versus purple. 

If $v$ is green, then the resulting fragment has a Hamiltonian path, and the options for $H$, $H_1$, $H_2$, and $H_3$ are shown in Figure \ref{fig:hypogreen}, in the appendix.

The key idea is that each cycle fragment in $H_1 - v$ corresponds to a cycle fragment in $H-w$ for some $w \in H$. Since $G-w$ is Hamiltonian, one such fragment can be completed to a Hamiltonian cycle $C$ in $G$. And the exterior $C - H$ can be glued to the corresponding fragment in $H_1 - v$, to give a Hamiltonian cycle, $C_1$.

Similarly, if we remove a purple vertex $v$ from $H_1$, $H_2$, or $H_3$ then the resulting fragments again cover the options in some $H - w$, as shown in Figure \ref{fig:hypopurple}, in the appendix.

%Let $w$ be a vertex in $H$. By assumption, $G - w$ has a Hamiltonian cycle, $C$. By symmetry $C \cap H$ is depicted in Figure \ref{fig:hypoH}.
%
%\begin{figure}
%    \centering
%    \includegraphics[width=0.8\linewidth]{Hypo_PathsH.jpg}
%    \caption{Each possible cycle fragment, for each $H -v$}
%    \label{fig:hypoH}
%\end{figure}
%
%Using horizontal and vertical reflections, we get
%Therefore, its endpoint decomposition is 
%\[
%3-2-4, \quad 2 - 3- 4, \quad 1 - 2 \ \ 3 \ \ 4, \quad 1 - 3 \ \ 2 \ \ 4, 
%\]
%\[
%1 \ \ 2-3 \ \ 4, \quad 1 - 3 \ \ 2 \ \ 4, \quad 2 - 1 - 4 - 3,  \quad 3 - 2 - 1 - 4.
%\]
%
%Since there are two fragments for each $H - v$, we can only assume that one extends to a Hamiltonian cycle in $G - v$.
%It suffices to show that each $H_i - w_i$ has a cycle fragment with both endpoint decompositions in a column of Figure \ref{fig:hypoH}, up to symmetry. 
%
%For $H_1$, we have the cycle fragments
%
%\begin{figure}
%    \centering
%    \includegraphics[width=0.5\linewidth]{PathsH1.jpg}
%    \caption{}
%    \label{fig:PathsH1}
%\end{figure}
%
%
%
%with $w_i \in H_i$ 
\end{proof}

\begin{cor}
    There are at least 66 hypohamiltonian planar graphs of order 38. 
\end{cor}

\begin{proof}
    The curved edges in $H_2$ are not utilized for any fragment in $H_2$ nor $H_2 - v$ for any $v$. Therefore, the hypohamiltonian result holds after removing either or both edges.

    Moreover, one can arbitrarily add edges that preserve the bipartite and planar conditions  and the hypohamiltonian result still holds. Notice that the additional edges cannot lead to a cycle fragment with a new endpoint configuration since $H_2$ has all allowable endpoint configurations possible for a bipartite (with two more green vertices) and planar graph. 

    %\begin{figure}
    %\centering
    %\includegraphics[width=0.8\linewidth]{Fig17-H2-mods.jpg}
    %\caption{Each modification of $H_2$ gives a new local configuration that admits fragments with the same endpoint decompositions.}
    %\label{fig:H2-modification}
    %\end{figure}

    % Complete figure: requires only \usepackage{tikz} and \usepackage{graphicx}.
% No custom commands or styles are needed in the document preamble.
\begin{figure}[tbp]
\centering
\resizebox{.8\linewidth}{!}{%
\begin{tikzpicture}[line cap=round,line join=round]
\definecolor{figRed}{HTML}{ED4138}
\definecolor{figBlue}{HTML}{1670EA}
\definecolor{figGreen}{HTML}{19834D}
\definecolor{figOrange}{HTML}{F6BF65}
\definecolor{figViolet}{HTML}{B023B9}
\definecolor{figLime}{HTML}{55B752}
\tikzset{
 fedge/.style={black,line width=1.05pt},
 fboundary/.style={figBlue,line width=1.15pt},
 fstubs/.style={figBlue,line width=1.15pt},
 fvertex/.style={circle,fill=figRed,inner sep=0pt,minimum size=2.3mm},
 fgreen/.style={fvertex,fill=figGreen},
 fnew/.style={fvertex,fill=figRed},
 fhighlight/.style={figOrange,line width=1.8pt,opacity=1,line cap=round,line join=round},
 fviolet/.style={figViolet,line width=1.8pt,line cap=round},
 flime/.style={figLime,line width=1.8pt,line cap=round},
}
\begin{scope}[shift={(0,0)}]
\coordinate (TL) at (0.000,3.900);
\coordinate (TM) at (1.720,3.910);
\coordinate (TR) at (3.490,3.900);
\coordinate (LM) at (0.020,2.170);
\coordinate (RM) at (3.490,2.170);
\coordinate (BL) at (0.000,0.140);
\coordinate (BM) at (1.720,0.140);
\coordinate (BR) at (3.490,0.140);
\coordinate (a) at (1.720,3.130);
\coordinate (b) at (0.800,2.060);
\coordinate (c) at (1.490,2.410);
\coordinate (d) at (2.750,2.180);
\coordinate (e) at (1.780,1.040);
\coordinate (f) at (2.220,1.730);
\draw[fedge] (TL)--(TM);
\draw[fedge] (TM)--(TR);
\draw[fedge] (TL)--(LM);
\draw[fedge] (LM)--(BL);
\draw[fedge] (TR)--(RM);
\draw[fedge] (RM)--(BR);
\draw[fedge] (BL)--(BM);
\draw[fedge] (BM)--(BR);
\draw[fedge] (TM)--(a);
\draw[fedge] (TM)--(d);
\draw[fedge] (LM)--(b);
\draw[fedge] (b)--(c);
\draw[fedge] (c)--(a);
\draw[fedge] (c)--(e);
\draw[fedge] (b)--(BM);
\draw[fedge] (a)--(f);
\draw[fedge] (f)--(e);
\draw[fedge] (e)--(BM);
\draw[fedge] (f)--(d);
\draw[fedge] (d)--(RM);
\draw[fhighlight] (LM) .. controls (0.4, 2.6) and (1.3, 2.9) .. (a);
\draw[fhighlight] (e) .. controls (2.2, 1.25) and (3.25, 1.7) .. (RM);
\draw[fviolet] (TM) .. controls (1.3, 3.2) and (0.9, 2.53) .. (b);
\draw[fviolet] (BM) .. controls (2.1, 0.8) and (2.5, 1.5) .. (d);
\draw[flime] (TL) -- (c);
\draw[flime] (BR) -- (f);
%.. controls (0.63, 2.9) and (3.1, 0.7) .. (BR);
\draw[fstubs] (TL)--++(-.45,0) (TL)--++(0,.45);
\draw[fstubs] (TR)--++(.45,0) (TR)--++(0,.45);
\draw[fstubs] (BL)--++(-.45,0) (BL)--++(0,-.45);
\draw[fstubs] (BR)--++(.45,0) (BR)--++(0,-.45);
\node[fvertex] at (TL) {};
\node[fvertex] at (TM) {};
\node[fvertex] at (TR) {};
\node[fvertex] at (LM) {};
\node[fvertex] at (RM) {};
\node[fvertex] at (BL) {};
\node[fvertex] at (BM) {};
\node[fvertex] at (BR) {};
\node[fvertex] at (a) {};
\node[fvertex] at (b) {};
\node[fvertex] at (c) {};
\node[fvertex] at (d) {};
\node[fvertex] at (e) {};
\node[fvertex] at (f) {};
\node[font=\footnotesize\bfseries] at (1.75,4.80) {$4\times4$};
\end{scope}
% Exact 90-degree clockwise rotation of the FIRST panel, with its orange edges omitted.
% Colored endpoints are LM-to-lower-interior and upper-interior-to-RM, as marked.
\begin{scope}[shift={(5.0,0)}]
  \begin{scope}[rotate around={-90:(1.745,2.02)}]
  \coordinate (TL) at (0.000,3.900);
  \coordinate (TM) at (1.720,3.910);
  \coordinate (TR) at (3.490,3.900);
  \coordinate (LM) at (0.020,2.170);
  \coordinate (RM) at (3.490,2.170);
  \coordinate (BL) at (0.000,0.140);
  \coordinate (BM) at (1.720,0.140);
  \coordinate (BR) at (3.490,0.140);
  \coordinate (a) at (1.720,3.130);
  \coordinate (b) at (0.800,2.060);
  \coordinate (c) at (1.490,2.410);
  \coordinate (d) at (2.750,2.180);
  \coordinate (e) at (1.780,1.040);
  \coordinate (f) at (2.220,1.730);
  \draw[fedge] (TL)--(TM);
  \draw[fedge] (TM)--(TR);
  \draw[fedge] (TL)--(LM);
  \draw[fedge] (LM)--(BL);
  \draw[fedge] (TR)--(RM);
  \draw[fedge] (RM)--(BR);
  \draw[fedge] (BL)--(BM);
  \draw[fedge] (BM)--(BR);
  \draw[fedge] (TM)--(a);
  \draw[fedge] (TM)--(d);
  \draw[fedge] (LM)--(b);
  \draw[fedge] (b)--(c);
  \draw[fedge] (c)--(a);
  \draw[fedge] (c)--(e);
  \draw[fedge] (b)--(BM);
  \draw[fedge] (a)--(f);
  \draw[fedge] (f)--(e);
  \draw[fedge] (e)--(BM);
  \draw[fedge] (f)--(d);
  \draw[fedge] (d)--(RM);
  \draw[fviolet] (TM) .. controls (1.3, 3.2) and (0.9, 2.53) .. (b);
  \draw[fviolet] (BM) .. controls (2.1, 0.8) and (2.5, 1.5) .. (d);
  % Each original hexagon receives a single diagonal, giving four quadrilaterals.
  % Hexagon TL-TM-a-c-b-LM: chord TL--c.
  \draw[flime] (TL) .. controls (0.46,3.44) and (1.06,2.77) .. (c);
  % Hexagon RM-BR-BM-e-f-d: chord BR--f.
  \draw[flime] (BR) .. controls (3.23,0.62) and (2.51,1.49) .. (f);
  \draw[fstubs] (TL)--++(-.45,0) (TL)--++(0,.45);
  \draw[fstubs] (TR)--++(.45,0) (TR)--++(0,.45);
  \draw[fstubs] (BL)--++(-.45,0) (BL)--++(0,-.45);
  \draw[fstubs] (BR)--++(.45,0) (BR)--++(0,-.45);
  \node[fvertex] at (TL) {};
  \node[fvertex] at (TM) {};
  \node[fvertex] at (TR) {};
  \node[fvertex] at (LM) {};
  \node[fvertex] at (RM) {};
  \node[fvertex] at (BL) {};
  \node[fvertex] at (BM) {};
  \node[fvertex] at (BR) {};
  \node[fvertex] at (a) {};
  \node[fvertex] at (b) {};
  \node[fvertex] at (c) {};
  \node[fvertex] at (d) {};
  \node[fvertex] at (e) {};
  \node[fvertex] at (f) {};
  \end{scope}
  \node[font=\footnotesize\bfseries] at (1.75,4.80) {$3\times3$};
\end{scope}
\begin{scope}[shift={(10.0,0)}]
  \begin{scope}[rotate around={-90:(1.745,2.02)}]
  \coordinate (TL) at (0.000,3.900);
  \coordinate (TM) at (1.720,3.910);
  \coordinate (TR) at (3.490,3.900);
  \coordinate (LM) at (0.020,2.170);
  \coordinate (RM) at (3.490,2.170);
  \coordinate (BL) at (0.000,0.140);
  \coordinate (BM) at (1.720,0.140);
  \coordinate (BR) at (3.490,0.140);
  \coordinate (a) at (1.720,3.130);
  \coordinate (b) at (0.800,2.060);
  \coordinate (c) at (1.490,2.410);
  \coordinate (d) at (2.750,2.180);
  \coordinate (e) at (1.780,1.040);
  \coordinate (f) at (2.220,1.730);
  \draw[fedge] (TL)--(TM);
  \draw[fedge] (TM)--(TR);
  \draw[fedge] (TL)--(LM);
  \draw[fedge] (LM)--(BL);
  \draw[fedge] (TR)--(RM);
  \draw[fedge] (RM)--(BR);
  \draw[fedge] (BL)--(BM);
  \draw[fedge] (BM)--(BR);
  \draw[fedge] (TM)--(a);
  %\draw[fedge] (TM)--(d);
  \draw[fedge] (LM)--(b);
  \draw[fedge] (b)--(c);
  \draw[fedge] (c)--(a);
  \draw[fedge] (c)--(e);
  %\draw[fedge] (b)--(BM);
  \draw[fedge] (LM) -- (e);
  \draw[fedge] (RM) -- (a);
  \draw[fedge] (a)--(f);
  \draw[fedge] (f)--(e);
  \draw[fedge] (e)--(BM);
  \draw[fedge] (f)--(d);
  \draw[fedge] (d)--(RM);
  \draw[fviolet] (TM) .. controls (1.3, 3.2) and (0.9, 2.53) .. (b);
  \draw[fviolet] (BM) .. controls (2.1, 0.8) and (2.5, 1.5) .. (d);
  % Each original hexagon receives a single diagonal, giving four quadrilaterals.
  % Hexagon TL-TM-a-c-b-LM: chord TL--c.
  \draw[flime] (TL) .. controls (0.46,3.44) and (1.06,2.77) .. (c);
  % Hexagon RM-BR-BM-e-f-d: chord BR--f.
  \draw[flime] (BR) .. controls (3.23,0.62) and (2.51,1.49) .. (f);
  \draw[fstubs] (TL)--++(-.45,0) (TL)--++(0,.45);
  \draw[fstubs] (TR)--++(.45,0) (TR)--++(0,.45);
  \draw[fstubs] (BL)--++(-.45,0) (BL)--++(0,-.45);
  \draw[fstubs] (BR)--++(.45,0) (BR)--++(0,-.45);
  \node[fvertex] at (TL) {};
  \node[fvertex] at (TM) {};
  \node[fvertex] at (TR) {};
  \node[fvertex] at (LM) {};
  \node[fvertex] at (RM) {};
  \node[fvertex] at (BL) {};
  \node[fvertex] at (BM) {};
  \node[fvertex] at (BR) {};
  \node[fvertex] at (a) {};
  \node[fvertex] at (b) {};
  \node[fvertex] at (c) {};
  \node[fvertex] at (d) {};
  \node[fvertex] at (e) {};
  \node[fvertex] at (f) {};
  \end{scope}
  \node[font=\footnotesize\bfseries] at (1.75,4.80) {$3\times3$};
\end{scope}
\begin{scope}[shift={(15.0,0)}]
\coordinate (TL) at (0.000,3.900);
\coordinate (TM) at (1.730,3.900);
\coordinate (TR) at (3.510,3.900);
\coordinate (LM) at (0.000,2.110);
\coordinate (RM) at (3.510,2.110);
\coordinate (BL) at (0.000,0.140);
\coordinate (BM) at (1.730,0.140);
\coordinate (BR) at (3.510,0.140);
\coordinate (a) at (0.570,3.120);
\coordinate (b) at (1.220,3.120);
\coordinate (c) at (0.570,2.370);
\coordinate (d) at (1.220,2.370);
\coordinate (e) at (1.450,1.570);
\coordinate (f) at (2.720,2.220);
%\draw[fedge] (TL)--(TM);
\draw[fedge] (TM)--(TR);
\draw[fedge] (TL)--(LM);
\draw[fedge] (LM)--(BL);
\draw[fedge] (TR)--(RM);
\draw[fedge] (RM)--(BR);
\draw[fedge] (BL)--(BM);
\draw[fedge] (BM)--(BR);
\draw[fedge] (TL)--(a);
\draw[fedge] (TM)--(b);
\draw[fedge] (a)--(b);
\draw[fedge] (a)--(c);
\draw[fedge] (b)--(d);
\draw[fedge] (c)--(d);
\draw[fedge] (c)--(LM);
\draw[fedge] (d)--(e);
\draw[fedge] (TM)--(e);
\draw[fedge] (TM)--(f);
\draw[fedge] (f)--(RM);
%\draw[fedge] (f)--(BR);
\draw[fedge] (e)--(BM);
%\draw[fedge] (e)--(LM);
% Distinguished choices exactly as in the approved fourth-panel sketch.
\draw[fviolet] (TL)--(TM);
\draw[fviolet] (LM) .. controls (0.55, 1.99) and (1.15, 1.75) .. (e);
\draw[flime] (BL) .. controls (0.55, 1.0) and (1.05, 1.95) .. (d);
\draw[fhighlight] (c) .. controls (0.95, 1.34) and (1.28, 0.65) .. (BM);
\draw[fedge] (BM) -- (f);
\draw[fstubs] (TL)--++(-.45,0) (TL)--++(0,.45);
\draw[fstubs] (TR)--++(.45,0) (TR)--++(0,.45);
\draw[fstubs] (BL)--++(-.45,0) (BL)--++(0,-.45);
\draw[fstubs] (BR)--++(.45,0) (BR)--++(0,-.45);
\node[fvertex] at (TL) {};
\node[fvertex] at (TM) {};
\node[fvertex] at (TR) {};
\node[fvertex] at (LM) {};
\node[fvertex] at (RM) {};
\node[fvertex] at (BL) {};
\node[fvertex] at (BM) {};
\node[fvertex] at (BR) {};
\node[fvertex] at (a) {};
\node[fvertex] at (b) {};
\node[fvertex] at (c) {};
\node[fvertex] at (d) {};
\node[fvertex] at (e) {};
\node[fvertex] at (f) {};
\node[font=\footnotesize\bfseries] at (1.75,4.80) {$4\times4\times2$};
\end{scope}
\end{tikzpicture}%
}
\caption{Here are 66 planar hypohamiltonian graphs of order 38 obtained by adding the colored edges independently while preserving the planar condition}
    \label{fig:hypo38}
\end{figure}
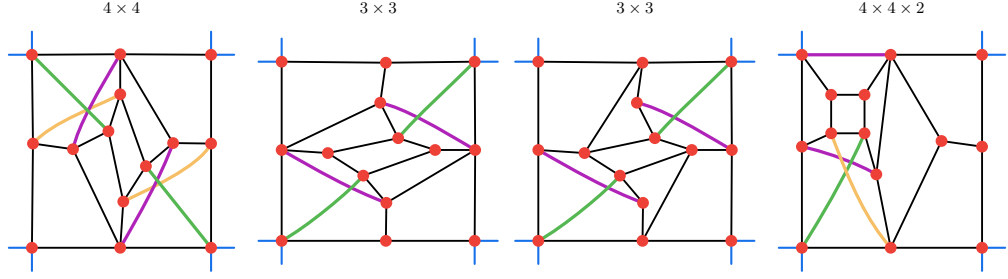

  %  \begin{figure}
  %  \centering
  %  \includegraphics[width=0.8\linewidth]{Fig17-Hypo38.jpg}
  %  \caption{Here are 66 planar hypohamiltonian graphs of order 38 obtained by adding the colored edges independently while preserving the planar condition}
  %  \label{fig:hypo38}
  %  \end{figure}
    
    In total, this gives four choices for an edge (including no edge) at two places in the graph. Additionally, we obtain new graphs by reflecting over the vertical axis and rotating by $\pi/2$ clockwise, see Figure \ref{fig:hypo38}. Subtracting out the symmetric graphs, one gets a total of $34 = 4 \times 4 + 3 \times 3 + 3 \times 3$ hypohamiltonian planar graphs of order 38. 

    Next, by Proposition \ref{prop:Thom_hypo} we can apply Thomassen's plus four construction to any of the four non-central 4-gons. Then one arrives at a hypohamiltonian planar graph with an $n$-gon for $n = 4, 5, 6, 7$. Additionally, one can add three possible edges in the 6-gon for a total of 4 such graphs. And one can independently add an edge in the 7-gon. This gives $4 \times 4 \times 2 = 32$ graphs. None of these graphs are isomorphic to any of the ones above, because they each contain an internal degree three vertex with only edges to boundary vertices.  
\end{proof}

\begin{rem}
    We do not attempt an exhaustive search of hypohamiltonian planar graphs of order 38, and we suspect there are many others. The goal of this example was to illustrate the large number of modifications possible from a single such graph.   
\end{rem}

\begin{defn} \label{defn:plustwo}
We refer to the modification replacing $H$ with $H_1$ as the \emph{plus two modification}, since it increases the number of vertices by two. 
\end{defn}

We view the plus two modification as the main modification in this section, whereas $H_2$ illustrates the large number of possible modifications, and $H_3$ demonstrates that the construction could preserve the existence of a subgraph $H$, allowing for iteration. Note that the modification replacing $H$ with $H_3$ is an alternative plus four modification.

\subsubsection{Imperfect Modifications}

We begin by defining a weakening of the hypohamiltonian condition. 

\begin{defn}
    Let $G = (V, E)$ be a graph and let $W \subset V$ be a subset of vertices.
    We say $G$ is \emph{hypohamiltonian except at $W$} if $G$ is not Hamiltonian and $G-v$ is Hamiltonian for all $v \in V - W$. 
\end{defn}

We remark that this definition makes no claim about the Hamiltonian property of $G - w$ for any $w \in W$. 

In this section, we explore modifications where the cycle fragment analysis is insufficient to conclude that the modification is hypohamiltonian. Nonetheless, such analysis links the Hamiltonian condition on the vertex deleted graphs, so the existence of two Hamiltonian paths is sufficient for the hypohamiltonian property. 

Let $H$ denote the first configuration in Figure \ref{fig:local_config}. Let $G$ be any graph containing $H$. Let $H_1$ denote the second configuration in Figure \ref{fig:local_config} and let $G_1$ be the graph obtained by replacing $H$ with $H_1$. Let $H_1'$ be the graph in Figure \ref{fig:H1prime}, obtained from $H_1$ by removing a single edge. Let $G_1'$ denote the new graph obtained by replacing $H$ with $H_1'$. 

% Figure 18: H_1' — complete paste-ready figure.
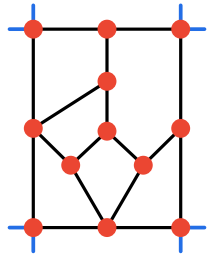
\begin{figure}[!htbp]
\centering
\begin{tikzpicture}[scale=0.75]
  \GCoordsHone
  \draw[graph edge] (TL)--(TM)--(TR)--(RM)--(BR)--(BM)--(BL)--(LM)--cycle;
  \draw[graph edge]
    (TM)--(u) (LM)--(u) (u)--(v) (LM)--(l)
    (l)--(v) (v)--(r) (r)--(RM) (l)--(BM) (r)--(BM);
  \draw[external stub]
    (TL)--++(-0.42,0) (TL)--++(0,0.42)
    (TR)--++(0.42,0) (TR)--++(0,0.42)
    (BL)--++(-0.42,0) (BL)--++(0,-0.42)
    (BR)--++(0.42,0) (BR)--++(0,-0.42);
  \GNodesHone
\end{tikzpicture}
 \caption{The graph $H_1'$ is derived from $H_1$ by removing a single edge.}
    \label{fig:H1prime}
\end{figure}

 %\begin{figure}
 %   \centering
 %   \includegraphics[width=0.3\linewidth]{Fig18-H1prime.jpg}
 %   \caption{The graph $H_1'$ depicted is derived from $H_1$ by removing a single edge.}
 %   \label{fig:H1prime}
%\end{figure}

In Theorem \ref{thm:hypoham}, we prove that $G$ hypohamiltonian implies $G_1$ hypohamiltonian. While the argument does not extend to $G_1'$ in total, pieces of the argument can be salvaged. Namely, we have the following: 

\begin{cor} \label{cor:G1prime_hypoexc}
    If $G$ is hypohamiltonian, then $G_1'$ is hypohamiltonian except at $H_1' - \{ 1, 2, 3, 4 \}$
\end{cor}

\begin{proof}
    We need to show the following three claims:
\begin{itemize}
    \item[(I)] If $G$ is not Hamiltonian, then $G_1$ and hence $G_1'$ is not Hamiltonian.
    \item[(II)] If $G - v$ is Hamiltonian for $v \not \in H$, then $G_1'-v$ is Hamiltonian.
    \item[(III)] If $G - i$ is Hamiltonian for $i \in \{ 1, 2, 3, 4 \}$ an external vertex in $H$ then $G_1' - i$ is Hamiltonian. 
\end{itemize}

Statement (I) is clear since removing an edge gives less flexibility to build a Hamiltonian cycle. Statement (II) follows from observing that the cycle fragments in $H_1$ in Figure \ref{fig:H1-Ham} do not make use of the removed edge. Statement (III) is also a consequence of the cycle fragments in Figure \ref{fig:H1-Ham} not using the removed edge, because they provide all possible fragments in $H_1 - i$, for each $i$. 

\end{proof}

%Above we have shown that $G_1'$ is hypohamiltonian, except at $H_1'$. Further, the four external vertices of $H_1'$, $\{ 1, 2, 3, 4 \}$, each have $G_1' -i$ Hamiltonian. %The green vertices removed in $H_1$ above have cycle fragments not using the removed edge. And $G_1' - v$ is Hamiltonian for both vertices containing the removed edge. This leaves the five highlighted vertices above. 

\begin{cor} \label{cor:paths}
    If $G$ is hypohamiltonian, with Hamiltonian paths from $1$ to $3$ in both $(G - H) \cup \{ 1, 2, 3, 4 \}$ and $(G - H) \cup \{ 1, 3 \}$, then $G_1'$ is hypohamiltonian. 
\end{cor}

\begin{proof}
    By Figure \ref{fig:H1Prime-Hypo}, removing each green vertex gives a fragment with the endpoint decomposition $1-3$.  Therefore, if $(G - H) \cup \{1, 3 \}$ has a Hamiltonian path from $1$ to $3$ then we can extend this fragment to a Hamiltonian cycle.  

Similarly, removing each purple vertex gives a fragment with the endpoint decomposition $1-3 \ 2 \ 4$, by Figure \ref{fig:H1Prime-Hypo}. Hence a Hamiltonian path in $(G - H) \cup \{ 1, 2, 3, 4 \}$ from $1$ to $3$ can be glued to the cycle fragment to give a Hamiltonian cycle. 
\end{proof}

\begin{cor} \label{cor:G1prime_hypo}
    Let $G$ be Tsai's 34 vertex example. Then $G_1'$ is hypohamiltonian. 
\end{cor}

\begin{proof}
    By Corollary \ref{cor:paths}, it suffices to find Hamiltonian paths from $1$ to $3$ in both $(G - H) \cup \{ 1, 2, 3, 4 \}$ and $(G - H) \cup \{ 1, 3 \}$. Indeed, Figure \ref{fig:H1Prime-paths} depicts such paths. 

 \begin{figure}
    \centering
    \includegraphics[width=0.8\linewidth]{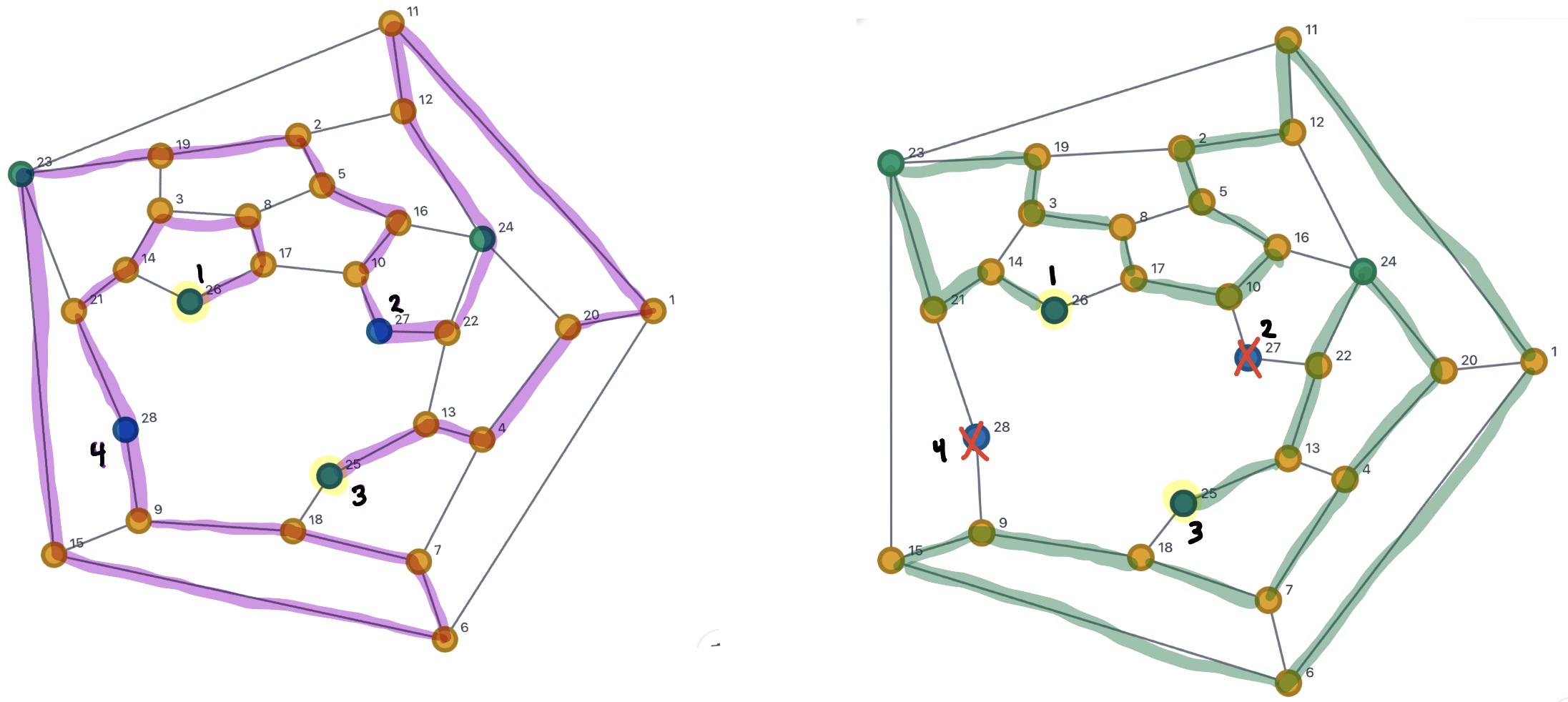}
    \caption{Hamiltonian paths in the complement of $H$ that allow cycle fragments with endpoint decomposition $1-3$ and $1-3 \ 2 \ 4$ to extend to a Hamiltonian cycle.}
    \label{fig:H1Prime-paths}
\end{figure}

\end{proof}

%%This is possibly not a central point. Commented out for now
%\begin{rem}
%    If a Hamiltonian graph has a single 6-gon, two 4-gons, and the remaining faces 5-gons, then Grinberg's theorem (explained below) modulo 3 says that any Hamiltonian cycle must separate the 6-gon from the two 4-gons. Therefore, the only viable cycle fragments in $H$ have endpoint configurations $1-3$ and $2-4$. Since $G_1' - v$ for each green vertex $v \in H_1'$ has cycle fragments with those endpoint configurations, $G_1' -v$ is Hamiltonian. In other words, one need not find the green Hamiltonian paths above. 
%\end{rem}

\begin{cor}  \label{cor:order36}
    There are at least 12 planar hypohamiltonian graphs of order 36.
\end{cor}

\begin{proof}
    
    One can rotate the previous argument $\pi/2$, $\pi$, or $3 \pi/2$ radians to obtain four such graphs. 
    
    By Figure \ref{fig:H1Prime_Path24} the exterior has Hamiltonian paths from $2$ to $4$ in both $(G - H) \cup \{ 1, 2, 3, 4 \}$ and $(G - H) \cup \{ 2, 4 \}$, as desired.
    
    \begin{figure}
    \centering
    \includegraphics[width=0.8\linewidth]{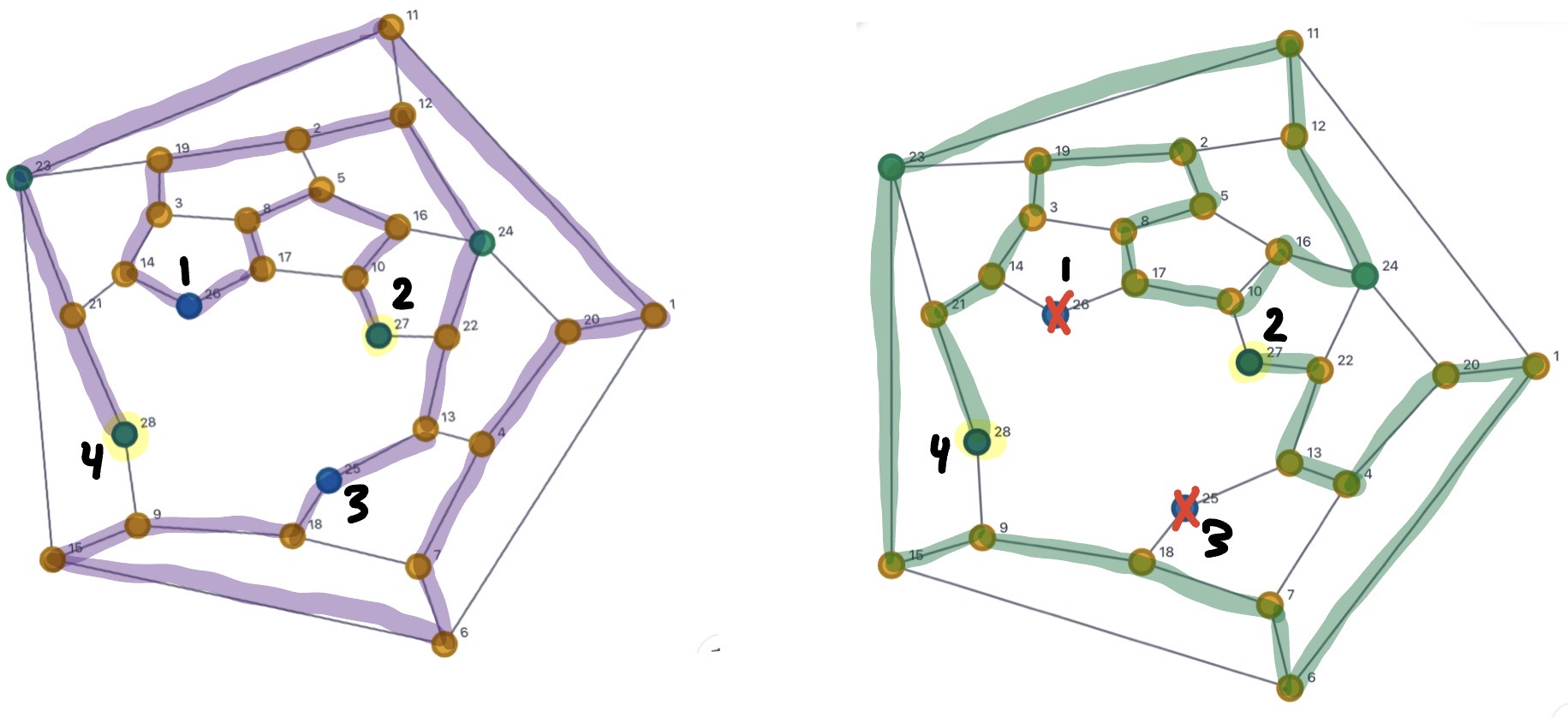}
    \caption{Hamiltonian paths in the complement of $H$ that allow cycle fragments with endpoint decomposition $2-4$ and $2-4 \ 1 \ 3$ to extend to a Hamiltonian cycle.}
    \label{fig:H1Prime_Path24}
    \end{figure}
    
    Additionally, one can add one of three possible edges to the 6-gon to obtain $4 \times 4 = 16$ more graphs. Many such graphs are isomorphic, leaving the following 12 pairwise non-isomorphic graphs.

    % Four new local configurations (transcribed from Fig22-Hypo36_Exam2.jpg).
% Complete TikZ figure: requires only \usepackage{tikz,graphicx}.
\begin{figure}[tbp]
\centering
\resizebox{0.8\linewidth}{!}{%
\begin{tikzpicture}[line cap=round,line join=round]
\definecolor{newRed}{HTML}{ED4138}
\definecolor{newBlue}{HTML}{1670EA}
\definecolor{newGreen}{HTML}{46AE56}
\definecolor{newPurple}{HTML}{9D22B7}
\definecolor{newOrange}{HTML}{F39B21}
\tikzset{
  graph line/.style={black,line width=1.1pt},
  external halfedge/.style={newBlue,line width=1.35pt},
  graph dot/.style={circle,fill=newRed,inner sep=0pt,minimum size=2.3mm},
  candidate green/.style={newGreen,line width=2.25pt},
  candidate purple/.style={newPurple,line width=2.25pt},
  candidate orange/.style={newOrange,line width=2.25pt}
}
  % Panel 1
  \begin{scope}[shift={(0.00,0)}]
    \coordinate (TL) at (0.000,2.850);
    \coordinate (TM) at (1.570,2.850);
    \coordinate (TR) at (3.140,2.850);
    \coordinate (LM) at (0.000,1.500);
    \coordinate (RM) at (3.140,1.500);
    \coordinate (BL) at (0.000,0.000);
    \coordinate (BM) at (1.570,0.000);
    \coordinate (BR) at (3.140,0.000);
    \coordinate (a) at (1.570,2.030);
    \coordinate (b) at (1.570,1.470);
    \coordinate (c) at (0.680,0.670);
    \coordinate (d) at (2.350,0.680);
    \draw[graph line] (TL)--(TM);
    \draw[graph line] (TM)--(TR);
    \draw[graph line] (TR)--(RM);
    \draw[graph line] (RM)--(BR);
    \draw[graph line] (BR)--(BM);
    \draw[graph line] (BM)--(BL);
    \draw[graph line] (BL)--(LM);
    \draw[graph line] (LM)--(TL);
    \draw[graph line] (TM)--(a);
    \draw[graph line] (a)--(b);
    \draw[graph line] (LM)--(a);
    \draw[graph line] (LM)--(c);
    \draw[graph line] (c)--(b);
    \draw[graph line] (c)--(BM);
    \draw[graph line] (b)--(d);
    \draw[graph line] (d)--(BM);
    \draw[graph line] (d)--(RM);
    \draw[candidate green] (TM) .. controls (2.02,2.26) and (1.97,1.20) .. (d);
    \draw[candidate purple] (TR) .. controls (2.79,2.82) and (2.05,1.91) .. (b);
    \draw[candidate orange] (a) .. controls (2.03,1.87) and (2.70,1.78) .. (RM);
    % Two short blue half-edges at every outer attachment vertex.
    \draw[external halfedge] (TL)--++(-.42,0) (TL)--++(0,.42);
    \draw[external halfedge] (TR)--++(.42,0) (TR)--++(0,.42);
    \draw[external halfedge] (BL)--++(-.42,0) (BL)--++(0,-.42);
    \draw[external halfedge] (BR)--++(.42,0) (BR)--++(0,-.42);
    \node[graph dot] at (TL) {};
    \node[graph dot] at (TM) {};
    \node[graph dot] at (TR) {};
    \node[graph dot] at (LM) {};
    \node[graph dot] at (RM) {};
    \node[graph dot] at (BL) {};
    \node[graph dot] at (BM) {};
    \node[graph dot] at (BR) {};
    \node[graph dot] at (a) {};
    \node[graph dot] at (b) {};
    \node[graph dot] at (c) {};
    \node[graph dot] at (d) {};
  \end{scope}
  % Panel 2
  \begin{scope}[shift={(4.12,0)}]
    \coordinate (TL) at (0.000,2.850);
    \coordinate (TM) at (1.570,2.850);
    \coordinate (TR) at (3.140,2.850);
    \coordinate (LM) at (0.000,1.500);
    \coordinate (RM) at (3.140,1.500);
    \coordinate (BL) at (0.000,0.000);
    \coordinate (BM) at (1.570,0.000);
    \coordinate (BR) at (3.140,0.000);
    \coordinate (a) at (1.570,2.030);
    \coordinate (b) at (1.570,1.470);
    \coordinate (c) at (0.680,0.670);
    \coordinate (d) at (2.350,0.680);
    \draw[graph line] (TL)--(TM);
    \draw[graph line] (TM)--(TR);
    \draw[graph line] (TR)--(RM);
    \draw[graph line] (RM)--(BR);
    \draw[graph line] (BR)--(BM);
    \draw[graph line] (BM)--(BL);
    \draw[graph line] (BL)--(LM);
    \draw[graph line] (LM)--(TL);
    \draw[graph line] (TM)--(a);
    \draw[graph line] (a)--(b);
    \draw[graph line] (LM)--(c);
    \draw[graph line] (c)--(b);
    \draw[graph line] (c)--(BM);
    \draw[graph line] (b)--(d);
    \draw[graph line] (d)--(BM);
    \draw[graph line] (d)--(RM);
    \draw[graph line] (a)--(RM);
    \draw[candidate green] (c) .. controls (0.94,1.42) and (1.25,2.48) .. (TM);
    \draw[candidate purple] (TL) .. controls (0.45,2.32) and (1.27,1.78) .. (b);
    % Two short blue half-edges at every outer attachment vertex.
    \draw[external halfedge] (TL)--++(-.42,0) (TL)--++(0,.42);
    \draw[external halfedge] (TR)--++(.42,0) (TR)--++(0,.42);
    \draw[external halfedge] (BL)--++(-.42,0) (BL)--++(0,-.42);
    \draw[external halfedge] (BR)--++(.42,0) (BR)--++(0,-.42);
    \node[graph dot] at (TL) {};
    \node[graph dot] at (TM) {};
    \node[graph dot] at (TR) {};
    \node[graph dot] at (LM) {};
    \node[graph dot] at (RM) {};
    \node[graph dot] at (BL) {};
    \node[graph dot] at (BM) {};
    \node[graph dot] at (BR) {};
    \node[graph dot] at (a) {};
    \node[graph dot] at (b) {};
    \node[graph dot] at (c) {};
    \node[graph dot] at (d) {};
  \end{scope}
  % Panel 3
  \begin{scope}[shift={(8.24,0)}]
    \coordinate (TL) at (0.000,2.850);
    \coordinate (TM) at (1.570,2.850);
    \coordinate (TR) at (3.140,2.850);
    \coordinate (LM) at (0.000,1.500);
    \coordinate (RM) at (3.140,1.500);
    \coordinate (BL) at (0.000,0.000);
    \coordinate (BM) at (1.570,0.000);
    \coordinate (BR) at (3.140,0.000);
    \coordinate (a) at (0.850,2.050);
    \coordinate (b) at (1.420,1.460);
    \coordinate (c) at (2.340,1.470);
    \coordinate (d) at (0.690,0.580);
    \draw[graph line] (TL)--(TM);
    \draw[graph line] (TM)--(TR);
    \draw[graph line] (TR)--(RM);
    \draw[graph line] (RM)--(BR);
    \draw[graph line] (BR)--(BM);
    \draw[graph line] (BM)--(BL);
    \draw[graph line] (BL)--(LM);
    \draw[graph line] (LM)--(TL);
    \draw[graph line] (TM)--(a);
    \draw[graph line] (a)--(LM);
    \draw[graph line] (a)--(b);
    \draw[graph line] (TM)--(c);
    \draw[graph line] (b)--(c);
    \draw[graph line] (c)--(RM);
    \draw[graph line] (LM)--(d);
    \draw[graph line] (d)--(b);
    \draw[graph line] (d)--(BM);
    \draw[candidate green] (d) .. controls (1.34,0.76) and (2.51,1.22) .. (RM);
    \draw[candidate purple] (b) .. controls (1.91,1.15) and (2.83,0.60) .. (BR);
    % Two short blue half-edges at every outer attachment vertex.
    \draw[external halfedge] (TL)--++(-.42,0) (TL)--++(0,.42);
    \draw[external halfedge] (TR)--++(.42,0) (TR)--++(0,.42);
    \draw[external halfedge] (BL)--++(-.42,0) (BL)--++(0,-.42);
    \draw[external halfedge] (BR)--++(.42,0) (BR)--++(0,-.42);
    \node[graph dot] at (TL) {};
    \node[graph dot] at (TM) {};
    \node[graph dot] at (TR) {};
    \node[graph dot] at (LM) {};
    \node[graph dot] at (RM) {};
    \node[graph dot] at (BL) {};
    \node[graph dot] at (BM) {};
    \node[graph dot] at (BR) {};
    \node[graph dot] at (a) {};
    \node[graph dot] at (b) {};
    \node[graph dot] at (c) {};
    \node[graph dot] at (d) {};
  \end{scope}
  % Panel 4
  \begin{scope}[shift={(12.36,0)}]
    \coordinate (TL) at (0.000,2.850);
    \coordinate (TM) at (1.570,2.850);
    \coordinate (TR) at (3.140,2.850);
    \coordinate (LM) at (0.000,1.500);
    \coordinate (RM) at (3.140,1.500);
    \coordinate (BL) at (0.000,0.000);
    \coordinate (BM) at (1.570,0.000);
    \coordinate (BR) at (3.140,0.000);
    \coordinate (p) at (0.750,1.480);
    \coordinate (q) at (1.400,1.470);
    \coordinate (r) at (2.520,2.200);
    \coordinate (s) at (2.390,0.640);
    \draw[graph line] (TL)--(TM);
    \draw[graph line] (TM)--(TR);
    \draw[graph line] (TR)--(RM);
    \draw[graph line] (RM)--(BR);
    \draw[graph line] (BR)--(BM);
    \draw[graph line] (BM)--(BL);
    \draw[graph line] (BL)--(LM);
    \draw[graph line] (LM)--(TL);
    \draw[graph line] (LM)--(p);
    \draw[graph line] (p)--(q);
    \draw[graph line] (TM)--(p);
    \draw[graph line] (TM)--(r);
    \draw[graph line] (r)--(q);
    \draw[graph line] (r)--(RM);
    \draw[graph line] (q)--(s);
    \draw[graph line] (s)--(RM);
    \draw[graph line] (s)--(BM);
    \draw[candidate purple] (BL) .. controls (0.41,0.50) and (0.85,1.10) .. (q);
    % Two short blue half-edges at every outer attachment vertex.
    \draw[external halfedge] (TL)--++(-.42,0) (TL)--++(0,.42);
    \draw[external halfedge] (TR)--++(.42,0) (TR)--++(0,.42);
    \draw[external halfedge] (BL)--++(-.42,0) (BL)--++(0,-.42);
    \draw[external halfedge] (BR)--++(.42,0) (BR)--++(0,-.42);
    \node[graph dot] at (TL) {};
    \node[graph dot] at (TM) {};
    \node[graph dot] at (TR) {};
    \node[graph dot] at (LM) {};
    \node[graph dot] at (RM) {};
    \node[graph dot] at (BL) {};
    \node[graph dot] at (BM) {};
    \node[graph dot] at (BR) {};
    \node[graph dot] at (p) {};
    \node[graph dot] at (q) {};
    \node[graph dot] at (r) {};
    \node[graph dot] at (s) {};
  \end{scope}
\end{tikzpicture}%
}
  \caption{For each graph drawn, consider the graph obtained by removing one or both colored edges. This gives $4 + 3 + 3 + 2 = 12$ possible graphs of order 36, obtained by putting the same 5-gon exterior around each local configuration.}
    \label{fig:hypo36_exam}
\end{figure}
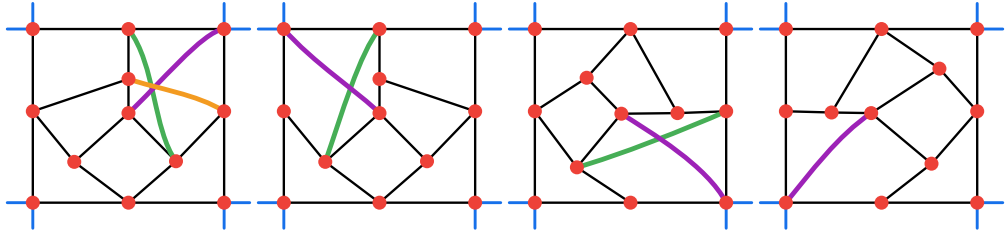

 %   \begin{figure}
 %   \centering
 %   \includegraphics[width=0.8\linewidth]{Fig22-Hypo36_Exam2.jpg}
 %   \caption{For each graph drawn, consider the graph obtained by removing one or both colored edges. This gives $4 + 3 + 3 + 2 = 12$ possible graphs of order 36, obtained by putting the same 5-gon exterior around each local configuration.}
 %   \label{fig:hypo36_exam}
 %   \end{figure}

    The fact that they are pairwise non-isomorphic follows from Tsai's exterior having no symmetry and the differing degree sequences along the boundary of the local configuration. 
\end{proof}

\subsection{The Plus Three Modification}

To preserve the bipartite condition with two more green vertices, one needs to add an \emph{even} number of vertices. Therefore, our previous modifications increase the number of vertices in $G$ by an even number. 

Next, we define a modification that increases the number of vertices by three. Let $G$ be a graph containing $Q$, the three 5-gon configuration in Figure \ref{fig:plus3}. Let $G'$ be the graph obtained by replacing $Q$ with $Q'$, the five 5-gon configuration in Figure \ref{fig:plus3}.

% Complete figure: requires only \usepackage{tikz} and \usepackage{graphicx}.
% No custom commands or styles are needed in the document preamble.
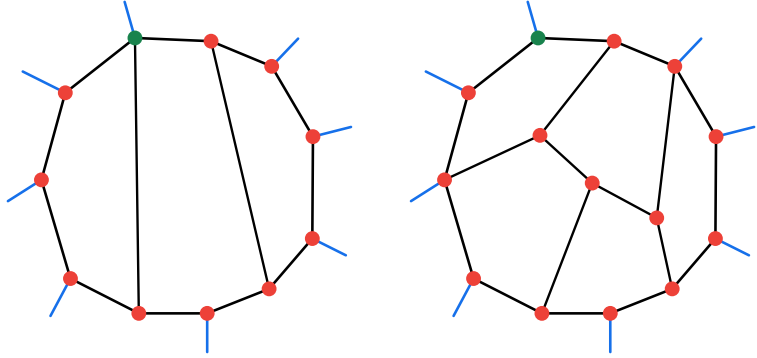
\begin{figure}[tbp]
\centering
\resizebox{.60\linewidth}{!}{%
\begin{tikzpicture}[line cap=round,line join=round]
\definecolor{figRed}{HTML}{ED4138}
\definecolor{figBlue}{HTML}{1670EA}
\definecolor{figGreen}{HTML}{19834D}
\definecolor{figOrange}{HTML}{F6BF65}
\definecolor{figViolet}{HTML}{B023B9}
\definecolor{figLime}{HTML}{55B752}
\tikzset{
 fedge/.style={black,line width=1.05pt},
 fboundary/.style={black,line width=1.15pt},
 fstubs/.style={figBlue,line width=1.15pt},
 fvertex/.style={circle,fill=figRed,inner sep=0pt,minimum size=2.3mm},
 fgreen/.style={fvertex,fill=figGreen},
 fnew/.style={fvertex,fill=figRed},
 fhighlight/.style={figOrange,line width=1.8pt,opacity=1,line cap=round,line join=round},
 fviolet/.style={figViolet,line width=1.8pt,line cap=round},
 flime/.style={figLime,line width=1.8pt,line cap=round},
}
\begin{scope}[shift={(0,0)}]
\begin{scope}
\coordinate (A) at (1.640,4.520);
\coordinate (B) at (2.820,4.470);
\coordinate (C) at (3.760,4.080);
\coordinate (D) at (4.400,2.990);
\coordinate (E) at (4.390,1.410);
\coordinate (F) at (3.720,0.630);
\coordinate (G) at (2.760,0.250);
\coordinate (H) at (1.700,0.250);
\coordinate (I) at (0.640,0.790);
\coordinate (J) at (0.190,2.320);
\coordinate (K) at (0.560,3.670);
\draw[fboundary] (A)--(B);
\draw[fboundary] (B)--(C);
\draw[fboundary] (C)--(D);
\draw[fboundary] (D)--(E);
\draw[fboundary] (E)--(F);
\draw[fboundary] (F)--(G);
\draw[fboundary] (G)--(H);
\draw[fboundary] (H)--(I);
\draw[fboundary] (I)--(J);
\draw[fboundary] (J)--(K);
\draw[fboundary] (K)--(A);
\draw[fstubs] (A)--(1.480,5.080);
\draw[fstubs] (C)--(4.170,4.510);
\draw[fstubs] (D)--(4.990,3.140);
\draw[fstubs] (E)--(4.910,1.150);
\draw[fstubs] (G)--(2.760,-0.350);
\draw[fstubs] (I)--(0.330,0.210);
\draw[fstubs] (J)--(-0.330,1.990);
\draw[fstubs] (K)--(-0.100,4.000);
\draw[fedge] (A)--(H);
\draw[fedge] (B)--(F);
\node[fgreen] at (A) {};
\node[fvertex] at (B) {};
\node[fvertex] at (C) {};
\node[fvertex] at (D) {};
\node[fvertex] at (E) {};
\node[fvertex] at (F) {};
\node[fvertex] at (G) {};
\node[fvertex] at (H) {};
\node[fvertex] at (I) {};
\node[fvertex] at (J) {};
\node[fvertex] at (K) {};
\end{scope}
\end{scope}
\begin{scope}[shift={(6.25,0)}]
\begin{scope}
\coordinate (A) at (1.640,4.520);
\coordinate (B) at (2.820,4.470);
\coordinate (C) at (3.760,4.080);
\coordinate (D) at (4.400,2.990);
\coordinate (E) at (4.390,1.410);
\coordinate (F) at (3.720,0.630);
\coordinate (G) at (2.760,0.250);
\coordinate (H) at (1.700,0.250);
\coordinate (I) at (0.640,0.790);
\coordinate (J) at (0.190,2.320);
\coordinate (K) at (0.560,3.670);
\coordinate (P) at (1.670,3.010);
\coordinate (Q) at (2.480,2.270);
\coordinate (R) at (3.480,1.730);
\draw[fboundary] (A)--(B);
\draw[fboundary] (B)--(C);
\draw[fboundary] (C)--(D);
\draw[fboundary] (D)--(E);
\draw[fboundary] (E)--(F);
\draw[fboundary] (F)--(G);
\draw[fboundary] (G)--(H);
\draw[fboundary] (H)--(I);
\draw[fboundary] (I)--(J);
\draw[fboundary] (J)--(K);
\draw[fboundary] (K)--(A);
\draw[fstubs] (A)--(1.480,5.080);
\draw[fstubs] (C)--(4.170,4.510);
\draw[fstubs] (D)--(4.990,3.140);
\draw[fstubs] (E)--(4.910,1.150);
\draw[fstubs] (G)--(2.760,-0.350);
\draw[fstubs] (I)--(0.330,0.210);
\draw[fstubs] (J)--(-0.330,1.990);
\draw[fstubs] (K)--(-0.100,4.000);
\draw[fedge] (J)--(P);
\draw[fedge] (B)--(P);
\draw[fedge] (P)--(Q);
\draw[fedge] (Q)--(H);
\draw[fedge] (Q)--(R);
\draw[fedge] (C)--(R);
\draw[fedge] (R)--(F);
\node[fgreen] at (A) {};
\node[fvertex] at (B) {};
\node[fvertex] at (C) {};
\node[fvertex] at (D) {};
\node[fvertex] at (E) {};
\node[fvertex] at (F) {};
\node[fvertex] at (G) {};
\node[fvertex] at (H) {};
\node[fvertex] at (I) {};
\node[fvertex] at (J) {};
\node[fvertex] at (K) {};
\node[fvertex] at (P) {};
\node[fvertex] at (Q) {};
\node[fvertex] at (R) {};
\end{scope}
\end{scope}
\end{tikzpicture}%
}
   \caption{Replace $Q$, the three 5-gon configuration on the left, with $Q'$, the five 5-gon configuration on the right by adding three vertices}
    \label{fig:plus3}
\end{figure}

 %\begin{figure}
 %   \centering
 %   \includegraphics[width=0.75\linewidth]{Fig23-Plus3.jpg}
 %   \caption{Replace $Q$, the three 5-gon configuration on the left, with $Q'$, the five 5-gon configuration on the right by adding three vertices}
 %   \label{fig:plus3}
%\end{figure}

Notice that $Q$ has a single vertex of degree 4 in $G$ whereas $Q'$ has two vertices of degree 4 in $G'$. 

\begin{thm} \label{thm:path_ext}
    If $G$ is Hamiltonian then $G'$ is Hamiltonian. More generally, if $G$ has a path visiting each vertex of $Q$ with endpoints not in $Q$ then $G'$ has a path visiting each vertex of $Q'$ with the same endpoints.
\end{thm}

\begin{proof}
    A path in $G$ visiting each vertex of $Q$ (and not beginning nor ending in $Q$) needs to use one of the two edges depicted in Figure \ref{fig:plus3Ham}. Replacing each edge in $G$ with the shown path in $G'$, builds a path in $G'$ visiting each vertex of $Q'$ with the same endpoints. 

% Figures 24 and 25 together: four drawings in two side-by-side pairs.
% Corrected highlighted edge in the leftmost 11-gon: B--C, NOT A--B.
% Complete figure: needs only \usepackage{tikz,graphicx} in the preamble.
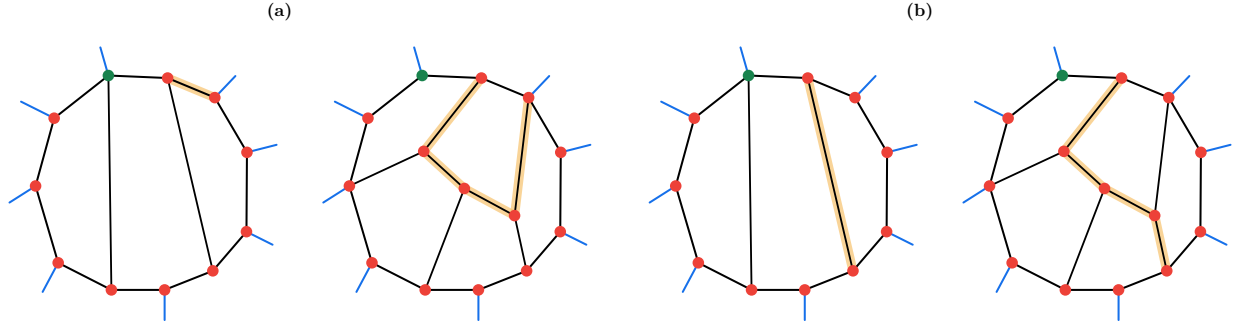
\begin{figure}[tbp]
\centering
\resizebox{.98\linewidth}{!}{%
\begin{tikzpicture}[line cap=round,line join=round]
\definecolor{figRed}{HTML}{ED4138}
\definecolor{figBlue}{HTML}{1670EA}
\definecolor{figGreen}{HTML}{19834D}
\definecolor{figOrange}{HTML}{F6BF65}
\definecolor{figViolet}{HTML}{B023B9}
\definecolor{figLime}{HTML}{55B752}
\tikzset{
 fedge/.style={black,line width=1.05pt},
 fboundary/.style={black,line width=1.15pt},
 fstubs/.style={figBlue,line width=1.15pt},
 fvertex/.style={circle,fill=figRed,inner sep=0pt,minimum size=2.3mm},
 fgreen/.style={fvertex,fill=figGreen},
 fnew/.style={fvertex,fill=figRed},
 fhighlight/.style={figOrange,line width=5.2pt,opacity=.67,line cap=round,line join=round},
 fviolet/.style={figViolet,line width=1.8pt,line cap=round},
 flime/.style={figLime,line width=1.8pt,line cap=round},
}

% (a) First replacement: original, then modified graph.
\begin{scope}[shift={(0,0)}]
\begin{scope}[shift={(0,0)}]
\begin{scope}
\coordinate (A) at (1.640,4.520);
\coordinate (B) at (2.820,4.470);
\coordinate (C) at (3.760,4.080);
\coordinate (D) at (4.400,2.990);
\coordinate (E) at (4.390,1.410);
\coordinate (F) at (3.720,0.630);
\coordinate (G) at (2.760,0.250);
\coordinate (H) at (1.700,0.250);
\coordinate (I) at (0.640,0.790);
\coordinate (J) at (0.190,2.320);
\coordinate (K) at (0.560,3.670);
\draw[fhighlight] (B)--(C);
\draw[fboundary] (A)--(B);
\draw[fboundary] (B)--(C);
\draw[fboundary] (C)--(D);
\draw[fboundary] (D)--(E);
\draw[fboundary] (E)--(F);
\draw[fboundary] (F)--(G);
\draw[fboundary] (G)--(H);
\draw[fboundary] (H)--(I);
\draw[fboundary] (I)--(J);
\draw[fboundary] (J)--(K);
\draw[fboundary] (K)--(A);
\draw[fstubs] (A)--(1.480,5.080);
\draw[fstubs] (C)--(4.170,4.510);
\draw[fstubs] (D)--(4.990,3.140);
\draw[fstubs] (E)--(4.910,1.150);
\draw[fstubs] (G)--(2.760,-0.350);
\draw[fstubs] (I)--(0.330,0.210);
\draw[fstubs] (J)--(-0.330,1.990);
\draw[fstubs] (K)--(-0.100,4.000);
\draw[fedge] (A)--(H);
\draw[fedge] (B)--(F);
\node[fgreen] at (A) {};
\node[fvertex] at (B) {};
\node[fvertex] at (C) {};
\node[fvertex] at (D) {};
\node[fvertex] at (E) {};
\node[fvertex] at (F) {};
\node[fvertex] at (G) {};
\node[fvertex] at (H) {};
\node[fvertex] at (I) {};
\node[fvertex] at (J) {};
\node[fvertex] at (K) {};
\end{scope}
\end{scope}
\begin{scope}[shift={(6.25,0)}]
\begin{scope}
\coordinate (A) at (1.640,4.520);
\coordinate (B) at (2.820,4.470);
\coordinate (C) at (3.760,4.080);
\coordinate (D) at (4.400,2.990);
\coordinate (E) at (4.390,1.410);
\coordinate (F) at (3.720,0.630);
\coordinate (G) at (2.760,0.250);
\coordinate (H) at (1.700,0.250);
\coordinate (I) at (0.640,0.790);
\coordinate (J) at (0.190,2.320);
\coordinate (K) at (0.560,3.670);
\coordinate (P) at (1.670,3.010);
\coordinate (Q) at (2.480,2.270);
\coordinate (R) at (3.480,1.730);
\draw[fhighlight] (B)--(P);
\draw[fhighlight] (P)--(Q);
\draw[fhighlight] (Q)--(R);
\draw[fhighlight] (R)--(C);
\draw[fboundary] (A)--(B);
\draw[fboundary] (B)--(C);
\draw[fboundary] (C)--(D);
\draw[fboundary] (D)--(E);
\draw[fboundary] (E)--(F);
\draw[fboundary] (F)--(G);
\draw[fboundary] (G)--(H);
\draw[fboundary] (H)--(I);
\draw[fboundary] (I)--(J);
\draw[fboundary] (J)--(K);
\draw[fboundary] (K)--(A);
\draw[fstubs] (A)--(1.480,5.080);
\draw[fstubs] (C)--(4.170,4.510);
\draw[fstubs] (D)--(4.990,3.140);
\draw[fstubs] (E)--(4.910,1.150);
\draw[fstubs] (G)--(2.760,-0.350);
\draw[fstubs] (I)--(0.330,0.210);
\draw[fstubs] (J)--(-0.330,1.990);
\draw[fstubs] (K)--(-0.100,4.000);
\draw[fedge] (J)--(P);
\draw[fedge] (B)--(P);
\draw[fedge] (P)--(Q);
\draw[fedge] (Q)--(H);
\draw[fedge] (Q)--(R);
\draw[fedge] (C)--(R);
\draw[fedge] (R)--(F);
\node[fgreen] at (A) {};
\node[fvertex] at (B) {};
\node[fvertex] at (C) {};
\node[fvertex] at (D) {};
\node[fvertex] at (E) {};
\node[fvertex] at (F) {};
\node[fvertex] at (G) {};
\node[fvertex] at (H) {};
\node[fvertex] at (I) {};
\node[fvertex] at (J) {};
\node[fvertex] at (K) {};
\node[fvertex] at (P) {};
\node[fvertex] at (Q) {};
\node[fvertex] at (R) {};
\end{scope}
\end{scope}
\node[font=\small] at (5.05,5.80) {\textbf{(a)}};
\end{scope}
% (b) Second replacement: original, then modified graph.
\begin{scope}[shift={(12.75,0)}]
\begin{scope}[shift={(0,0)}]
\begin{scope}
\coordinate (A) at (1.640,4.520);
\coordinate (B) at (2.820,4.470);
\coordinate (C) at (3.760,4.080);
\coordinate (D) at (4.400,2.990);
\coordinate (E) at (4.390,1.410);
\coordinate (F) at (3.720,0.630);
\coordinate (G) at (2.760,0.250);
\coordinate (H) at (1.700,0.250);
\coordinate (I) at (0.640,0.790);
\coordinate (J) at (0.190,2.320);
\coordinate (K) at (0.560,3.670);
\draw[fhighlight] (B)--(F);
\draw[fboundary] (A)--(B);
\draw[fboundary] (B)--(C);
\draw[fboundary] (C)--(D);
\draw[fboundary] (D)--(E);
\draw[fboundary] (E)--(F);
\draw[fboundary] (F)--(G);
\draw[fboundary] (G)--(H);
\draw[fboundary] (H)--(I);
\draw[fboundary] (I)--(J);
\draw[fboundary] (J)--(K);
\draw[fboundary] (K)--(A);
\draw[fstubs] (A)--(1.480,5.080);
\draw[fstubs] (C)--(4.170,4.510);
\draw[fstubs] (D)--(4.990,3.140);
\draw[fstubs] (E)--(4.910,1.150);
\draw[fstubs] (G)--(2.760,-0.350);
\draw[fstubs] (I)--(0.330,0.210);
\draw[fstubs] (J)--(-0.330,1.990);
\draw[fstubs] (K)--(-0.100,4.000);
\draw[fedge] (A)--(H);
\draw[fedge] (B)--(F);
\node[fgreen] at (A) {};
\node[fvertex] at (B) {};
\node[fvertex] at (C) {};
\node[fvertex] at (D) {};
\node[fvertex] at (E) {};
\node[fvertex] at (F) {};
\node[fvertex] at (G) {};
\node[fvertex] at (H) {};
\node[fvertex] at (I) {};
\node[fvertex] at (J) {};
\node[fvertex] at (K) {};
\end{scope}
\end{scope}
\begin{scope}[shift={(6.25,0)}]
\begin{scope}
\coordinate (A) at (1.640,4.520);
\coordinate (B) at (2.820,4.470);
\coordinate (C) at (3.760,4.080);
\coordinate (D) at (4.400,2.990);
\coordinate (E) at (4.390,1.410);
\coordinate (F) at (3.720,0.630);
\coordinate (G) at (2.760,0.250);
\coordinate (H) at (1.700,0.250);
\coordinate (I) at (0.640,0.790);
\coordinate (J) at (0.190,2.320);
\coordinate (K) at (0.560,3.670);
\coordinate (P) at (1.670,3.010);
\coordinate (Q) at (2.480,2.270);
\coordinate (R) at (3.480,1.730);
\draw[fhighlight] (B)--(P);
\draw[fhighlight] (P)--(Q);
\draw[fhighlight] (Q)--(R);
\draw[fhighlight] (R)--(F);
\draw[fboundary] (A)--(B);
\draw[fboundary] (B)--(C);
\draw[fboundary] (C)--(D);
\draw[fboundary] (D)--(E);
\draw[fboundary] (E)--(F);
\draw[fboundary] (F)--(G);
\draw[fboundary] (G)--(H);
\draw[fboundary] (H)--(I);
\draw[fboundary] (I)--(J);
\draw[fboundary] (J)--(K);
\draw[fboundary] (K)--(A);
\draw[fstubs] (A)--(1.480,5.080);
\draw[fstubs] (C)--(4.170,4.510);
\draw[fstubs] (D)--(4.990,3.140);
\draw[fstubs] (E)--(4.910,1.150);
\draw[fstubs] (G)--(2.760,-0.350);
\draw[fstubs] (I)--(0.330,0.210);
\draw[fstubs] (J)--(-0.330,1.990);
\draw[fstubs] (K)--(-0.100,4.000);
\draw[fedge] (J)--(P);
\draw[fedge] (B)--(P);
\draw[fedge] (P)--(Q);
\draw[fedge] (Q)--(H);
\draw[fedge] (Q)--(R);
\draw[fedge] (C)--(R);
\draw[fedge] (R)--(F);
\node[fgreen] at (A) {};
\node[fvertex] at (B) {};
\node[fvertex] at (C) {};
\node[fvertex] at (D) {};
\node[fvertex] at (E) {};
\node[fvertex] at (F) {};
\node[fvertex] at (G) {};
\node[fvertex] at (H) {};
\node[fvertex] at (I) {};
\node[fvertex] at (J) {};
\node[fvertex] at (K) {};
\node[fvertex] at (P) {};
\node[fvertex] at (Q) {};
\node[fvertex] at (R) {};
\end{scope}
\end{scope}
\node[font=\small] at (5.05,5.80) {\textbf{(b)}};
\end{scope}
\end{tikzpicture}%
}
  \caption{One can modify a Hamiltonian cycle in $G$ to get one in $G'$ by employing one of the two local modifications shown.}
    \label{fig:plus3Ham}
\end{figure}

 %\begin{figure}
 %   \centering
 %   \includegraphics[width=0.4\linewidth]{Fig24-3Plus_Ham1.jpg} \hspace*{1.5cm}
 %   \includegraphics[width=0.4\linewidth]{Fig25-3Plus_Ham2.jpg}
 %   \caption{One can modify a Hamiltonian path in $G$ to get one in $G'$ by employing one of the two local modifications shown.}
 %   \label{fig:plus3Ham}
%\end{figure}
\end{proof}

Unlike the previous subsections, we do not prove that if $G$ is hypohamiltonian then so is $G'$. Rather, we observe that this construction can be used when $G$ is Wiener's almost hypohamiltonian graph of order 31 to produce two hypohamiltonian graphs of order 34 and eight of order 37. This recovers all of Tsai's examples in \cite{Tsai25} and produces two new examples. Although we should remark that we only discovered the plus three modification by carefully comparing Tsai's order 34 and 37 examples. 

To begin the construction in a given graph $G$, one needs a degree 4 vertex. Remove one of its edges and then an edge of its neighbor that creates an 11-gon. In Wiener's example, there is a single degree 4 vertex outside the 4-gon configuration. And for each edge, there are two choices for the second edge to remove. For the two choices in Figure \ref{fig:3Plus34}, we recover Tsai's order 34 examples. All other choices give graphs isomorphic to one of these two. 

 \begin{figure}
    \centering
    \includegraphics[width=0.8\linewidth]{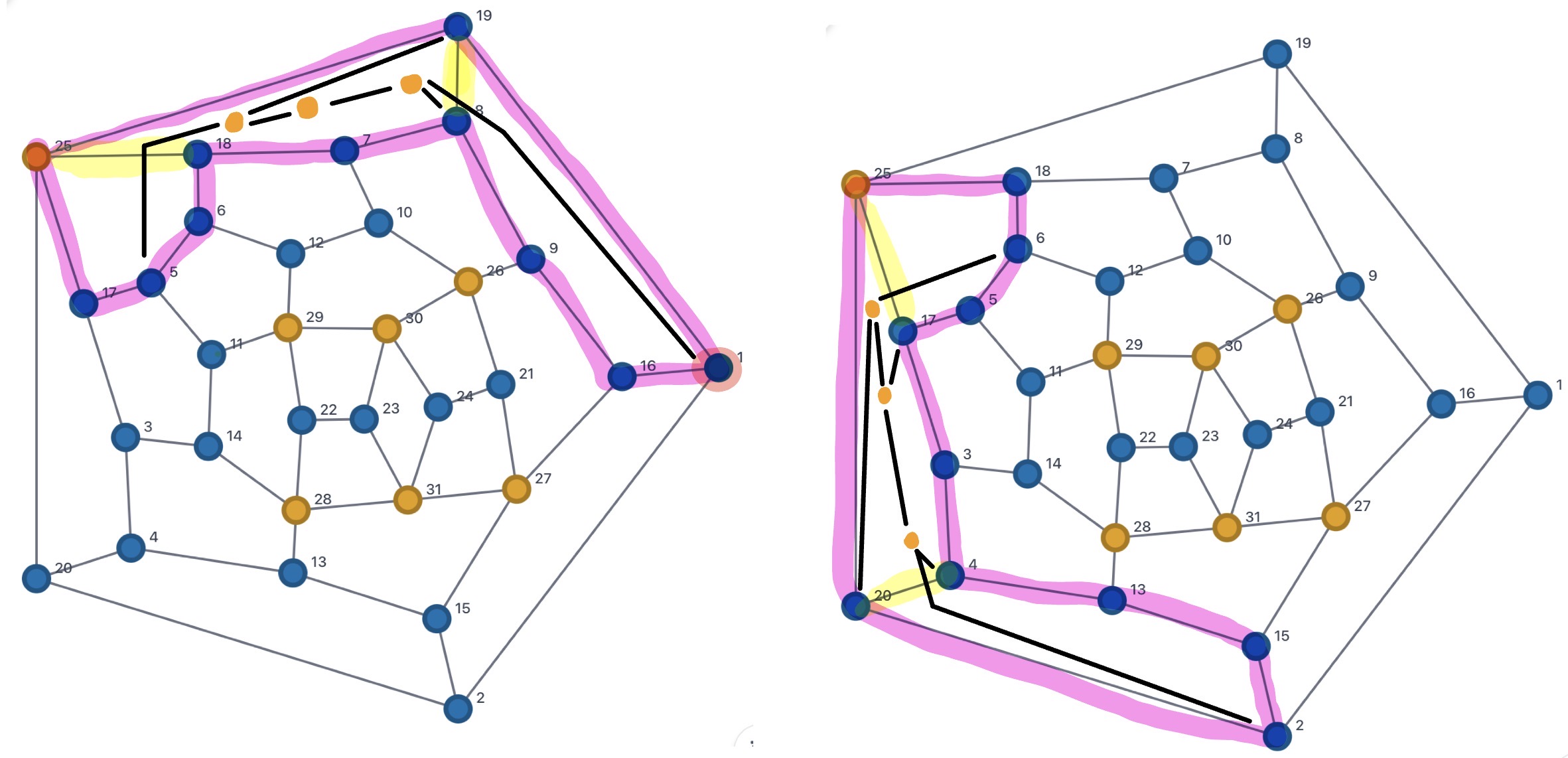}
    \caption{The plus three modification for the indicated $Q$ give Tsai's order 34 hypohamiltonian graphs}
    \label{fig:3Plus34}
\end{figure}

Now notice that for the four choices shown in Figure \ref{fig:3Plus37}, we recover Tsai's order 37 examples (up to rotating the central copy of $H$) and produce an additional example.

 \begin{figure}
    \centering
    \includegraphics[width=0.8\linewidth]{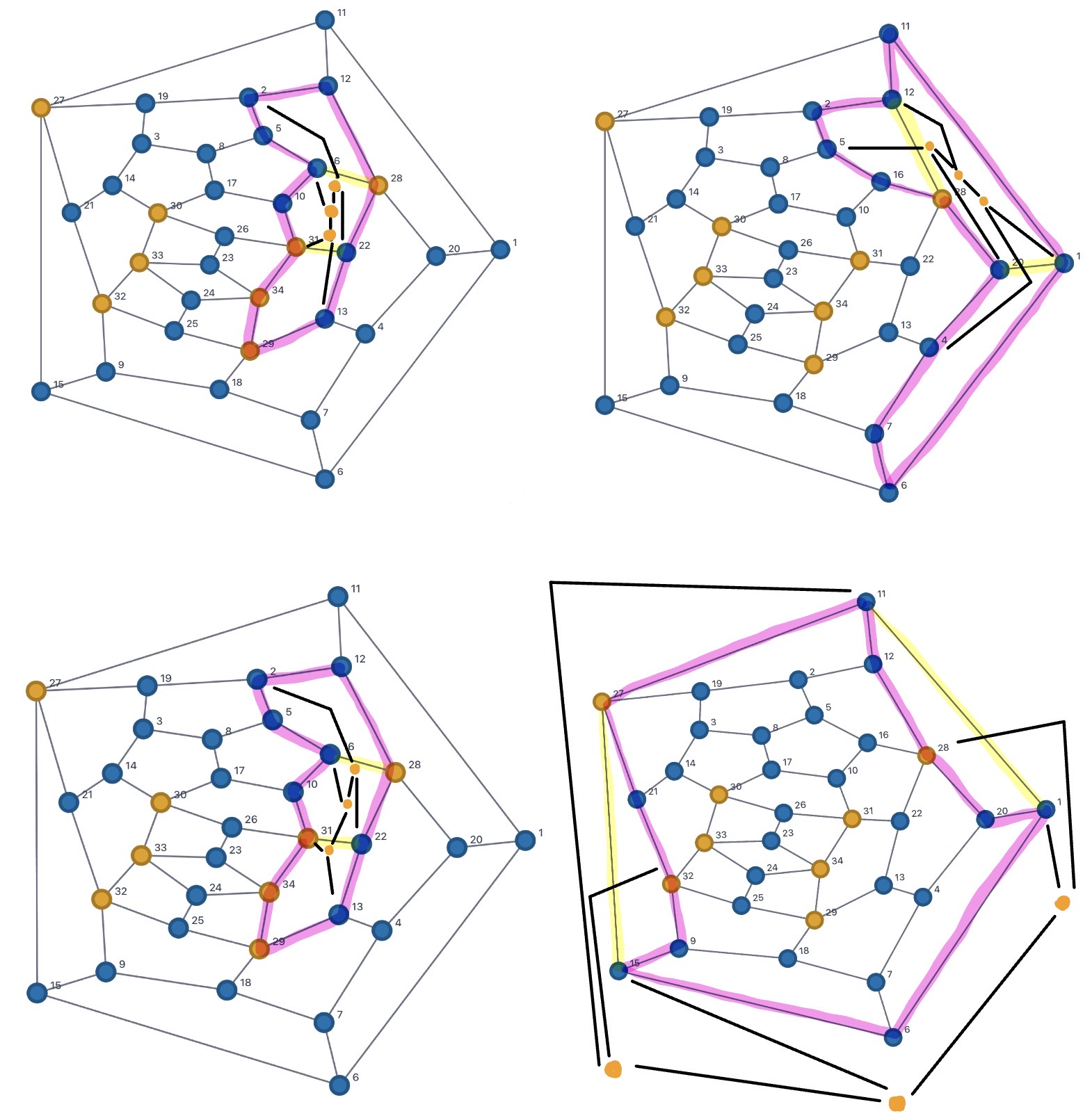}
    \caption{The plus three modification for the indicated $Q$ gives Tsai's order 37 hypohamiltonian graphs and one new example.}
    \label{fig:3Plus37}
\end{figure}

Up to checking the hypohamiltonian property for vertices in $Q'$, we have shown the following: 

\begin{cor} \label{cor:order37}
    There are at least 8 planar hypohamiltonian graphs of order 37. 
\end{cor}

We can use hypohamiltonian graphs of order 36 and 37 to build hypohamiltonian graphs of order 39. 

\begin{cor} \label{cor:order39}
    There are at least 48 planar hypohamiltonian graphs of order 39. 
\end{cor}

\begin{proof}
    We want to show that the plus three modification preserves the hypohamiltonian condition for the 12 planar hypohamiltonian graphs of order 36 in Corollary \ref{cor:order36}, giving four choices of exterior for a total of 48 graphs. 
    
    For the removal of the internal vertices, Theorem \ref{thm:path_ext} ensures that the four paths in Corollary \ref{cor:order36} and Corollary \ref{cor:G1prime_hypo} extend under the plus three modification. 

    Next, view each graph as obtained from the (imperfect) plus two modification of a graph of size 37. Then Corollary \ref{cor:G1prime_hypoexc} implies that each graph is hypohamiltonian except possibly at the internal vertices.  
\end{proof}

\subsection{Thomassen's construction revisited}

While Thomassen's original construction was in the setting of cubic graphs, it is commonly applied more generally to polyhedral graphs containing a 4-gon. Note that these applications are at times inconsistent or not fully justified. To remedy this, we proof several well-known theorems using cycle fragment endpoint analysis. 

Let $H$ and $H'$ be as in Figure \ref{fig:Thomassen}, but with each vertex of arbitrary degree at least three. Let $G$ be any graph containing $H$ and $G'$ the modification of $G$ by replacing $H$ with $H'$.

\begin{prop} \label{prop:Thom_Ham}
    If $G'$ is Hamiltonian then $G$ is Hamiltonian. Conversely, if $G$ is Hamiltonian and $H$ has a cubic vertex, then $G'$ is Hamiltonian.
\end{prop}

\begin{proof}
    If $G'$ has Hamiltonian cycle $C'$, then any fragment of $C'$ in $H'$ can be modified to a fragment in $H$ by replacing e.g., $aa'b'c'd'd$ or $aa'd'd$ with $ad$. Hence $G$ is Hamiltonian.

    Conversely, if $C$ is a Hamiltonian cycle in $G$ and a vertex $v \in H$ has degree three, then $C$ must use an edge in $H$. Suppose the edge is $ad$. Replacing it with $aa'b'c'd'd$ gives a fragment in $H'$ that extends to a Hamiltonian cycle $C'$ in $G'$.
 \end{proof}

Figure \ref{fig:Thomassen_Failure} demonstrates that the existence of a cubic vertex in $H$ is necessary.

% Complete TikZ figure, preserving the two configurations in the sketch.
% Requires \usepackage{tikz} and \usepackage{graphicx}.
\begin{figure}[tbp]
\centering
\resizebox{0.65\linewidth}{!}{%
\definecolor{figRed}{HTML}{ED4138}
\begin{tikzpicture}[line cap=round,line join=round,
  graph edge/.style={draw=black,line width=1.3pt},
  graph vertex/.style={circle,fill=figRed,inner sep=0pt,minimum size=3.0mm}]

% Left panel: eight boundary vertices and four interior edges.
\begin{scope}
  \coordinate (t) at (2.00,5.15);
  \coordinate (ul) at (1.05,4.20);
  \coordinate (ur) at (2.95,4.20);
  \coordinate (ml) at (0.00,3.15);
  \coordinate (mr) at (4.00,3.15);
  \coordinate (ll) at (1.05,2.10);
  \coordinate (lr) at (2.95,2.10);
  \coordinate (b) at (2.00,1.08);
  % Outer eight-cycle.
  \draw[graph edge] (t)--(ul)--(ml)--(ll)--(b)--(lr)--(mr)--(ur)--cycle;
  % Four chords (including the two vertical chords).
  \draw[graph edge] (ul)--(ur) (ul)--(ll) (ur)--(lr) (ll)--(lr);
  \foreach \v in {t,ul,ur,ml,mr,ll,lr,b}{\node[graph vertex] at (\v) {};}
\end{scope}

% Right panel: same outside, with the two horizontal chords
% replaced by a four-cycle connected to its four neighbours.
\begin{scope}[xshift=6.20cm]
  \coordinate (t) at (2.00,5.15);
  \coordinate (ul) at (1.05,4.20);
  \coordinate (ur) at (2.95,4.20);
  \coordinate (ml) at (0.00,3.15);
  \coordinate (mr) at (4.00,3.15);
  \coordinate (ll) at (1.05,2.10);
  \coordinate (lr) at (2.95,2.10);
  \coordinate (b) at (2.00,1.08);
  \coordinate (p) at (1.58,3.57);
  \coordinate (q) at (2.42,3.57);
  \coordinate (r) at (1.58,2.74);
  \coordinate (s) at (2.42,2.74);
  % The outer octagon and two unchanged vertical chords.
  \draw[graph edge] (t)--(ul)--(ml)--(ll)--(b)--(lr)--(mr)--(ur)--cycle;
  \draw[graph edge] (ul)--(ll) (ur)--(lr);
  % Inserted four-cycle and the four attachment edges.
  \draw[graph edge] (p)--(q)--(s)--(r)--cycle;
  \draw[graph edge] (ul)--(p) (ur)--(q) (ll)--(r) (lr)--(s);
  \foreach \v in {t,ul,ur,ml,mr,ll,lr,b,p,q,r,s}{\node[graph vertex] at (\v) {};}
\end{scope}
\end{tikzpicture}%
}
 \caption{The left graph is Hamiltonian, while the right graph, obtained by applying Thomassen's construction to the central 4-gon, is not.}
   \label{fig:Thomassen_Failure}
\end{figure}
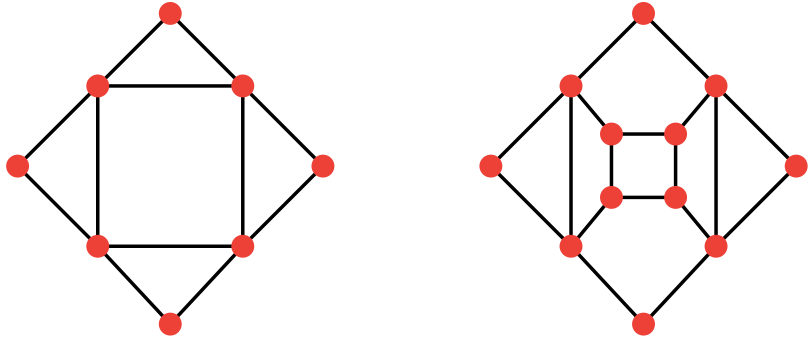

%\begin{figure}
%    \centering
%    \includegraphics[width=0.60\linewidth]{Fig28-Thomassen_Failure.jpg}
%    \caption{The left graph is Hamiltonian, while the right graph, obtained by applying Thomassen's construction, is not.}
%    \label{fig:Thomassen_Failure}
%\end{figure}

The following result and proof are due to C.T. Zamfirescu's in his PhD thesis. We only record them here for completeness. 

%The next result is due to  \cite{Z15}
\begin{prop} \label{prop:Thom_hypo} \cite[Prop 2.11]{Z16}
    If $G'$ is hypohamiltonian and $H$ has a cubic vertex, then $G$ is hypohamiltonian. Conversely, if $G$ is hypohamiltonian and $H$ has adjacent cubic vertices, then $G'$ is hypohamiltonian.
\end{prop}

\begin{proof}
Suppose that $G'$ is hypohamiltonian and $H$ has a cubic vertex. Then by Proposition \ref{prop:Thom_Ham}, $G$ is not Hamiltonian. 

By assumption $G'-v$ is Hamiltonian for all $v$. We want to show that $G-v$ is Hamiltonian for $v \in G$ and hence $v \not \in V' := \{ a', b', c', d' \}$.  Let $C'_v$ denote the Hamiltonian cycle in $G'-v$. Observe that $C_v'$ contains, up to cycle permutation of the vertices, the fragments $aa'b'c'd'd$ or $aa'd'd \cup bb'c'c$. Hence we can modify $C_v'$ by removing the instances of elements in $V'$ to get a Hamiltonian cycle $C_v$ in $G - v$.\\
\\
Conversely, suppose that $c$ and $d$ are adjacent cubic vertices in $G$. Assume that $G-v$ is Hamiltonian for all $v$. Then each Hamiltonian cycle must include one of the three edges in $H$: $bc$, $cd$, or $da$. Hence it can be extended to a Hamiltonian cycle in $G' - v$.\\
\\
It remains to show that $G' - v'$ is Hamiltonian for each $v' \in V'$. By symmetry, it suffices to show this for $a'$ and $d'$.

The left side of Figure \ref{fig:Thom_fragment} shows the two possible cycle fragments in $H - a$: $bcd$ and $b \cup cd$. Since $G-a$ is Hamiltonian, one of them extends to a Hamiltonian cycle in $G$. Consequently, one of the fragments shown in $G'$, with the same endpoint decompositions, extends to a Hamiltonian cycle in $G' - d'$. 

 \begin{figure}
    \centering
    \includegraphics[width=0.4\linewidth] {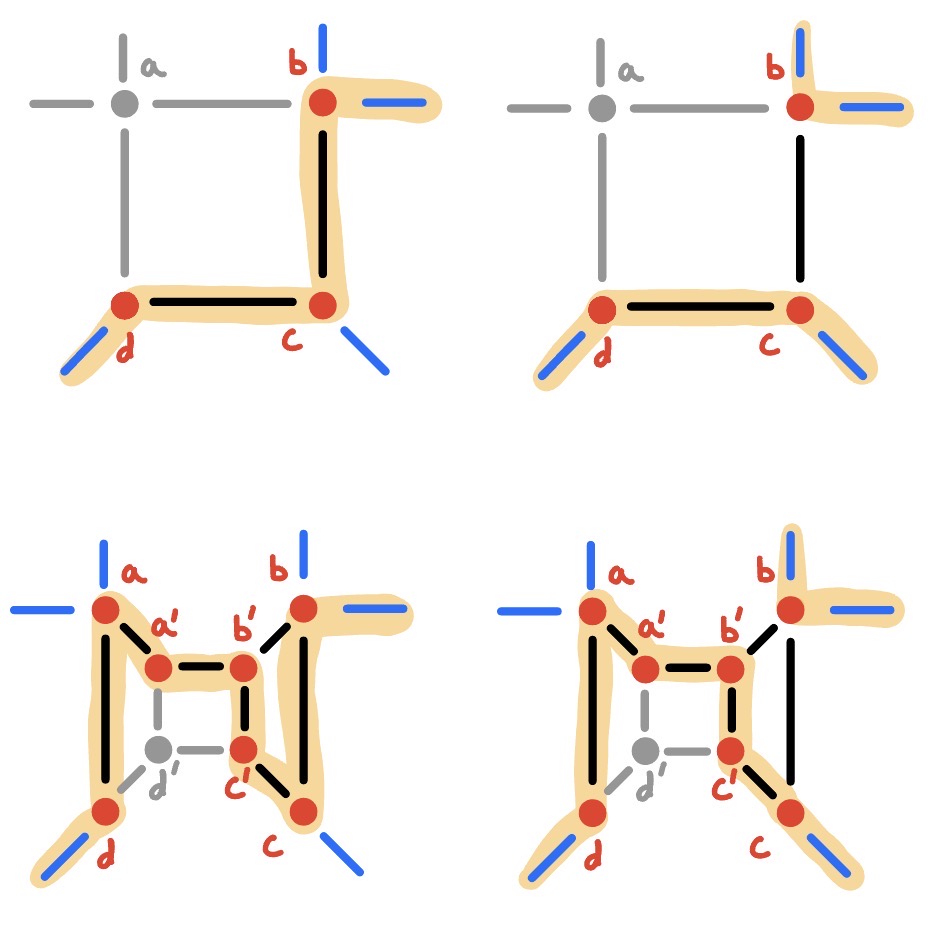} \hspace*{1cm}
     \includegraphics[width=0.4\linewidth]{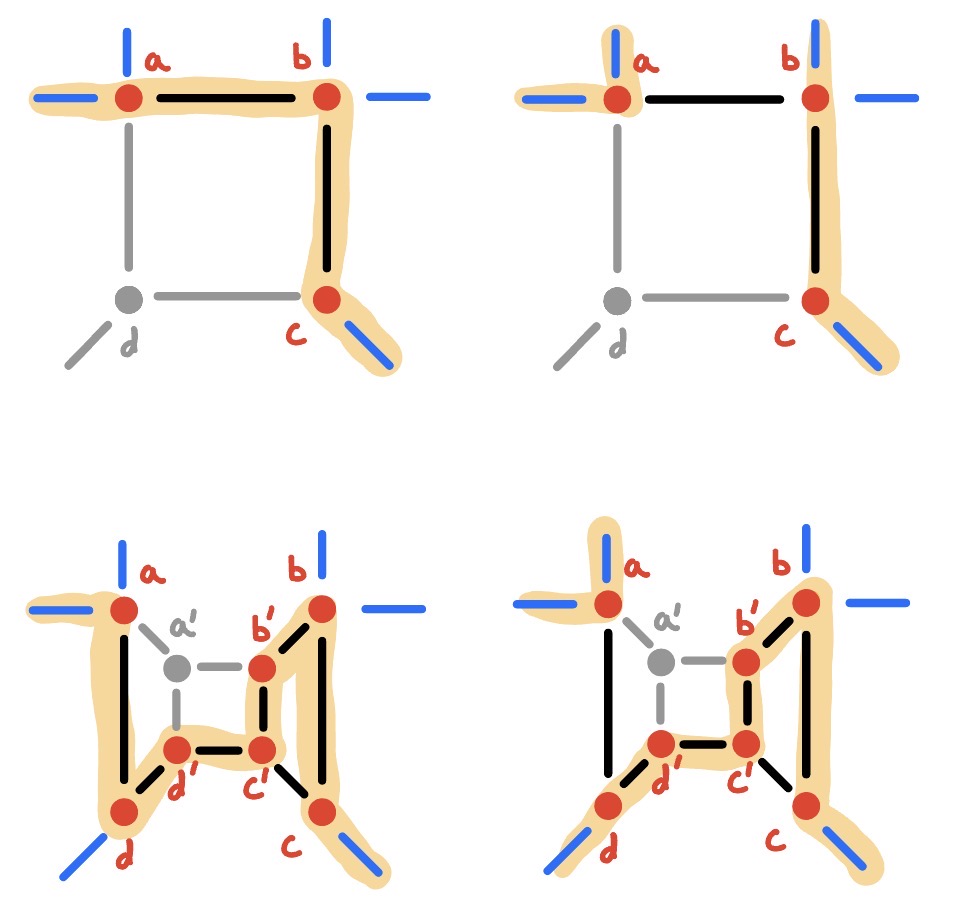}
    \caption{Endpoint decompositions in $G'-a'$ and $G'-d'$ match those in $G-d$ and $G-a$ respectively}
    \label{fig:Thom_fragment}
\end{figure}

Similarly, the right side of Figure \ref{fig:Thom_fragment} shows the two possible cycle fragments in $H - d$: $abc$ or $a \cup bc$. Since $G-d$ is Hamiltonian, this implies that one of these fragments extends to a Hamiltonian cycle. Therefore, the same is true of the two possible fragments in $H' - a'$. We conclude that $G'-a'$ is Hamiltonian. 
\end{proof}

\begin{rem} \label{rem:Tsai_mistake}
In the discussion following Lemma 3.3 in \cite{Tsai25}, Tsai asserts the following stronger claim: If $H$ has a cubic vertex and $G$ is hypohamiltonian, then $G'$ is hypohamiltonian. But his cited reference appears to prove the converse: $G'$ hypohamiltonian implies $G$ hypohamiltonian. 
\end{rem}

\begin{prop} \label{prop:Thom_hypo_paths}
    Suppose $G$ is hypohamiltonian and $H$ has a cubic vertex. Further assume that $G - \{b, d \}$ has a Hamiltonian path from $a$ to $c$ and $G - \{ a, c \}$ has a Hamiltonian path from $b$ to $d$. Then $G'$ is hypohamiltonian.
\end{prop}

\begin{proof}
    The proof of Proposition \ref{prop:Thom_hypo} implies that if $G$ is hypohamiltonian and $H$ has a cubic vertex, then $G'$ is hypohamiltonian, except possibly at $V' = \{ a', b', c', d' \}$ or at the cubic vertex in $H$, call it $d$. 
    
    The fragments in Figure \ref{fig:Thom_fragments2} can each be completed to a cycle using either a Hamiltonian path from $a$ to $c$ in $G - \{b, d \}$ or a Hamiltonian path from $b$ to $d$ in $G - \{ a, c \}$.

    \begin{figure}
        \centering
        \includegraphics[width=0.9\linewidth]{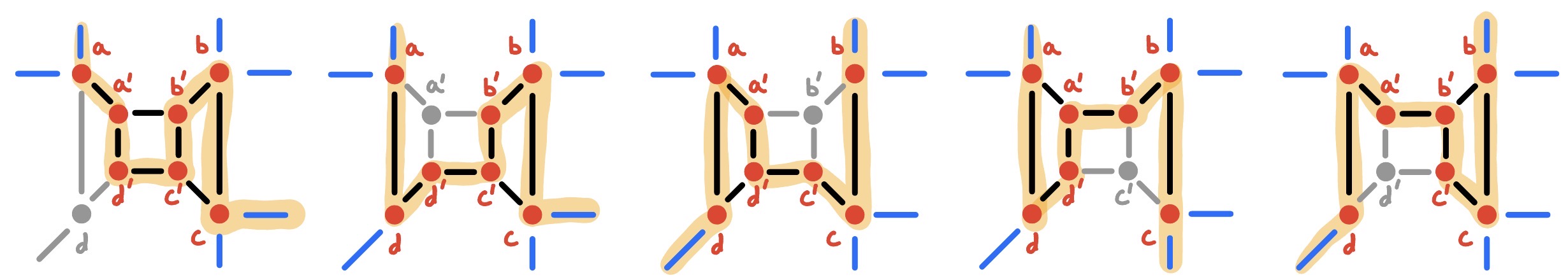}
        \caption{For each $v \in \{ a', b', c', d', d \}$, the graph $G'-v$ has a Hamiltonian cycle, if these fragments can be extended.}
        \label{fig:Thom_fragments2}
    \end{figure}
    
\end{proof}

\begin{rem}
    We are unsure if the existence of such Hamiltonian paths is automatic for hypohamiltonian graphs.     
\end{rem}

\begin{cor} \label{cor:Tsai38}
    In Tsai's 34 vertex hypohamiltonian graphs, one can apply Thomassen's plus four construction to the central 4-gon to obtain a hypohamiltonian graph with 38 vertices. 
\end{cor}

Of course Tsai already proved this computationally. We merely observe that it follows theoretically, from the existence of the same Hamiltonian paths that we found earlier. 

\begin{proof}
    Figure \ref{fig:Thom_Tsai} shows how to begin building Hamiltonian cycles from $b$ to $d$ in $G- \{a, c\}$ and from $a$ to $c$ to $G- \{b, d \}$. Each cycle fragment extends to a cycle in the exterior from Figure \ref{fig:H1Prime-Hypo}.  

% Complete Figure 31 -- only TikZ and graphicx packages are required.
% The two panels use EXACTLY the earlier graph H (the original unrotated H).
% External corner vertices 1,2,3,4 each have TWO blue dangling stubs.
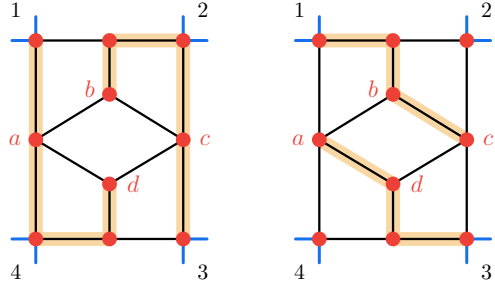
\begin{figure}[tbp]
\centering
\resizebox{.4\linewidth}{!}{%
\begin{tikzpicture}[line cap=round,line join=round]
\definecolor{graphRed}{HTML}{ED4138}
\definecolor{stubBlue}{HTML}{1670EA}
\definecolor{paleOrange}{HTML}{F6C77D}
\tikzset{
 graphline/.style={draw=black,line width=1.15pt},
 external/.style={draw=stubBlue,line width=1.5pt},
 vertex/.style={circle,fill=graphRed,draw=none,inner sep=0pt,minimum size=2.6mm},
 orange path/.style={draw=paleOrange,line width=7pt,opacity=.70,line cap=round,line join=round},
 terminal label/.style={font=\normalsize,text=black},
 inner label/.style={font=\large,text=graphRed},
}
% Original H coordinates / graph from the prior figures, not a new interpretation.
\def\HCoordinates{%
 \coordinate (TL) at (0,3.50);  \coordinate (TM) at (1.30,3.50);
 \coordinate (TR) at (2.60,3.50); \coordinate (RM) at (2.60,1.75);
 \coordinate (BR) at (2.60,0);    \coordinate (BM) at (1.30,0);
 \coordinate (BL) at (0,0);      \coordinate (LM) at (0,1.75);
 \coordinate (u) at (1.30,2.55); \coordinate (v) at (1.30,0.97);
}
\def\HEdgesAndStubs{%
 % The original eight-cycle and the six interior edges of H.
 \draw[graphline] (TL)--(TM)--(TR)--(RM)--(BR)--(BM)--(BL)--(LM)--cycle;
 \draw[graphline] (TM)--(u) (LM)--(u) (RM)--(u)
                  (LM)--(v) (RM)--(v) (BM)--(v);
 % Exactly TWO blue dangling edges at EACH attachment vertex.
 \draw[external] (TL)--++(-.42,0) (TL)--++(0,.42);
 \draw[external] (TR)--++(.42,0) (TR)--++(0,.42);
 \draw[external] (BR)--++(.42,0) (BR)--++(0,-.42);
 \draw[external] (BL)--++(-.42,0) (BL)--++(0,-.42);
}
\def\HVerticesAndLabels{%
 \foreach \n in {TL,TM,TR,RM,BR,BM,BL,LM,u,v}{\node[vertex] at (\n) {};}
 \node[terminal label] at (-.35,4.04) {$1$};
 \node[terminal label] at (2.95,4.04) {$2$};
 \node[terminal label] at (2.95,-.58) {$3$};
 \node[terminal label] at (-.35,-.58) {$4$};
 \node[inner label] at (-.37,1.75) {$a$};
 \node[inner label] at (.95,2.65) {$b$};
 \node[inner label] at (2.98,1.75) {$c$};
 \node[inner label] at (1.72,.97) {$d$};
}
% First highlighted decomposition.
\begin{scope}[shift={(0,0)}]
 \HCoordinates
 \draw[orange path] (TL)--(LM)--(BL)--(BM)--(v);
 \draw[orange path] (u)--(TM)--(TR)--(RM)--(BR);
 \HEdgesAndStubs
 \HVerticesAndLabels
\end{scope}
% Second highlighted decomposition; same graph and placement.
\begin{scope}[shift={(5.0,0)}]
 \HCoordinates
 \draw[orange path] (TL)--(TM)--(u)--(RM);
 \draw[orange path] (LM)--(v)--(BM)--(BR);
 \HEdgesAndStubs
 \HVerticesAndLabels
\end{scope}
\end{tikzpicture}%
}
\caption{The existence of the requisite Hamiltonian paths follow from earlier analysis of the exterior together with these cycle fragments.}
        \label{fig:Thom_Tsai}
\end{figure}

  %  \begin{figure}
  %      \centering
  %      \includegraphics[width=0.5\linewidth]{Fig31-Thom_fragments_Tsai.jpg}
  %      \caption{The existence of the requisite Hamiltonian paths follow from earlier analysis of the exterior together with these cycle fragments.}
  %      \label{fig:Thom_Tsai}
  %  \end{figure}
\end{proof}

\subsection{Corollaries}
One value of hypohamiltonian graphs is in their relationship with hypotraceable graphs. In this subsection we record a few corollaries of our work, following \cite{Tsai25}. 

\begin{thm} \cite[Theorem 4.4]{JMOPZ17}
    Let $G_i = (V_i, E_i)$ for $i = 1, \dots, 4$ be four planar hypohamiltonian graphs. Then there is a planar hypotraceable graph of order $|V_1| + |V_2| + |V_3| + |V_4| - 6$.
\end{thm}

\begin{cor} \label{cor:hypotrace}
    There exists a hypotraceable graph of order $n$ for $n = 130$ and  all $n \geq 132$.
\end{cor}

\begin{proof}
    Letting each $G_i$ be Tsai's order 34 graph gives a hypotraceable graph of order 130. 

    Since there exists a hypohamiltonian of order $m$ for all $m \geq 36$ one can replace $G_1$ with a graph of order $m$ to get a hypotraceable graph of order $130 + (m - 34)$.
\end{proof}

A graph $G$ is \emph{almost hypohamiltonian} if there exists a vertex $w \in V$ with $G$ hypohamiltonian except at $\{ w \}$ and $G-w$ not Hamiltonian. 

\begin{cor} \label{cor:almost}
    There exists an almost hypohamiltonian planar graph of order $n$ for all $n = 31$ and $n \geq 33$.
\end{cor}

\begin{proof}
    By Thomassen's plus four construction, it suffices to exhibit such graphs of order 31 (and hence 35), 33, 34, and 36. Notice that the proof that the plus two construction preserves the hypohamiltonian property also shows that it preserves the almost hypohamiltonian property, provided the exceptional vertex lies in the exterior of $H$. Since this is the case in Wiener's almost hypohamiltonian planar graph of order 31, we can modify it to one of order 33. 
    
    So it remains to find an almost hypohamiltonian planar graph of order 34, containing $H$, with exceptional vertex in the exterior of $H$. 
    
    And indeed one can check that the graph of order 34 in Figure \ref{fig:almost34} is almost hypohamiltonian with exceptional vertex 34. 

    \begin{figure}
    \centering
    \includegraphics[width=0.4\linewidth]{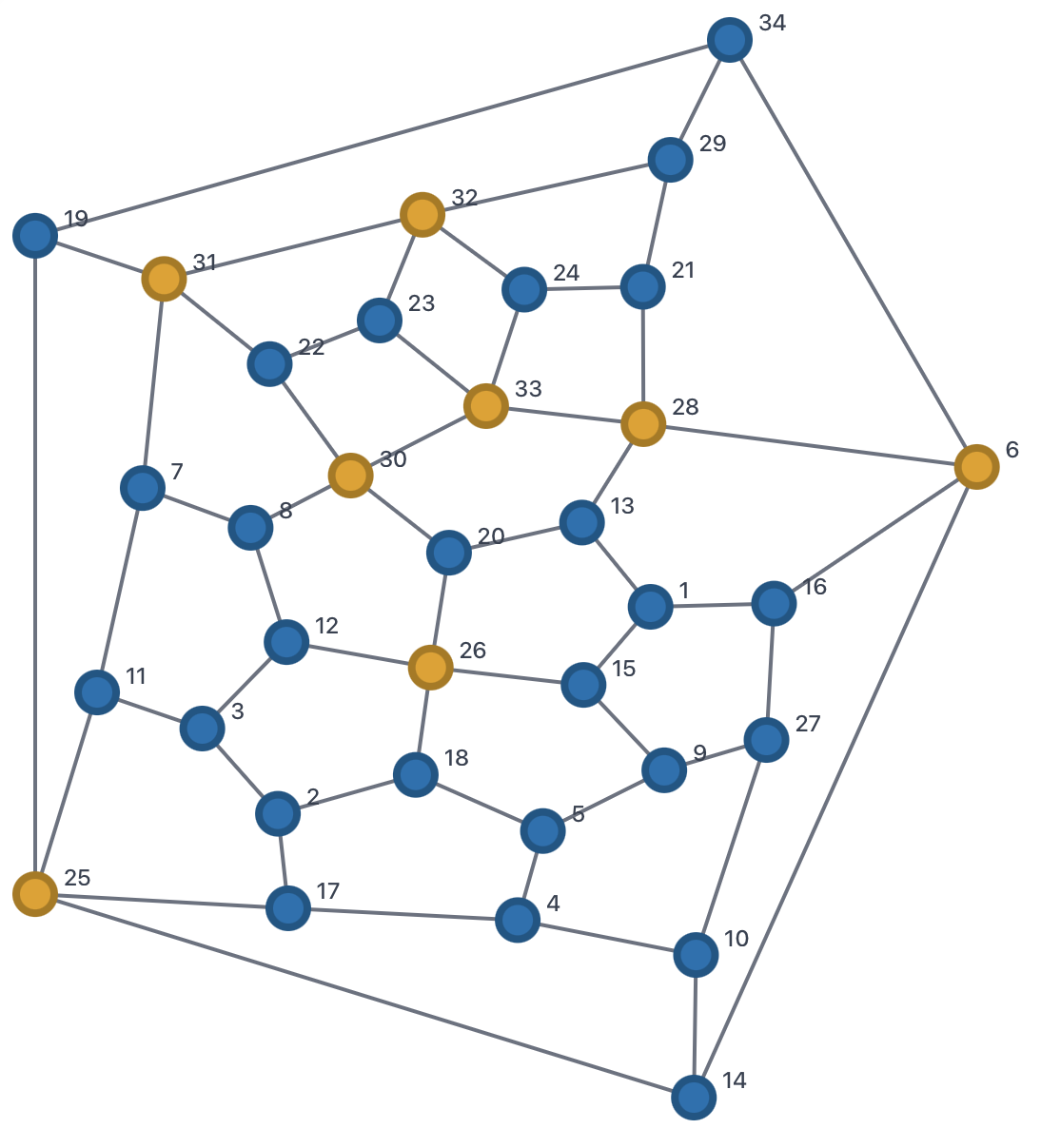}
    \caption{An almost hypohamiltonian graph of order 34, discovered by applying the minus three construction to one of Tsai's hypohamiltonian graphs of order 37.}
    \label{fig:almost34}
\end{figure}
    
\end{proof}

\section{Grinberg's Condition}

%A Hamiltonian cycle $C$ in $G$ imposes the following connectivity condition on $G$:
%\begin{quote}
%    For all $n$, removing $n$ vertices from $G$ can create at most $n-1$ components. 
%\end{quote}
%This condition is known as toughness (or 1-tough) and is written as $t(G) = 1$. Toughness for $n=2$ implies that the minimal degree of vertex of $G$ is at least two. Toughness for $n = \lfloor |V|/2 \rfloor$ implies that any subset $W \subset V$ of size greater than $|V|/2$ must contain an edge. 

%several restrictions on the connectivity of $G$. For example, if $G$ is Hamiltonian then:
%\begin{itemize}
%    \item[(a)] (minimum degree) the minimum degree of a vertex $v \in V$ is at least two ($\delta(G) \geq 2$)
%    \item[(b)] (toughness) removing $n$ vertices can create at most $n-1$ components ($t(G) \geq 1$)
   % \item[(c)] if $G$ is bipartite, then $|V|$ is even
%    \item[(c)] (independence) any subset $W \subset V$ of size $|W| > |V|/2$ must contain an edge ($\alpha(G) \leq |V|/2$)
%\end{itemize}

%For example, the Goldner--Harary graph in Figure \ref{fig:GoldnerHarary} is not Hamiltonian because it is not tough: the graph has 11 vertices with no edges between the 6 red ones

%\begin{figure}
%    \centering
%    \includegraphics[width=0.5\linewidth]{IndependentSet.jpeg}
%    \caption{The Goldner--Haray graph is not sufficiently connected to have a Hamiltonian cycle.}
%    \label{fig:GoldnerHarary}
%\end{figure}

%Toughness is a necessary but insufficient condition. Indeed, any hypohamiltonian graph is tough, since the removal of any vertex yields a Hamiltonian graph. 

Grinberg's Theorem gives a restriction on a hypothetical Hamiltonian cycle in a planar graph $G$.

\begin{thm}[Grinberg \cite{G68}]
Suppose $G$ is a planar graph with a Hamiltonian cycle $C$. Let $g_n$ denote the number of $n$-gons inside of $C$ and let $h_n$ denote the number of $n$-gons outside of $C$\footnote{including the external, unbounded face.}. Then
\begin{equation} \label{eq:Grinberg}
\sum_{n \geq 3} (n-2)(g_n - h_n) = 0.
\end{equation}
\end{thm}

One can view Grinberg's theorem as giving a necessary (but insufficient) condition for a graph to be Hamiltonian, as we now explain. 

\begin{defn}
    Let $G$ be a planar graph with $f_n$ $n$-gons. We say that $G$ is \emph{numerically Hamiltonian} if there exists $g_n, h_n \in \bN$ satisfying $g_n + h_n = f_n$ and Grinberg's equation
    \[
    \sum_{n \geq 3} (n-2)(g_n - h_n) = 0.
    \]
    Moreover, we say that $G$ lies in the \emph{Grinberg gap} if $G$ is numerically Hamiltonian but not Hamiltonian. 
\end{defn}

We are unaware of a systematic study of graphs in the Grinberg gap. The majority of the work in this direction comes from the hypohamiltonian community in light of the remarkable paper \cite{JMOPZ17}, proving that hypohamiltonian graphs of order less than 42 necessarily lie in the Grinberg gap. 

\begin{rem}
    A second necessary yet insufficient condition for a graph $G$ to be Hamiltonian is that the removal of $n$ vertices in $G$ separates $G$ into at most $n$ components. Such a graph is called \emph{tough} and we are primarily interested in tough graphs in the Grinberg gap. Hypohamiltonian graphs $G$ are necessarily tough, since $G-v$ is Hamiltonian for every vertex $v$.
\end{rem}

\subsection{Local constraints}

Certain graphs in Grinberg gap are sufficiently connected and with admissible face datum for a Hamiltonian cycle, yet their configuration of faces impose local restrictions preventing one. 

In more detail, we consider local constraints on Hamiltonian cycles that produce inequalities symmetrically on $h_n$ and $g_n$. For example, Figure \ref{fig:3config} contains a configuration of three 4-gons.

% Complete combined figure 33+34. Needs only TikZ and graphicx.
% The right-hand graph is EXACTLY H, rotated 90 degrees CLOCKWISE.
% Its vertex set, eight-cycle, and six internal edges are transported from
% GCoordsH / GEdgesH via (x,y) -> (y, 2.6-x) and uniform scaling 1.15.
% Thus the left inner vertex is adjacent to LM,TM,BM;
% the right inner vertex is adjacent to RM,TM,BM, with NO diamond on four additional internal vertices.
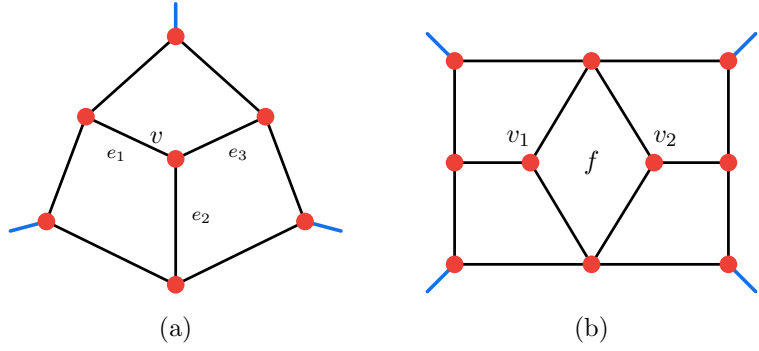
\begin{figure}[tbp]
\centering
\resizebox{.6\linewidth}{!}{%
\begin{tikzpicture}[line cap=round,line join=round]
\definecolor{graphRed}{HTML}{ED4138}
\definecolor{stubBlue}{HTML}{1670EA}
\tikzset{
 graphline/.style={draw=black,line width=1.15pt},
 external/.style={draw=stubBlue,line width=1.5pt},
 vertex/.style={circle,fill=graphRed,draw=none,inner sep=0pt,minimum size=2.6mm},
 vertex label/.style={font=\normalsize,text=black},
}
% Panel (a): original three-edge configuration, unchanged.
\begin{scope}[shift={(0,0)}]
 \coordinate (T) at (2.000,3.650);
 \coordinate (UL) at (0.680,2.470);
 \coordinate (UR) at (3.320,2.470);
 \coordinate (LL) at (0.100,0.930);
 \coordinate (LR) at (3.900,0.930);
 \coordinate (B) at (2.000,0.000);
 \coordinate (v1) at (2.000,1.850);
 \draw[graphline] (T)--(UL)--(LL)--(B)--(LR)--(UR)--cycle;
 \draw[graphline] (UL)--(v1)--(UR) (v1)--(B);
 \draw[external] (T)--++(0,.48) (LL)--++(-.53,-.14) (LR)--++(.53,-.14);
 \foreach \n in {T,UL,UR,LL,LR,B,v1}{\node[vertex] at (\n) {};}
 \node[vertex label] at (1.72,2.17) {$v$};
 \node[font=\scriptsize] at (1.13,1.90) {$e_1$};
 \node[font=\scriptsize] at (2.38,.99) {$e_2$};
 \node[font=\scriptsize] at (2.92,1.92) {$e_3$};
 \node[font=\normalsize] at (2,-.65) {(a)};
\end{scope}
% Panel (b): H ROTATED CLOCKWISE THROUGH EXACTLY 90 DEGREES.
% Coordinates before translation are the images of the old H coordinates
% (TL,TM,TR,RM,BR,BM,BL,LM,u,v) under (x,y) |-> (y,2.6-x),
% multiplied by 1.15. The old interior v maps to v1 at the left,
% while the old interior u maps to v2 at the right.
\begin{scope}[shift={(6.1,0)}]
 \coordinate (TL) at (0,3.29);
 \coordinate (TM) at (2.0125,3.29);
 \coordinate (TR) at (4.025,3.29);
 \coordinate (RM) at (4.025,1.795);
 \coordinate (BR) at (4.025,0.30);
 \coordinate (BM) at (2.0125,0.30);
 \coordinate (BL) at (0,0.30);
 \coordinate (LM) at (0,1.795);
 \coordinate (v1) at (1.1155,1.795); % image of original v
 \coordinate (v2) at (2.9325,1.795); % image of original u
 % Rotated outer eight-cycle; no shortcuts or missing midpoints.
 \draw[graphline] (TL)--(TM)--(TR)--(RM)--(BR)--(BM)--(BL)--(LM)--cycle;
 % Six exact rotated H edges: TM-v1, LM-v1, BM-v1;
 %                          TM-v2, RM-v2, BM-v2.
 \draw[graphline] (TM)--(v1) (LM)--(v1) (BM)--(v1)
                  (TM)--(v2) (RM)--(v2) (BM)--(v2);
 % Same four outward-facing stubs as the handwritten picture.
 \draw[external] (TL)--++(-.40,.40);
 \draw[external] (TR)--++(.40,.40);
 \draw[external] (BR)--++(.40,-.40);
 \draw[external] (BL)--++(-.40,-.40);
 \foreach \n in {TL,TM,TR,RM,BR,BM,BL,LM,v1,v2}{\node[vertex] at (\n) {};}
 \node[vertex label] at (.94,2.15) {$v_1$};
 \node[vertex label] at (3.10,2.15) {$v_2$};
 \node[font=\normalsize] at (2.0125,1.795) {$f$};
 \node[font=\normalsize] at (2.0125,-.65) {(b)};
\end{scope}
\end{tikzpicture}%
}
 \caption{On the left is a local configuration of three 4-gons and on the right is a local configuration of five 4-gons.}
\label{fig:3config}
\end{figure}

%\begin{figure}
%    \centering
%    \includegraphics[width=0.3\linewidth]{Fig33-3config2}
%    \caption{A local configuration of three 4-gons}
%    \label{fig:3config}
%\end{figure}
 
Assume this configuration is completed to a graph with no additional 4-gons. Then $g_4$ and $h_4$ satisfying $g_4 + h_4 = f_4 = 3$. After exchanging the inside and outside, we may assume that $g_4 \leq h_4$ giving two possibilities $(g_4, h_4) = (0, 3)$ or $(1, 2)$. The former configuration cannot occur, as we now explain. 

Any Hamiltonian cycle must reach the central vertex $v$ and therefore use two of the three edges among $\{ e_1, e_2, e_3 \}$. In particular, the Hamiltonian cycle must separate two 4-gons, so $g_4 \neq 0$. 

Now assume the completed graph $G$ has exclusively pentagons outside of this configuration. Then Grinberg's theorem requires
\[
0 = 2(g_4-h_4) + 3(g_5-h_5) = -2 + 3(g_5-h_5).  
\]
which has no solutions since modulo 3 it reads $0 = 1$.

Yet, if there are an even number of pentagons, $f_5 = 2n$, then without the local constraint $g_4, h_4 \geq 1$, one could take $(g_4, g_5) = (3, n-1)$ and $(h_4, h_5) = (0, n+1)$ to avoid this obstruction modulo 3. In other words, Grinberg's theorem is susceptible to solutions $\{ g_n, h_n \}$ that are infeasible geometrically due to local restrictions. 

\begin{rem}
    The graph $G$ cannot be hypohamiltonian, since the removal of $v$ gives a single hexagon and remaining faces pentagons, which violates Grinberg's theorem modulo 3. Conversely, the configuration does appear in an almost hypohamiltonian graph with exceptional vertex $v$, see Figure 7 in \cite{Z15}.    
\end{rem}

More generally, if a graph has $m$ vertices of degree 3 with all neighboring faces distinct $n$-gons, then $g_n, h_n \geq m$.\\ 
\\
The distinctness of the $n$-gons can be weakened for certain configurations. For example, notice that the graph on the right of Figure \ref{fig:3config} has two cubic vertices $v_1$ and $v_2$. They share the central face $f$, but a Hamiltonian cycle cannot contain $f$ without also containing a neighboring face. Therefore, if this local configuration is completed to a graph $G$ with a Hamiltonian cycle then $g_4, h_4 \geq 2$. Further, if the remainder of $G$ consists of pentagons then Grinberg's theorem modulo 3 reads $2(g_4 - h_4) \equiv 0$ and hence if $g_4 \leq h_4$ then one needs $g_4 = 1$ for a Hamiltonian cycle, a contradiction. 

Figure \ref{fig:locconfig} shows some small examples of configurations with an odd number of 4-gons, say $2k+1$, where any Hamiltonian cycle must separate them into $k$ and $k+1$.

% Complete figure: requires only \usepackage{tikz} and \usepackage{graphicx}.
% No custom commands or styles are needed in the document preamble.
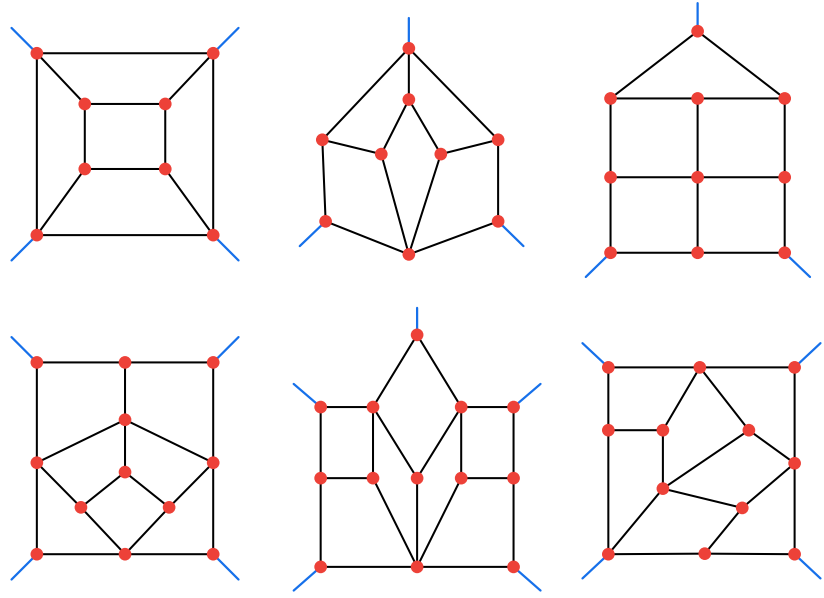
\begin{figure}[tbp]
\centering
\resizebox{.65\linewidth}{!}{%
\begin{tikzpicture}[line cap=round,line join=round]
\definecolor{figRed}{HTML}{ED4138}
\definecolor{figBlue}{HTML}{1670EA}
\definecolor{figGreen}{HTML}{19834D}
\definecolor{figOrange}{HTML}{F6BF65}
\definecolor{figViolet}{HTML}{B023B9}
\definecolor{figLime}{HTML}{55B752}
\tikzset{
 fedge/.style={black,line width=1.05pt},
 fboundary/.style={figBlue,line width=1.15pt},
 fstubs/.style={figBlue,line width=1.15pt},
 fvertex/.style={circle,fill=figRed,inner sep=0pt,minimum size=2.3mm},
 fgreen/.style={fvertex,fill=figGreen},
 fnew/.style={fvertex,fill=figRed},
 fhighlight/.style={figOrange,line width=5.2pt,opacity=.67,line cap=round,line join=round},
 fviolet/.style={figViolet,line width=1.8pt,line cap=round},
 flime/.style={figLime,line width=1.8pt,line cap=round},
}
% Six distinct local planar configurations
% Configuration 1
\begin{scope}[shift={(0.00,5.70)}]
\coordinate (TL) at (0.250,3.750);
\coordinate (TR) at (3.450,3.750);
\coordinate (BL) at (0.250,0.450);
\coordinate (BR) at (3.450,0.450);
\coordinate (p) at (1.120,2.830);
\coordinate (q) at (2.580,2.830);
\coordinate (r) at (1.120,1.650);
\coordinate (s) at (2.580,1.650);
\draw[fedge] (TL)--(TR);
\draw[fedge] (TR)--(BR);
\draw[fedge] (BR)--(BL);
\draw[fedge] (BL)--(TL);
\draw[fedge] (TL)--(p);
\draw[fedge] (TR)--(q);
\draw[fedge] (BL)--(r);
\draw[fedge] (BR)--(s);
\draw[fedge] (p)--(q);
\draw[fedge] (q)--(s);
\draw[fedge] (s)--(r);
\draw[fedge] (r)--(p);
\draw[fstubs] (TL)--++(-0.46,0.46);
\draw[fstubs] (TR)--++(0.46,0.46);
\draw[fstubs] (BL)--++(-0.46,-0.46);
\draw[fstubs] (BR)--++(0.46,-0.46);
\node[fvertex] at (TL) {};
\node[fvertex] at (TR) {};
\node[fvertex] at (BL) {};
\node[fvertex] at (BR) {};
\node[fvertex] at (p) {};
\node[fvertex] at (q) {};
\node[fvertex] at (r) {};
\node[fvertex] at (s) {};
\end{scope}
% Configuration 2
\begin{scope}[shift={(5.20,5.70)}]
\coordinate (T) at (1.800,3.840);
\coordinate (L) at (0.230,2.180);
\coordinate (R) at (3.420,2.180);
\coordinate (LL) at (0.290,0.700);
\coordinate (RR) at (3.420,0.700);
\coordinate (B) at (1.800,0.100);
\coordinate (A) at (1.800,2.910);
\coordinate (C) at (1.300,1.920);
\coordinate (D) at (2.380,1.920);
\draw[fedge] (T)--(L);
\draw[fedge] (T)--(R);
\draw[fedge] (L)--(LL);
\draw[fedge] (R)--(RR);
\draw[fedge] (LL)--(B);
\draw[fedge] (RR)--(B);
\draw[fedge] (T)--(A);
\draw[fedge] (A)--(C);
\draw[fedge] (A)--(D);
\draw[fedge] (L)--(C);
\draw[fedge] (R)--(D);
\draw[fedge] (C)--(B);
\draw[fedge] (D)--(B);
\draw[fstubs] (T)--++(0,0.55);
\draw[fstubs] (LL)--++(-0.47,-0.45);
\draw[fstubs] (RR)--++(0.46,-0.45);
\node[fvertex] at (T) {};
\node[fvertex] at (L) {};
\node[fvertex] at (R) {};
\node[fvertex] at (LL) {};
\node[fvertex] at (RR) {};
\node[fvertex] at (B) {};
\node[fvertex] at (A) {};
\node[fvertex] at (C) {};
\node[fvertex] at (D) {};
\end{scope}
% Configuration 3
\begin{scope}[shift={(10.40,5.70)}]
\coordinate (T) at (1.840,4.150);
\coordinate (a) at (0.260,2.930);
\coordinate (b) at (1.840,2.930);
\coordinate (c) at (3.420,2.930);
\coordinate (d) at (0.260,1.500);
\coordinate (e) at (1.840,1.500);
\coordinate (f) at (3.420,1.500);
\coordinate (g) at (0.260,0.130);
\coordinate (h) at (1.840,0.130);
\coordinate (i) at (3.420,0.130);
\draw[fedge] (T)--(a);
\draw[fedge] (T)--(c);
\draw[fedge] (a)--(b);
\draw[fedge] (b)--(c);
\draw[fedge] (a)--(d);
\draw[fedge] (b)--(e);
\draw[fedge] (c)--(f);
\draw[fedge] (d)--(e);
\draw[fedge] (e)--(f);
\draw[fedge] (d)--(g);
\draw[fedge] (e)--(h);
\draw[fedge] (f)--(i);
\draw[fedge] (g)--(h);
\draw[fedge] (h)--(i);
\draw[fstubs] (T)--++(0,0.51);
\draw[fstubs] (g)--++(-0.45,-0.44);
\draw[fstubs] (i)--++(0.46,-0.44);
\node[fvertex] at (T) {};
\node[fvertex] at (a) {};
\node[fvertex] at (b) {};
\node[fvertex] at (c) {};
\node[fvertex] at (d) {};
\node[fvertex] at (e) {};
\node[fvertex] at (f) {};
\node[fvertex] at (g) {};
\node[fvertex] at (h) {};
\node[fvertex] at (i) {};
\end{scope}
% Configuration 4: H_1, exactly the earlier graph, with one external blue stub at each corner.
\begin{scope}[shift={(0.00,0.00)}]
\coordinate (TL) at (0.250,3.840);
\coordinate (TM) at (1.850,3.840);
\coordinate (TR) at (3.450,3.840);
\coordinate (LM) at (0.250,2.020);
\coordinate (RM) at (3.450,2.020);
\coordinate (BL) at (0.250,0.360);
\coordinate (BM) at (1.850,0.360);
\coordinate (BR) at (3.450,0.360);
\coordinate (u) at (1.850,2.800);
\coordinate (v) at (1.850,1.850);
\coordinate (l) at (1.050,1.210);
\coordinate (r) at (2.650,1.210);
% Outer eight-cycle of H_1.
\draw[fedge] (TL)--(TM)--(TR)--(RM)--(BR)--(BM)--(BL)--(LM)--cycle;
% Ten interior edges of H_1 (not H or H_1').
\draw[fedge] (TM)--(u) (LM)--(u) (RM)--(u) (u)--(v)
   (LM)--(l) (l)--(v) (v)--(r) (r)--(RM)
   (l)--(BM) (r)--(BM);
% Exactly one external stub from each corner.
\draw[fstubs] (TL)--++(-0.46,0.46);
\draw[fstubs] (TR)--++(0.46,0.46);
\draw[fstubs] (BL)--++(-0.46,-0.46);
\draw[fstubs] (BR)--++(0.46,-0.46);
\foreach \p in {TL,TM,TR,LM,RM,BL,BM,BR,u,v,l,r}{\node[fvertex] at (\p) {};}
\end{scope}
% Configuration 5: reflection-symmetric about the vertical line x=1.950.
% Every noncentral coordinate and every edge has a mirror partner.
\begin{scope}[shift={(5.20,0.00)}]
\coordinate (T) at (1.950,4.340);
\coordinate (UL) at (0.200,3.030);
\coordinate (UC) at (1.150,3.030);
\coordinate (UR) at (2.750,3.030);
\coordinate (TR) at (3.700,3.030);
\coordinate (LM) at (0.200,1.740);
\coordinate (L) at (1.150,1.740);
\coordinate (C) at (1.950,1.740);
\coordinate (R) at (2.750,1.740);
\coordinate (RM) at (3.700,1.740);
\coordinate (BL) at (0.200,0.130);
\coordinate (BM) at (1.950,0.130);
\coordinate (BR) at (3.700,0.130);
% Symmetric upper outline and lower boundary.
\draw[fedge] (T)--(UC) (T)--(UR);
\draw[fedge] (UL)--(UC) (UR)--(TR);
\draw[fedge] (UL)--(LM)--(BL)--(BM)--(BR)--(RM)--(TR);
% Paired edges on the two halves and central spoke.
\draw[fedge] (LM)--(L) (RM)--(R);
\draw[fedge] (UC)--(L) (UR)--(R);
\draw[fedge] (UC)--(C) (UR)--(C);
\draw[fedge] (L)--(BM) (R)--(BM) (C)--(BM);
% One blue stub per external terminal; mirror-compatible slopes.
\draw[fstubs] (T)--++(0,0.49);
\draw[fstubs] (UL)--++(-0.49,0.42);
\draw[fstubs] (TR)--++(0.49,0.42);
\draw[fstubs] (BL)--++(-0.49,-0.43);
\draw[fstubs] (BR)--++(0.49,-0.43);
\foreach \p in {T,UL,UC,UR,TR,LM,L,C,R,RM,BL,BM,BR}{\node[fvertex] at (\p) {};}
\end{scope}
% Configuration 6
\begin{scope}[shift={(10.40,0.00)}]
\coordinate (TL) at (0.220,3.750);
\coordinate (TM) at (1.880,3.750);
\coordinate (TR) at (3.600,3.750);
\coordinate (LM) at (0.220,2.610);
\coordinate (p) at (1.210,2.610);
\coordinate (q) at (2.770,2.610);
\coordinate (RM) at (3.600,2.010);
\coordinate (BL) at (0.220,0.360);
\coordinate (r) at (1.210,1.550);
\coordinate (s) at (2.650,1.200);
\coordinate (u) at (1.970,0.370);
\coordinate (BR) at (3.600,0.360);
\draw[fedge] (TL)--(TM);
\draw[fedge] (TM)--(TR);
\draw[fedge] (TL)--(LM);
\draw[fedge] (LM)--(BL);
\draw[fedge] (TR)--(RM);
\draw[fedge] (RM)--(BR);
\draw[fedge] (BL)--(u);
\draw[fedge] (u)--(BR);
\draw[fedge] (LM)--(p);
\draw[fedge] (TM)--(p);
\draw[fedge] (TM)--(q);
\draw[fedge] (q)--(RM);
\draw[fedge] (p)--(r);
\draw[fedge] (r)--(q);
\draw[fedge] (r)--(s);
\draw[fedge] (s)--(RM);
\draw[fedge] (BL)--(r);
\draw[fedge] (s)--(u);
\draw[fstubs] (TL)--++(-0.47,0.44);
\draw[fstubs] (TR)--++(0.47,0.44);
\draw[fstubs] (BL)--++(-0.47,-0.44);
\draw[fstubs] (BR)--++(0.47,-0.44);
\node[fvertex] at (TL) {};
\node[fvertex] at (TM) {};
\node[fvertex] at (TR) {};
\node[fvertex] at (LM) {};
\node[fvertex] at (p) {};
\node[fvertex] at (q) {};
\node[fvertex] at (RM) {};
\node[fvertex] at (BL) {};
\node[fvertex] at (r) {};
\node[fvertex] at (s) {};
\node[fvertex] at (u) {};
\node[fvertex] at (BR) {};
\end{scope}
\end{tikzpicture}%
}
    \caption{Local configurations of five and seven 4-gons}
    \label{fig:locconfig}
\end{figure}

%\begin{figure}
%    \centering
%    \includegraphics[width=0.7\linewidth]{Fig35-locconfig3.jpg}
%    \caption{Local configurations of five and seven 4-gons}
%    \label{fig:locconfig}
%\end{figure}

Note that our local configurations consist of an odd number of 4-gons by design. Since the local restrictions \emph{reduce} the possible size of $|g_4 - h_4|$, if $f_4$ is even, then one can take $g_4 = h_4$ satisfying both Grinberg's equation and the local constraints. 

\subsection{Pentagulations}

Let $C$ be a cycle. A \emph{pre-pentagulation} of $C$ is a planar graph $G$ whose faces consist exclusively of pentagons and $C$. A \emph{pentagulation} of $C$ is a 3-connected pre-pentagulation of $C$. 

More generally, if $H$ is a graph with boundary $C$, then a pentagulation $H$ is a 3-connected graph $H \cup_C G$ where $G$ is a pre-pentagulation of $C$.

\begin{prop}
    Any planar graph that is 3-connected except possibly on one face has a pentagulation. Moreover, one can find a pentagulation that uses prescribed external edges, provided each collection of four neighboring vertices have two prescribed external edges. 
    %the edges are not too sparse for pentagulation. 
\end{prop}

\begin{proof}
    The condition on the prescribed external edges ensures no two are spaced five vertices apart and hence cannot be completed to a pentagon. 
    
    Let $H$ denote a planar graph that is 3-connected, except possibly along a cycle $C_0$.
    
    We add three concentric cycles of vertices: $C_1$, $C_2$, and $C_3$ (and then additional edges) as follows:
\begin{align*}
C_1 &= x_1 x_2 x_3 \cdots x_n x_1 \\
C_2 &= a_1 b_1 a_2 b_2 a_3 \cdots a_n b_n a_1 \\
C_3 &= c_1 c_2 c_3 \cdots c_n c_1 
\end{align*}
with edges 
\begin{align*}
&a_i x_i, \quad i = 1, \dots, n \\
&b_i c_i, \quad i = 1, \dots, n.
\end{align*}

    The edges between $C_0$ and $C_1$ are given by the external edges in $H$. This determines the size $n$. If $n$ is not congruent to 2 modulo 3, then we add at most two pentagons, each increasing the size of $C_1$ by one, as shown in Figure \ref{fig:pentagulation}

    \begin{figure}
    \centering
    \includegraphics[width=0.7\linewidth]{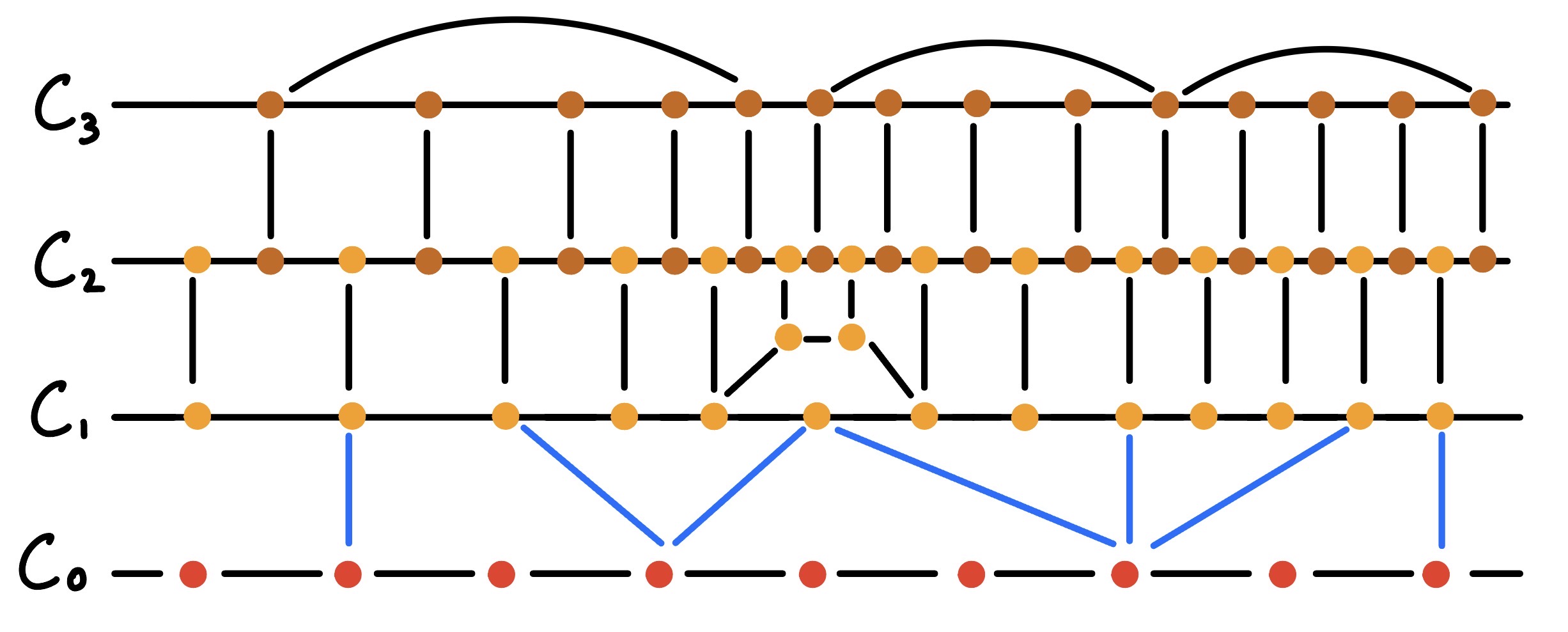}
    \caption{Each row of vertices is the cycle $C_0$, $C_1$, $C_2$, $C_3$. This gives a pentagulation of $H$ with boundary $C_0$, that uses a single pentagon to alter the size of $C_1$.}
    \label{fig:pentagulation}
\end{figure}

    Then we can add edges, e.g. $c_1 c_5$ in $C_3$. Each edge reduces the size of the external boundary by 3. Therefore, we can continue until the external boundary has length five, as desired.

\end{proof}

This proposition gives a recipe for constructing graphs in the Grinberg gap by starting with a restricted local configuration of 4-gons and pentagulating. 

 \subsection{Tutte Graph}

For graphs of small order in the Grinberg gap consisting entirely of 4-gons and 5-gons, we suspect that they arise from a restricted 4-gon configuration pentagulated. 

However, observe that one can add edges to the Tutte graph \cite{Tutte46} until it consists entirely of 4-gons and 5-gons, while preserving the non-Hamiltonian property. 

The key idea, shown in Figure \ref{fig:Tutte_edges}, is to overload the central vertex where the three Tutte fragments meet. Then, as in the original Tutte graph, one can show that any Hamiltonian cycle would necessarily extend a non-existent Hamiltonian path within one fragment. 

\begin{figure}
    \centering
    \includegraphics[width=0.4\linewidth]{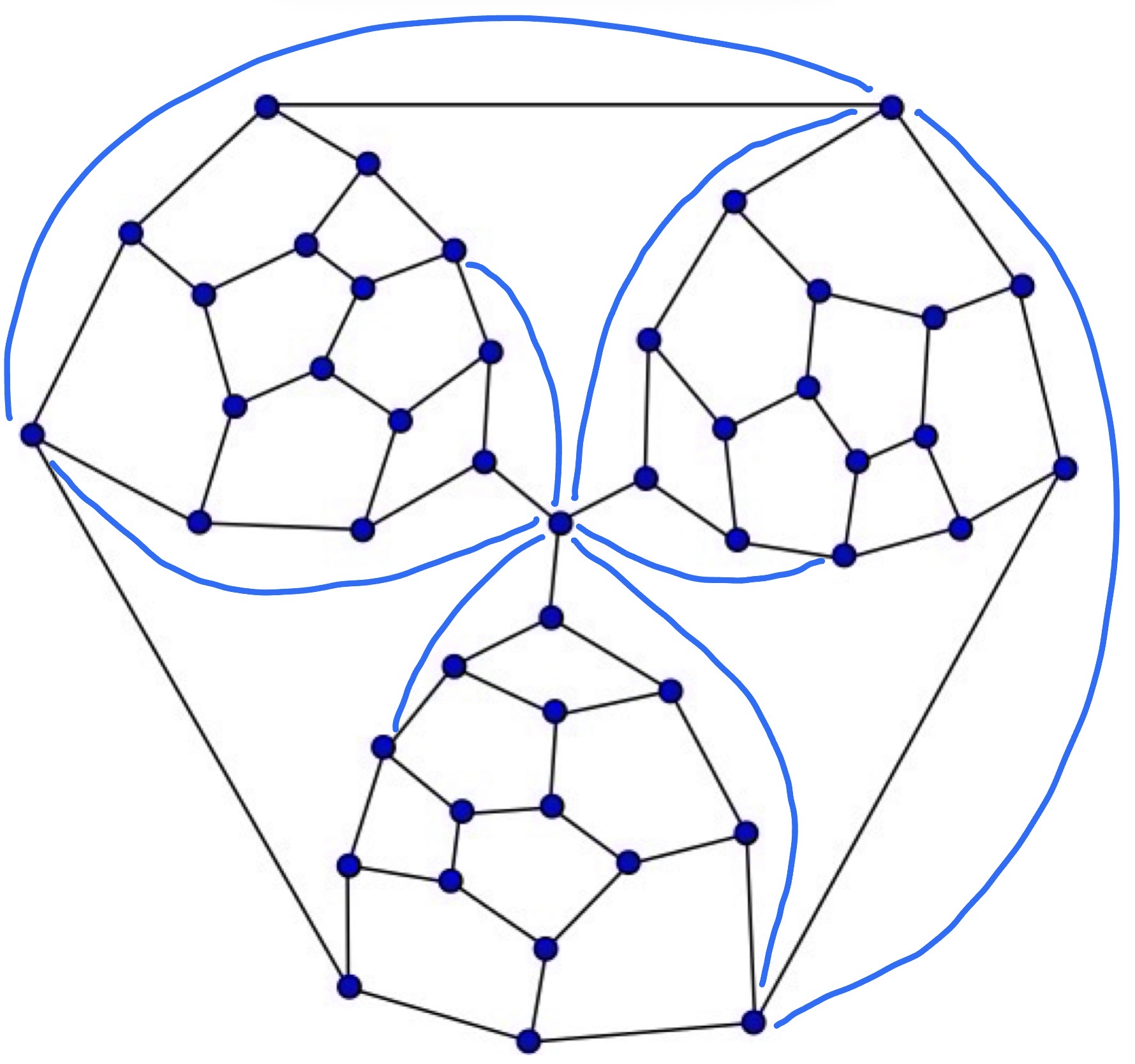}
    \caption{Starting with the Tutte graph and adding six edges to the central vertex and two external edges, one can build a graph consisting entirely of 4-gons and 5-gons without a Hamiltonian cycle.}
    \label{fig:Tutte_edges}
\end{figure}

This example illustrates that the Grinberg gap, even for graphs with solely 4-gons and 5-gons, is more subtle than restrictive 4-gon configurations. 

%Moved to a remark above
%As a final remark, Grinberg's theorem can be used to simplify the cycle fragment analysis in our plus two construction. Namely, of the fragments in $H$ in Figure \ref{fig:H1-HypoGreen}, only $1-3$ and $2-4$ satisfy Grinberg's theorem modulo 3, if the exterior of $H$ consists of pentagons. 

\section{Quadrilaterals and Pentagons}

%Grinberg's theorem is useful when a planar graph consists solely of 4-gons and 5-gons, since
%\[
%2(g_4 - h_4) + 3(g_5-h_5) = 0
%\]
%implies that $g_4-h_4$ is a multiple of 3 and $g_5-h_5$ is even. In particular, $f_5 = g_5 + h_5 = (g_5 - h_5) + 2h_5$ is even. 

Let $G$ be a graph consisting exclusively of 4-gons and 5-gons. So $f= f_4 + f_5$. By double counting the edges, Euler's formula gives 
\[
2e = 2f + 2v - 4 = 2f_4 + 2f_5 + 2v - 4
\]
while counting per face gives
\[
2e = 4 f_4 + 5 f_5.
\]
These equations together yield the following relationship among the triple $(f_4, f_5, v)$  
\[
2 f_4 + 3 f_5 = 2v - 4.
\]

In particular, solving for $f_5$ gives the whole (even) number
\[
f_5 = \frac{2v-2f_4-4}{3}.
\]
Hence modulo 3, $v -f_4 - 1 \equiv 0$ or $v -1 \equiv f_4$. 

Picking small odd values for $f_4$: 3, 5, 7, 9 we can determine the possible values for $v$ between 26 and 40, and compute the excess degree (exc. deg.), i.e., the total degree above 3 edges per vertex.

\[
\begin{tabular}{c|c|c | c | c | c}
   $f_4$ & $v$ & $f_5$ & $f$ & $e$ & exc deg.  \\
 \hline
\multirow{5}{*}{3}  & 26 & 14 & 17 & 41 & 4 \\
                    & 29 & 16 & 19 & 46 & 5 \\
                    & 32 & 18 & 21 & 51 & 6 \\
                    & 35 & 20 & 23 & 56 & 7 \\
                    & 38 & 22 & 25 & 61 & 8 \\
\hline
\multirow{5}{*}{5}  & 28 & 14 & 19 & 45 & 6 \\
                    & 31 & 16 & 21 & 50 & 7 \\
                    & 34 & 18 & 23 & 55 & 8 \\
                    & 37 & 20 & 25 & 60 & 9 \\
                    & 40 & 22 & 27 & 65 & 10 \\
\hline
\multirow{5}{*}{7}  & 27 & 12 & 19 & 44 & 7 \\
                    & 30 & 14 & 21 & 49 & 8 \\
                    & 33 & 16 & 23 & 54 & 9 \\
                    & 36 & 18 & 25 & 59 & 10 \\
                    & 39 & 20 & 27 & 64 & 11 \\
\hline
\multirow{5}{*}{9}  & 26 & 10 & 19 & 43 & 8 \\
                    & 29 & 12 & 21 & 48 & 9 \\
                    & 32 & 14 & 23 & 53 & 10 \\
                    & 35 & 16 & 25 & 58 & 11\\
                    & 38 & 18 & 27 & 63 & 12 \\
\end{tabular}
\]

Using the table, we establish the following:

\begin{prop} \label{prop:Tsai_min}
    Let $G$ be a planar hypohamiltonian graph built from the five 4-gon local configuration $H$ in Figure \ref{fig:local_config}, by pentagulation. Then $|G| \geq 34$ and $|G| \equiv 1 \text{ mod } 3$. 
\end{prop}

\begin{proof}
    Notice that the configuration has six degree 4 vertices, and hence $G$ has excess degree at least 6. It follows from the table above that $|G| \geq 28$. 

    Since $f_4 = 5$, $|G|-1 \equiv f_4 = 5 \equiv 2$ modulo 3, so $|G| \equiv 1$ modulo 3. 

    Plantri \cite{BM07} allowed us to perform an exhaustive search of graphs $G$ satisfying the following conditions: (I) $G$ is a pentagulation of $H$, (II) all external vertices in $H$ have degree 4, (III) $|G| = 28$ or $31$. The search returned only Wiener's almost hypohamiltonian example of order 31.
\end{proof}

\begin{rem}
    Suppose the prescribed edges in $H$ are modified so that $H$ has a cubic attaching vertex $v$. Let $w$ denote its unique neighbor in $G-H$. Then notice that $G-w$ is not Hamiltonian, so $G$ is not hypohamiltonian. 
    
    Indeed, assume that $G-w$ is Hamiltonian and fix a Hamiltonian cycle. Removing $w$ creates a $(3\deg(w)-2)$-gon, which we may assume lies outside the cycle. Then Grinberg's formula gives
    \[
    (3\deg(w) -2)(-1) + (4-2)(g_4-h_4) + (5-2)(g_5 - h_5) = 0
    \]
    which modulo 3 implies that $g_4 - h_4 \equiv 2$ modulo 3. 

    But the only option is $(g_4, h_4) = (3, 2)$ so $g_4 - h_4 = 1$, contradicting the existence of a Hamiltonian cycle. 
\end{rem}

This result explains why Tsai's examples have order 34 and 37, as their orders must be congruent to 1 modulo 3. 

\subsection{Conjecture}

Based on the paper \cite{JMOPZ17} by Jooyandeh, McKay, \"{O}stergard, Pettersson, and Zamfirescu, with a few heuristic arguments, we conjecture the following:

 \begin{conj} \label{conj:main}
    The smallest planar hypohamiltonian graph has order 34.    
\end{conj}

First we note that the smallest known lower bound for the order of a planar hypohamiltonian graph is due to Goedgebeur and Zamfirescu. They developed and implemented an algorithm to search for low order planar hypohamiltonian graphs \cite{GZ17}. They prove that the smallest such graphs have order at least 23. Moreover, if the graph has girth 4 (i.e., a 4-gon and no 3-gons) then it has order at least 27. And if the graph has girth 5, then it has order at least 45 by \cite[Theorem 4.6]{JMOPZ17}.\\
\\
Let $G$ be a planar hypohamiltonian graph of order at most 34. \\
\\
\noindent \underline{Heuristic 1}: $G$ does not contain a triangle.  \\
\\
The smallest known hypohamiltonian planar graph containing a triangle has order 216, utilizing a construction of Thomassen combined with Wiener's almost hypohamiltonian graph of order 31. While this is rather surprising at first glance, a closer inspection reveals a stringent condition. If $G$ is hypohamiltonian with $v$ a vertex of a triangle. Then the Hamiltonian cycle for $G - v$ cannot use the remaining edge of the triangle. If it did, one could detour around the triangle to obtain a Hamiltonian cycle for $G$, a contradiction. In particular, every vertex of such a triangle must have degree at least four. 

By \cite[Theorem 4.6]{JMOPZ17}, the girth of $G$ is at most four, meaning if we assume Heuristic 1, then $G$ contains a 4-gon. 

Further, in \cite[Theorem 2.3, Theorem 2.4, Corollary 3.7]{JMOPZ17}, the authors prove that $G$ lies in the Grinberg gap (i.e., is non-Hamiltonian but has a solution to Grinberg's equation). In particular, $G$ cannot contain a single 4-gon with the remaining faces of size congruent to 2 modulo 3.\\
\\
\noindent  \underline{Heuristic 2}: After adding edges, we can assume $G$ consists of 4-gons or 5-gons. \\
\\
For instance, a 6-gon can be divided into two 4-gons and a 7-gon into a 4-gon and a 5-gon. If $G$ is hypohamiltonian, then adding edges to $G$ can only destroy the hypohamiltonian condition if the graph becomes Hamiltonian. So adding the new edges yields a graph in Grinberg gap, then it remains hypohamiltonian. 

We are not conjecturing that maximal edge hypohamiltonian graphs contain exclusively 4 and 5-gons in general. But the more variability in $n$-gons and how they are connected, the more likely Grinberg's theorem will prohibit a Hamiltonian cycle \emph{after} deleting a vertex. This is particularly important for vertex removal on the boundary between 4-gons and 5-gons.  \\
\\
\noindent  \underline{Heuristic 3}: $G$ can be written in the form $G = H \cup_C P$ where $H$ is a 4-gon configuration, $P$ is a 5-gon configuration, and $C$ is the boundary cycle between them. \\
\\
For example, all of the graphs of order 40 in \cite{JMOPZ17} and all the graphs of order less than 40 in \cite{Tsai25} and this paper have this property. Note that the super graph of the Tutte graph above is a counterexample to this heuristic outside of the hypohamiltonian setting and for larger order. Yet the local configurations in Figure \ref{fig:locconfig} appear to give the smallest order 3-connected graphs in the Grinberg gap. 

Notice that the number of 4-gons in $G$, $f_4$, must be odd to prevent the choice $g_4 = h_4$. And we ruled out the case $f_4= 1$ above. \\
%one for $G$ to lie in the Grinberg gap. And $f_4$ cannot be even if $G$ is  
%Notice that the super graph of the Tutte graph above is a counterexample to this heuristic for larger, non-hypohamiltonian graphs. But for hypohamiltonian graphs consisting of 4-gons and 5-gons with at most 34 vertices, 
%is a pentagulation of a local configuration of 4-gons, 
%For $G$ to lie in the Grinberg gap with only 4-gons and 5-gons, we suspect it must be susceptible to local constraints. Therefore, the graph consists of a local configuration of 4-gons, with a pentagulation. 
%Additionally, the number $f_4$ must be odd to prevent the choice $g_4 = h_4$.\\
\\
If $f_4 = 3$ then the configuration in Figure \ref{fig:3config} is the only configuration of 4-gons that yields a graph in the Grinberg gap. However, removing the central vertex $v$ leaves a graph $G - v$ with a single 6-gon and all other faces 5-gons. So $G- v$ is not Hamiltonian by Grinberg's theorem and the original graph $G$ is not hypohamiltonian.\\
\\
If $f_4 = 5$ and $|G| < 34$, then Proposition \ref{prop:Tsai_min} implies the 4-gon configuration cannot be $H$ from Figure \ref{fig:local_config}. By Euler's formula considerations above, the order of $G$ is congruent to 1 modulo 3, hence must be 28 or 31.
The other viable local configurations do not seem to have smaller pentagulations. Further, one cannot apply the plus three construction in reverse (i.e., a minus three construction) in any of these cases. \\
\\
 If $f_4 \geq 7$ then the size of the surrounding 5-gon configuration must shrink. We suspect larger configurations of 4-gons will not lead to smaller order graphs. Note that even the smallest possible 3-connected completion of a cycle by 5-gons will likely contain at least 25 vertices in light of the examples of minimal pentagulations in \cite{K24}. However, one cannot apply those configurations directly since the 3-connected assumption needs to be weakened on the boundary. \\
 \\
 One approach to proving (or disproving) this conjecture would be to computationally check if the conjecture holds for graphs consisting of 4-gons and 5-gons and then analyze each heuristic. Ideally, this paper inspires more work on the subject. 

\section{Appendix: Cycle Fragments}

Here we provide a list of figures containing various cycle fragments:

\begin{figure}[ht]
\centering
\begin{tikzpicture}[
  orange path/.style={draw=orange!48,line width=5.2pt,
                     line cap=round,line join=round}
]

% Compact 5-by-3 arrangement. Every panel uses the same graph.
% Only the two orange paths and the heading change.
\def\SevenCoords{%
  \coordinate (v1) at (0,3.2);
  \coordinate (a)  at (1.15,3.2);
  \coordinate (v2) at (2.30,3.2);
  \coordinate (r)  at (2.30,1.6);
  \coordinate (v3) at (2.30,0);
  \coordinate (b)  at (1.15,0);
  \coordinate (v4) at (0,0);
  \coordinate (l)  at (0,1.6);
  \coordinate (u)  at (1.15,2.45);
  \coordinate (d)  at (1.15,0.85);
  \coordinate (v1u) at (0,3.55);
  \coordinate (v1l) at (-0.35,3.2);
  \coordinate (v2u) at (2.30,3.55);
  \coordinate (v2r) at (2.65,3.2);
  \coordinate (v3r) at (2.65,0);
  \coordinate (v3d) at (2.30,-0.35);
  \coordinate (v4l) at (-0.35,0);
  \coordinate (v4d) at (0,-0.35);
}
\def\SevenBase{%
  \draw[external stub]
     (v1)--(v1u) (v1)--(v1l)
     (v2)--(v2u) (v2)--(v2r)
     (v3)--(v3r) (v3)--(v3d)
     (v4)--(v4l) (v4)--(v4d);
  \draw[graph edge]
     (v1)--(a)--(v2)
     (v1)--(l)--(v4)
     (v4)--(b)--(v3)
     (v2)--(r)--(v3)
     (a)--(u) (u)--(l) (u)--(r)
     (l)--(d) (d)--(r) (d)--(b);
  \foreach \p in {v1,a,v2,r,v3,b,v4,l,u,d}{%
    \node[graph vertex] at (\p) {};
  }
}
% Arguments: x-position, y-position, label, first path, second path.
\def\SevenPanel#1#2#3#4#5{%
  \begin{scope}[shift={(#1,#2)},scale=0.66]
    \SevenCoords
    \draw[orange path] #4;
    \draw[orange path] #5;
    \SevenBase
    \node[font=\scriptsize,text=partpurple!65!black]
      at (1.15,4.06) {$#3$};
  \end{scope}
}

% Row 1
\SevenPanel{0}{7.2}{1\! -\! 2,\;3\! -\! 4}
  {(v1u)--(v1)--(l)--(u)--(a)--(v2)--(v2r)}
  {(v4l)--(v4)--(b)--(d)--(r)--(v3)--(v3d)}
\SevenPanel{2.65}{7.2}{1\! -\! 4,\;2\! -\! 3}
  {(v1u)--(v1)--(l)--(v4)--(v4l)}
  {(v2r)--(v2)--(a)--(u)--(r)--(d)--(b)--(v3)--(v3d)}
\SevenPanel{5.30}{7.2}{1,\;2\! -\! 3}
  {(v1u)--(v1)--(v1l)}
  {(v2r)--(v2)--(a)--(u)--(l)--(v4)--(b)--(d)--(r)--(v3)--(v3d)}
\SevenPanel{7.95}{7.2}{1,\;2\! -\! 4}
  {(v1u)--(v1)--(v1l)}
  {(v2r)--(v2)--(a)--(u)--(r)--(v3)--(b)--(d)--(l)--(v4)--(v4l)}
\SevenPanel{10.60}{7.2}{1,\;3\! -\! 4}
  {(v1u)--(v1)--(v1l)}
  {(v4l)--(v4)--(b)--(d)--(l)--(u)--(a)--(v2)--(r)--(v3)--(v3d)}

% Row 2
\SevenPanel{0}{3.6}{2,\;1\! -\! 3}
  {(v2u)--(v2)--(v2r)}
  {(v1u)--(v1)--(a)--(u)--(l)--(v4)--(b)--(d)--(r)--(v3)--(v3d)}
\SevenPanel{2.65}{3.6}{2,\;1\! -\! 4}
  {(v2u)--(v2)--(v2r)}
  {(v1u)--(v1)--(a)--(u)--(r)--(v3)--(b)--(d)--(l)--(v4)--(v4l)}
\SevenPanel{5.30}{3.6}{2,\;3\! -\! 4}
  {(v2u)--(v2)--(v2r)}
  {(v4l)--(v4)--(l)--(v1)--(a)--(u)--(r)--(d)--(b)--(v3)--(v3d)}
\SevenPanel{7.95}{3.6}{1\! -\! 2,\;3}
  {(v1u)--(v1)--(l)--(v4)--(b)--(d)--(r)--(u)--(a)--(v2)--(v2r)}
  {(v3r)--(v3)--(v3d)}
\SevenPanel{10.60}{3.6}{2\! -\! 4,\;3}
  {(v2r)--(v2)--(a)--(v1)--(l)--(u)--(r)--(d)--(b)--(v4)--(v4l)}
  {(v3r)--(v3)--(v3d)}

% Row 3; centred under the first two rows
\SevenPanel{1.325}{0}{1\! -\! 4,\;3}
  {(v1u)--(v1)--(a)--(v2)--(r)--(u)--(l)--(d)--(b)--(v4)--(v4l)}
  {(v3r)--(v3)--(v3d)}
\SevenPanel{3.975}{0}{1\! -\! 2,\;4}
  {(v1u)--(v1)--(l)--(d)--(b)--(v3)--(r)--(u)--(a)--(v2)--(v2r)}
  {(v4l)--(v4)--(v4d)}
\SevenPanel{6.625}{0}{2\! -\! 3,\;4}
  {(v2r)--(v2)--(a)--(v1)--(l)--(u)--(r)--(d)--(b)--(v3)--(v3d)}
  {(v4l)--(v4)--(v4d)}
\SevenPanel{9.275}{0}{1\! -\! 3,\;4}
  {(v1u)--(v1)--(a)--(v2)--(r)--(u)--(l)--(d)--(b)--(v3)--(v3d)}
  {(v4l)--(v4)--(v4d)}

\end{tikzpicture}
\caption{Cycle fragments in $H$ realize all possible endpoint decompositions. Using symmetries in $H$, it suffices to consider the cycle fragments in the first row.}
\label{fig:H-Ham}
\end{figure}
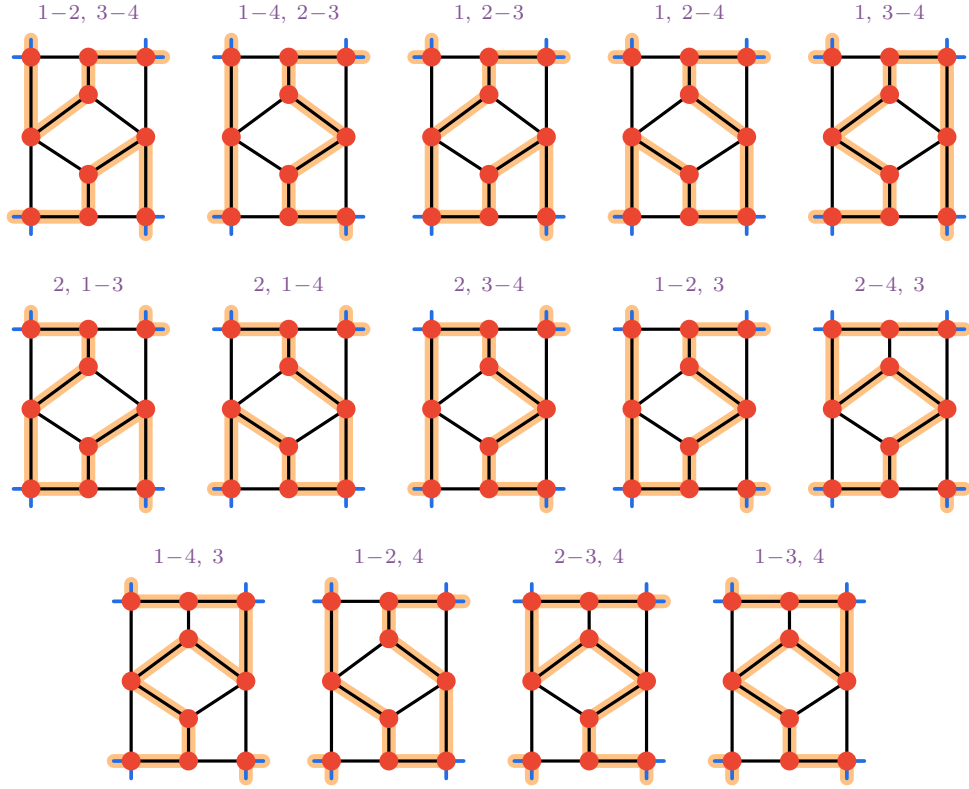

% FIGURE 08 — corrected 4-column x 2-row layout, 8 verified two-path covers
% Requires the shared styles: graph edge, graph vertex, external stub.
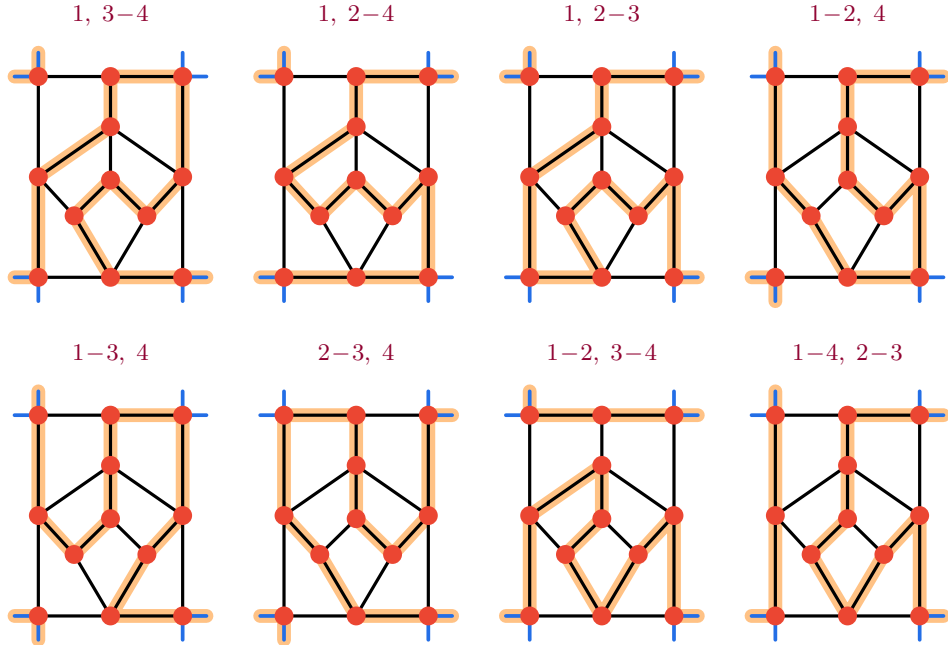
\begin{figure}[ht]
\centering
\begin{tikzpicture}[
  orange path/.style={draw=orange!48,line width=5.2pt,
                      line cap=round,line join=round}
]
% This entire figure is self-contained apart from shared TikZ styles.
\def\EightCoords{%
  \coordinate (TL) at (0,3.2);
  \coordinate (TM) at (1.15,3.2);
  \coordinate (TR) at (2.30,3.2);
  \coordinate (RM) at (2.30,1.6);
  \coordinate (BR) at (2.30,0);
  \coordinate (BM) at (1.15,0);
  \coordinate (BL) at (0,0);
  \coordinate (LM) at (0,1.6);
  \coordinate (u) at (1.15,2.40);
  \coordinate (v) at (1.15,1.55);
  \coordinate (l) at (0.57,0.98);
  \coordinate (r) at (1.73,0.98);
  \coordinate (TLu) at (0,3.58);
  \coordinate (TLl) at (-0.38,3.2);
  \coordinate (TRu) at (2.30,3.58);
  \coordinate (TRr) at (2.68,3.2);
  \coordinate (BRr) at (2.68,0);
  \coordinate (BRd) at (2.30,-0.38);
  \coordinate (BLl) at (-0.38,0);
  \coordinate (BLd) at (0,-0.38);
}
\def\EightBase{%
  \draw[external stub]
     (TL)--(TLu) (TL)--(TLl)
     (TR)--(TRu) (TR)--(TRr)
     (BR)--(BRr) (BR)--(BRd)
     (BL)--(BLl) (BL)--(BLd);
  \draw[graph edge]
     (TL)--(TM)--(TR)--(RM)--(BR)--(BM)--(BL)--(LM)--cycle
     (TM)--(u) (LM)--(u) (RM)--(u) (u)--(v)
     (LM)--(l) (l)--(v) (v)--(r) (r)--(RM)
     (l)--(BM) (r)--(BM);
  \foreach \q in {TL,TM,TR,RM,BR,BM,BL,LM,u,v,l,r}{%
    \node[graph vertex] at (\q) {};
  }
}
% Coordinates are shared; highlights are drawn behind edges/vertices.
\def\EightPanel#1#2#3#4#5{%
  \begin{scope}[shift={(#1,#2)},scale=0.83]
    \EightCoords
    \draw[orange path] #3;
    \draw[orange path] #4;
    \EightBase
    \node[font=\footnotesize,text=purple!75!black,inner sep=0pt] at (1.15,4.17) {$#5$};
  \end{scope}
}

% Panel 1: 1,\;3\! -\!4
\EightPanel{0.00}{4.48}
  {(TLu)--(TL)--(TLl)}
  {(BLl)--(BL)--(LM)--(u)--(TM)--(TR)--(RM)--(r)--(v)--(l)--(BM)--(BR)--(BRr)}
  {1,\;3\! -\!4}

% Panel 2: 1,\;2\! -\!4
\EightPanel{3.25}{4.48}
  {(TLu)--(TL)--(TLl)}
  {(BLl)--(BL)--(BM)--(BR)--(RM)--(r)--(v)--(l)--(LM)--(u)--(TM)--(TR)--(TRr)}
  {1,\;2\! -\!4}

% Panel 3: 1,\;2\! -\!3
\EightPanel{6.50}{4.48}
  {(TLu)--(TL)--(TLl)}
  {(TRr)--(TR)--(TM)--(u)--(LM)--(BL)--(BM)--(l)--(v)--(r)--(RM)--(BR)--(BRr)}
  {1,\;2\! -\!3}

% Panel 4: 1\! -\!2,\;4
\EightPanel{9.75}{4.48}
  {(BLl)--(BL)--(BLd)}
  {(TLu)--(TL)--(LM)--(l)--(BM)--(BR)--(RM)--(r)--(v)--(u)--(TM)--(TR)--(TRr)}
  {1\! -\!2,\;4}

% Panel 5: 1\! -\!3,\;4
\EightPanel{0.00}{0.00}
  {(BLl)--(BL)--(BLd)}
  {(TLu)--(TL)--(LM)--(l)--(v)--(u)--(TM)--(TR)--(RM)--(r)--(BM)--(BR)--(BRr)}
  {1\! -\!3,\;4}

% Panel 6: 2\! -\!3,\;4
\EightPanel{3.25}{0.00}
  {(BLl)--(BL)--(BLd)}
  {(TRr)--(TR)--(RM)--(r)--(v)--(u)--(TM)--(TL)--(LM)--(l)--(BM)--(BR)--(BRr)}
  {2\! -\!3,\;4}

% Panel 7: 1\! -\!2,\;3\! -\!4
\EightPanel{6.50}{0.00}
  {(TLu)--(TL)--(TM)--(TR)--(TRr)}
  {(BLl)--(BL)--(LM)--(u)--(v)--(l)--(BM)--(r)--(RM)--(BR)--(BRr)}
  {1\! -\!2,\;3\! -\!4}

% Panel 8: 1\! -\!4,\;2\! -\!3
\EightPanel{9.75}{0.00}
  {(TLu)--(TL)--(LM)--(BL)--(BLl)}
  {(TRr)--(TR)--(TM)--(u)--(v)--(l)--(BM)--(r)--(RM)--(BR)--(BRr)}
  {1\! -\!4,\;2\! -\!3}

\end{tikzpicture}
  \caption{For each fragment in $H$, there is one in $H_1$ with the same endpoints and vice versa.} 
    \label{fig:H1-Ham}
\end{figure}

% Figure 09 — complete paste-ready figure.
% Shared styles and graph-coordinate macros from previous figures are used.
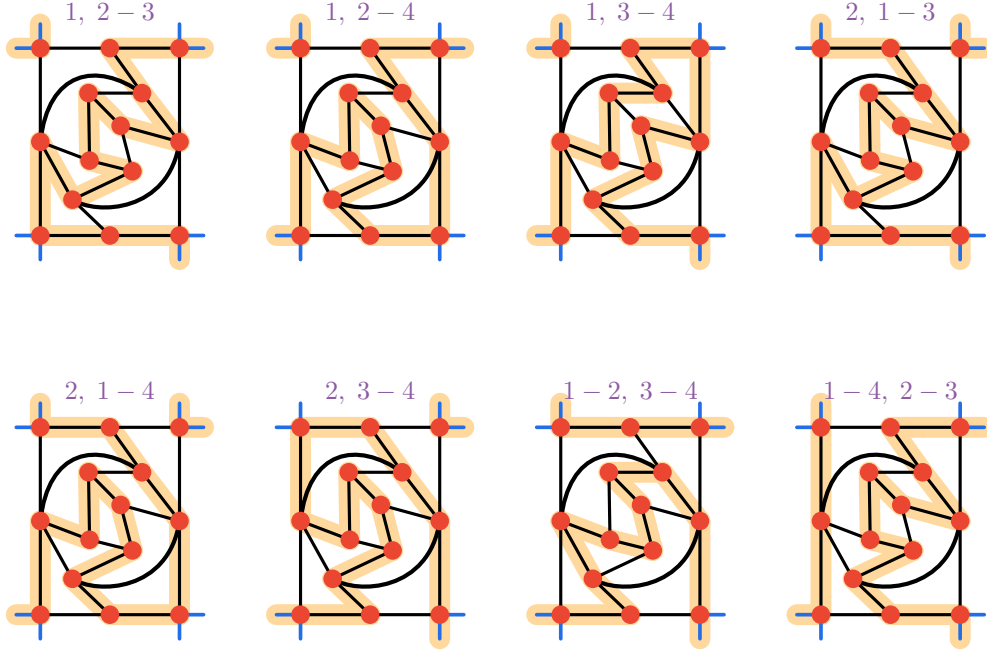
\begin{figure}[p]
\centering
\resizebox{0.8\linewidth}{!}{%
\begin{tikzpicture}[orange highlight/.style={draw=pathorange,line width=8pt,line cap=round,line join=round}]
% Use TWO special curved edges in H_2; neither appears in orange.
\def\DrawHtwoTwoCurves{%
  \GOuterEdges
  \draw[graph edge] (TM)--(b) (b)--(RM) (a)--(b)
    (a)--(c) (a)--(d) (c)--(RM) (c)--(e)
    (d)--(e) (LM)--(d) (LM)--(f) (f)--(BM) (f)--(e);
  \draw[graph edge,line width=1.6pt] (LM) .. controls (.15,3.08) and (1.05,3.26) .. (b);
  \draw[graph edge,line width=1.6pt] (f) .. controls (1.05,.38) and (2.40,.42) .. (RM);
}
  % Panel 1: 1,\;2-3
  \begin{scope}[shift={(0.00,5.10)},scale=0.72]
    \GCoordsHtwo
    \draw[orange highlight] (TL)--++(-0.45,0);
    \draw[orange highlight] (TL)--++(0,0.45);
    \draw[orange highlight] (TR)--(TM);
    \draw[orange highlight] (TM)--(b);
    \draw[orange highlight] (b)--(RM);
    \draw[orange highlight] (RM)--(c);
    \draw[orange highlight] (c)--(a);
    \draw[orange highlight] (a)--(d);
    \draw[orange highlight] (d)--(e);
    \draw[orange highlight] (e)--(f);
    \draw[orange highlight] (f)--(LM);
    \draw[orange highlight] (LM)--(BL);
    \draw[orange highlight] (BL)--(BM);
    \draw[orange highlight] (BM)--(BR);
    \draw[orange highlight] (TR)--++(0.45,0);
    \draw[orange highlight] (BR)--++(0,-0.45);
    \DrawHtwoTwoCurves
    \GNodesHtwo
    \node[text=partpurple!70!black,font=\small] at (1.3,4.13) {$ 1,\;2-3 $};
  \end{scope}
  % Panel 2: 1,\;2-4
  \begin{scope}[shift={(3.50,5.10)},scale=0.72]
    \GCoordsHtwo
    \draw[orange highlight] (TL)--++(-0.45,0);
    \draw[orange highlight] (TL)--++(0,0.45);
    \draw[orange highlight] (TR)--(TM);
    \draw[orange highlight] (TM)--(b);
    \draw[orange highlight] (b)--(RM);
    \draw[orange highlight] (RM)--(BR);
    \draw[orange highlight] (BR)--(BM);
    \draw[orange highlight] (BM)--(f);
    \draw[orange highlight] (f)--(e);
    \draw[orange highlight] (e)--(c);
    \draw[orange highlight] (c)--(a);
    \draw[orange highlight] (a)--(d);
    \draw[orange highlight] (d)--(LM);
    \draw[orange highlight] (LM)--(BL);
    \draw[orange highlight] (TR)--++(0.45,0);
    \draw[orange highlight] (BL)--++(-0.45,0);
    \DrawHtwoTwoCurves
    \GNodesHtwo
    \node[text=partpurple!70!black,font=\small] at (1.3,4.13) {$ 1,\;2-4 $};
  \end{scope}
  % Panel 3: 1,\;3-4
  \begin{scope}[shift={(7.00,5.10)},scale=0.72]
    \GCoordsHtwo
    \draw[orange highlight] (TL)--++(-0.45,0);
    \draw[orange highlight] (TL)--++(0,0.45);
    \draw[orange highlight] (BR)--(BM);
    \draw[orange highlight] (BM)--(f);
    \draw[orange highlight] (f)--(e);
    \draw[orange highlight] (e)--(c);
    \draw[orange highlight] (c)--(RM);
    \draw[orange highlight] (RM)--(TR);
    \draw[orange highlight] (TR)--(TM);
    \draw[orange highlight] (TM)--(b);
    \draw[orange highlight] (b)--(a);
    \draw[orange highlight] (a)--(d);
    \draw[orange highlight] (d)--(LM);
    \draw[orange highlight] (LM)--(BL);
    \draw[orange highlight] (BR)--++(0,-0.45);
    \draw[orange highlight] (BL)--++(-0.45,0);
    \DrawHtwoTwoCurves
    \GNodesHtwo
    \node[text=partpurple!70!black,font=\small] at (1.3,4.13) {$ 1,\;3-4 $};
  \end{scope}
  % Panel 4: 2,\;1-3
  \begin{scope}[shift={(10.50,5.10)},scale=0.72]
    \GCoordsHtwo
    \draw[orange highlight] (TR)--++(0.45,0);
    \draw[orange highlight] (TR)--++(0,0.45);
    \draw[orange highlight] (TL)--(TM);
    \draw[orange highlight] (TM)--(b);
    \draw[orange highlight] (b)--(RM);
    \draw[orange highlight] (RM)--(c);
    \draw[orange highlight] (c)--(a);
    \draw[orange highlight] (a)--(d);
    \draw[orange highlight] (d)--(e);
    \draw[orange highlight] (e)--(f);
    \draw[orange highlight] (f)--(LM);
    \draw[orange highlight] (LM)--(BL);
    \draw[orange highlight] (BL)--(BM);
    \draw[orange highlight] (BM)--(BR);
    \draw[orange highlight] (TL)--++(0,0.45);
    \draw[orange highlight] (BR)--++(0,-0.45);
    \DrawHtwoTwoCurves
    \GNodesHtwo
    \node[text=partpurple!70!black,font=\small] at (1.3,4.13) {$ 2,\;1-3 $};
  \end{scope}
  % Panel 5: 2,\;1-4
  \begin{scope}[shift={(0.00,0.00)},scale=0.72]
    \GCoordsHtwo
    \draw[orange highlight] (TR)--++(0.45,0);
    \draw[orange highlight] (TR)--++(0,0.45);
    \draw[orange highlight] (TL)--(TM);
    \draw[orange highlight] (TM)--(b);
    \draw[orange highlight] (b)--(RM);
    \draw[orange highlight] (RM)--(BR);
    \draw[orange highlight] (BR)--(BM);
    \draw[orange highlight] (BM)--(f);
    \draw[orange highlight] (f)--(e);
    \draw[orange highlight] (e)--(c);
    \draw[orange highlight] (c)--(a);
    \draw[orange highlight] (a)--(d);
    \draw[orange highlight] (d)--(LM);
    \draw[orange highlight] (LM)--(BL);
    \draw[orange highlight] (TL)--++(0,0.45);
    \draw[orange highlight] (BL)--++(-0.45,0);
    \DrawHtwoTwoCurves
    \GNodesHtwo
    \node[text=partpurple!70!black,font=\small] at (1.3,4.13) {$ 2,\;1-4 $};
  \end{scope}
  % Panel 6: 2,\;3-4
  \begin{scope}[shift={(3.50,0.00)},scale=0.72]
    \GCoordsHtwo
    \draw[orange highlight] (TR)--++(0.45,0);
    \draw[orange highlight] (TR)--++(0,0.45);
    \draw[orange highlight] (BR)--(RM);
    \draw[orange highlight] (RM)--(b);
    \draw[orange highlight] (b)--(TM);
    \draw[orange highlight] (TM)--(TL);
    \draw[orange highlight] (TL)--(LM);
    \draw[orange highlight] (LM)--(d);
    \draw[orange highlight] (d)--(a);
    \draw[orange highlight] (a)--(c);
    \draw[orange highlight] (c)--(e);
    \draw[orange highlight] (e)--(f);
    \draw[orange highlight] (f)--(BM);
    \draw[orange highlight] (BM)--(BL);
    \draw[orange highlight] (BR)--++(0,-0.45);
    \draw[orange highlight] (BL)--++(-0.45,0);
    \DrawHtwoTwoCurves
    \GNodesHtwo
    \node[text=partpurple!70!black,font=\small] at (1.3,4.13) {$ 2,\;3-4 $};
  \end{scope}
  % Panel 7: 1-2,\;3-4
  \begin{scope}[shift={(7.00,0.00)},scale=0.72]
    \GCoordsHtwo
    \draw[orange highlight] (TL)--(TM);
    \draw[orange highlight] (TM)--(TR);
    \draw[orange highlight] (TL)--++(0,0.45);
    \draw[orange highlight] (TR)--++(0.45,0);
    \draw[orange highlight] (BR)--(RM);
    \draw[orange highlight] (RM)--(b);
    \draw[orange highlight] (b)--(a);
    \draw[orange highlight] (a)--(c);
    \draw[orange highlight] (c)--(e);
    \draw[orange highlight] (e)--(d);
    \draw[orange highlight] (d)--(LM);
    \draw[orange highlight] (LM)--(f);
    \draw[orange highlight] (f)--(BM);
    \draw[orange highlight] (BM)--(BL);
    \draw[orange highlight] (BR)--++(0,-0.45);
    \draw[orange highlight] (BL)--++(-0.45,0);
    \DrawHtwoTwoCurves
    \GNodesHtwo
    \node[text=partpurple!70!black,font=\small] at (1.3,4.13) {$ 1-2,\;3-4 $};
  \end{scope}
  % Panel 8: 1-4,\;2-3
  \begin{scope}[shift={(10.50,0.00)},scale=0.72]
    \GCoordsHtwo
    \draw[orange highlight] (TL)--(LM);
    \draw[orange highlight] (LM)--(BL);
    \draw[orange highlight] (TL)--++(0,0.45);
    \draw[orange highlight] (BL)--++(-0.45,0);
    \draw[orange highlight] (TR)--(TM);
    \draw[orange highlight] (TM)--(b);
    \draw[orange highlight] (b)--(RM);
    \draw[orange highlight] (RM)--(c);
    \draw[orange highlight] (c)--(a);
    \draw[orange highlight] (a)--(d);
    \draw[orange highlight] (d)--(e);
    \draw[orange highlight] (e)--(f);
    \draw[orange highlight] (f)--(BM);
    \draw[orange highlight] (BM)--(BR);
    \draw[orange highlight] (TR)--++(0.45,0);
    \draw[orange highlight] (BR)--++(0,-0.45);
    \DrawHtwoTwoCurves
    \GNodesHtwo
    \node[text=partpurple!70!black,font=\small] at (1.3,4.13) {$ 1-4,\;2-3 $};
  \end{scope}
\end{tikzpicture}
}
 \caption{For each fragment in $H$, there is one in $H_2$ with the same endpoints and vice versa.} 
    \label{fig:H2-Ham}
\end{figure}

%\begin{figure}
%    \centering
%    \includegraphics[width=0.6\linewidth]{Fig09-H2-Ham.jpg}
%    \caption{For each fragment in $H$, there is one in $H_2$ with the same endpoints and vice versa.} 
%    \label{fig:H2-Ham}
%\end{figure}

% Figure 10 — complete paste-ready figure.
% Shared styles and graph-coordinate macros from previous figures are used.
\begin{figure}[p]
\centering
\resizebox{0.75\linewidth}{!}{%
\begin{tikzpicture}[orange highlight/.style={draw=pathorange,line width=8pt,line cap=round,line join=round}]
  % Panel 1: 1,\;2-3
  \begin{scope}[shift={(0.00,5.10)},scale=0.72]
    \GCoordsHthree
    % singleton at 1 = TL
    \draw[orange highlight] (TL)--++(-0.45,0);
    \draw[orange highlight] (TL)--++(0,0.45);
    % path 2-3: TR ... BR
    \draw[orange highlight] (TR)--(TM);
    \draw[orange highlight] (TM)--(a);
    \draw[orange highlight] (a)--(c);
    \draw[orange highlight] (c)--(b);
    \draw[orange highlight] (b)--(RM);
    \draw[orange highlight] (RM)--(f);
    \draw[orange highlight] (f)--(d);
    \draw[orange highlight] (d)--(e);
    \draw[orange highlight] (e)--(LM);
    \draw[orange highlight] (LM)--(BL);
    \draw[orange highlight] (BL)--(BM);
    \draw[orange highlight] (BM)--(BR);
    \draw[orange highlight] (TR)--++(0.45,0);
    \draw[orange highlight] (BR)--++(0,-0.45);
    \GEdgesHthree
    \GNodesHthree
    \node[text=partpurple!70!black,font=\small] at (1.3,4.13) {$ 1,\;2-3 $};
  \end{scope}
  % Panel 2: 1,\;2-4
  \begin{scope}[shift={(3.50,5.10)},scale=0.72]
    \GCoordsHthree
    % singleton at 1 = TL
    \draw[orange highlight] (TL)--++(-0.45,0);
    \draw[orange highlight] (TL)--++(0,0.45);
    % path 2-4: TR ... BL
    \draw[orange highlight] (TR)--(TM);
    \draw[orange highlight] (TM)--(a);
    \draw[orange highlight] (a)--(c);
    \draw[orange highlight] (c)--(b);
    \draw[orange highlight] (b)--(RM);
    \draw[orange highlight] (RM)--(BR);
    \draw[orange highlight] (BR)--(BM);
    \draw[orange highlight] (BM)--(f);
    \draw[orange highlight] (f)--(d);
    \draw[orange highlight] (d)--(e);
    \draw[orange highlight] (e)--(LM);
    \draw[orange highlight] (LM)--(BL);
    \draw[orange highlight] (TR)--++(0.45,0);
    \draw[orange highlight] (BL)--++(-0.45,0);
    \GEdgesHthree
    \GNodesHthree
    \node[text=partpurple!70!black,font=\small] at (1.3,4.13) {$ 1,\;2-4 $};
  \end{scope}
  % Panel 3: 1,\;3-4
  \begin{scope}[shift={(7.00,5.10)},scale=0.72]
    \GCoordsHthree
    % singleton at 1 = TL
    \draw[orange highlight] (TL)--++(-0.45,0);
    \draw[orange highlight] (TL)--++(0,0.45);
    % path 3-4: BR ... BL
    \draw[orange highlight] (BR)--(BM);
    \draw[orange highlight] (BM)--(f);
    \draw[orange highlight] (f)--(RM);
    \draw[orange highlight] (RM)--(TR);
    \draw[orange highlight] (TR)--(TM);
    \draw[orange highlight] (TM)--(a);
    \draw[orange highlight] (a)--(c);
    \draw[orange highlight] (c)--(b);
    \draw[orange highlight] (b)--(d);
    \draw[orange highlight] (d)--(e);
    \draw[orange highlight] (e)--(LM);
    \draw[orange highlight] (LM)--(BL);
    \draw[orange highlight] (BR)--++(0,-0.45);
    \draw[orange highlight] (BL)--++(-0.45,0);
    \GEdgesHthree
    \GNodesHthree
    \node[text=partpurple!70!black,font=\small] at (1.3,4.13) {$ 1,\;3-4 $};
  \end{scope}
  % Panel 4: 2,\;1-3
  \begin{scope}[shift={(10.50,5.10)},scale=0.72]
    \GCoordsHthree
    % singleton at 2 = TR
    \draw[orange highlight] (TR)--++(0.45,0);
    \draw[orange highlight] (TR)--++(0,0.45);
    % path 1-3: TL ... BR
    \draw[orange highlight] (TL)--(TM);
    \draw[orange highlight] (TM)--(a);
    \draw[orange highlight] (a)--(c);
    \draw[orange highlight] (c)--(b);
    \draw[orange highlight] (b)--(RM);
    \draw[orange highlight] (RM)--(f);
    \draw[orange highlight] (f)--(d);
    \draw[orange highlight] (d)--(e);
    \draw[orange highlight] (e)--(LM);
    \draw[orange highlight] (LM)--(BL);
    \draw[orange highlight] (BL)--(BM);
    \draw[orange highlight] (BM)--(BR);
    \draw[orange highlight] (TL)--++(0,0.45);
    \draw[orange highlight] (BR)--++(0,-0.45);
    \GEdgesHthree
    \GNodesHthree
    \node[text=partpurple!70!black,font=\small] at (1.3,4.13) {$ 2,\;1-3 $};
  \end{scope}
  % Panel 5: 2,\;1-4
  \begin{scope}[shift={(0.00,0.00)},scale=0.72]
    \GCoordsHthree
    % singleton at 2 = TR
    \draw[orange highlight] (TR)--++(0.45,0);
    \draw[orange highlight] (TR)--++(0,0.45);
    % path 1-4: TL ... BL
    \draw[orange highlight] (TL)--(TM);
    \draw[orange highlight] (TM)--(a);
    \draw[orange highlight] (a)--(c);
    \draw[orange highlight] (c)--(b);
    \draw[orange highlight] (b)--(RM);
    \draw[orange highlight] (RM)--(BR);
    \draw[orange highlight] (BR)--(BM);
    \draw[orange highlight] (BM)--(f);
    \draw[orange highlight] (f)--(d);
    \draw[orange highlight] (d)--(e);
    \draw[orange highlight] (e)--(LM);
    \draw[orange highlight] (LM)--(BL);
    \draw[orange highlight] (TL)--++(0,0.45);
    \draw[orange highlight] (BL)--++(-0.45,0);
    \GEdgesHthree
    \GNodesHthree
    \node[text=partpurple!70!black,font=\small] at (1.3,4.13) {$ 2,\;1-4 $};
  \end{scope}
  % Panel 6: 2,\;3-4
  \begin{scope}[shift={(3.50,0.00)},scale=0.72]
    \GCoordsHthree
    % singleton at 2 = TR
    \draw[orange highlight] (TR)--++(0.45,0);
    \draw[orange highlight] (TR)--++(0,0.45);
    % path 3-4: BR ... BL
    \draw[orange highlight] (BR)--(BM);
    \draw[orange highlight] (BM)--(f);
    \draw[orange highlight] (f)--(RM);
    \draw[orange highlight] (RM)--(b);
    \draw[orange highlight] (b)--(d);
    \draw[orange highlight] (d)--(e);
    \draw[orange highlight] (e)--(c);
    \draw[orange highlight] (c)--(a);
    \draw[orange highlight] (a)--(TM);
    \draw[orange highlight] (TM)--(TL);
    \draw[orange highlight] (TL)--(LM);
    \draw[orange highlight] (LM)--(BL);
    \draw[orange highlight] (BR)--++(0,-0.45);
    \draw[orange highlight] (BL)--++(-0.45,0);
    \GEdgesHthree
    \GNodesHthree
    \node[text=partpurple!70!black,font=\small] at (1.3,4.13) {$ 2,\;3-4 $};
  \end{scope}
  % Panel 7: 1-2,\;3-4
  \begin{scope}[shift={(7.00,0.00)},scale=0.72]
    \GCoordsHthree
    % path 1-2
    \draw[orange highlight] (TL)--(TM);
    \draw[orange highlight] (TM)--(TR);
    \draw[orange highlight] (TL)--++(0,0.45);
    \draw[orange highlight] (TR)--++(0.45,0);
    % path 3-4
    \draw[orange highlight] (BR)--(BM);
    \draw[orange highlight] (BM)--(f);
    \draw[orange highlight] (f)--(RM);
    \draw[orange highlight] (RM)--(b);
    \draw[orange highlight] (b)--(d);
    \draw[orange highlight] (d)--(e);
    \draw[orange highlight] (e)--(c);
    \draw[orange highlight] (c)--(a);
    \draw[orange highlight] (a)--(LM);
    \draw[orange highlight] (LM)--(BL);
    \draw[orange highlight] (BR)--++(0,-0.45);
    \draw[orange highlight] (BL)--++(-0.45,0);
    \GEdgesHthree
    \GNodesHthree
    \node[text=partpurple!70!black,font=\small] at (1.3,4.13) {$ 1-2,\;3-4 $};
  \end{scope}
  % Panel 8: 1-4,\;2-3
  \begin{scope}[shift={(10.50,0.00)},scale=0.72]
    \GCoordsHthree
    % path 1-4
    \draw[orange highlight] (TL)--(LM);
    \draw[orange highlight] (LM)--(BL);
    \draw[orange highlight] (TL)--++(0,0.45);
    \draw[orange highlight] (BL)--++(-0.45,0);
    % path 2-3
    \draw[orange highlight] (TR)--(RM);
    \draw[orange highlight] (RM)--(b);
    \draw[orange highlight] (b)--(TM);
    \draw[orange highlight] (TM)--(a);
    \draw[orange highlight] (a)--(c);
    \draw[orange highlight] (c)--(e);
    \draw[orange highlight] (e)--(d);
    \draw[orange highlight] (d)--(f);
    \draw[orange highlight] (f)--(BM);
    \draw[orange highlight] (BM)--(BR);
    \draw[orange highlight] (TR)--++(0.45,0);
    \draw[orange highlight] (BR)--++(0,-0.45);
    \GEdgesHthree
    \GNodesHthree
    \node[text=partpurple!70!black,font=\small] at (1.3,4.13) {$ 1-4,\;2-3 $};
  \end{scope}
\end{tikzpicture}
}
 \caption{For each fragment in $H$, there is one in $H_3$ with the same endpoints and vice versa.} 
    \label{fig:H3-Ham}
\end{figure}
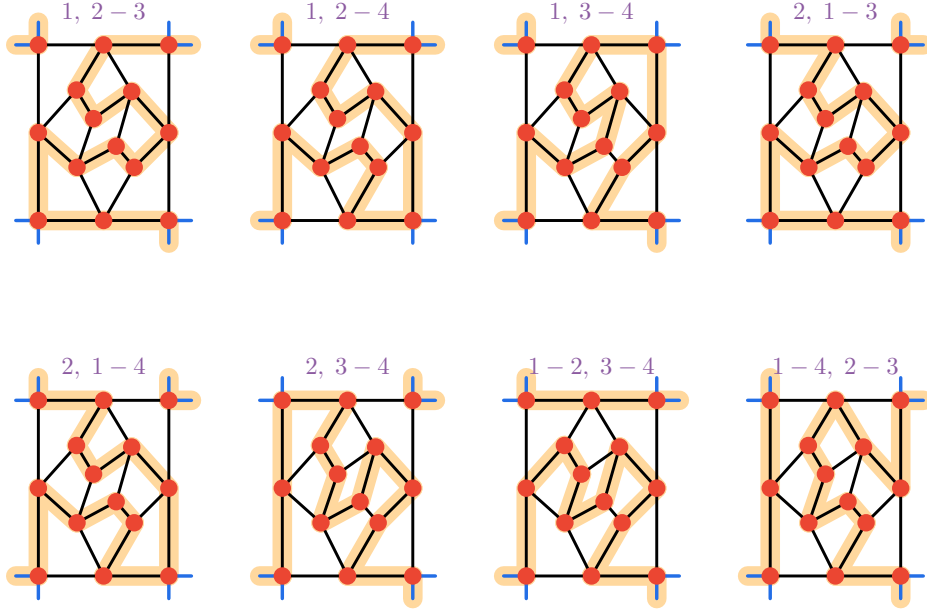

%\begin{figure}
%    \centering
%    \includegraphics[width=0.6\linewidth]{Fig10-H3-Ham.jpg}
%    \caption{For each fragment in $H$, there is one in $H_3$ with the same endpoints and vice versa.} 
%    \label{fig:H3-Ham}
%\end{figure}

% COMBINED FIGURE 11--13: 35 algorithmically verified Hamiltonian paths.
% A single shared H row followed by H_1, H_2, H_3 (two deleted vertices each).
% Requires existing shared graph preamble (\GCoordsH, \GEdgesH, \GNodesH etc.)
% and tikz plus graphicx. No external figure files are needed.
\begin{figure}[p]
\centering
\resizebox{!}{0.82\textheight}{%
\begin{tikzpicture}
  % Panel-local paths use existing shared graph styles.
  % 1 = top-left, 2 = top-right, 3 = bottom-right, 4 = bottom-left.
  \tikzset{combined path/.style={draw=pathorange,line width=8pt,
       line cap=round,line join=round}}

  % Row 1, column 1: H minus u, 1--3
  % Verified path: TL--TM--TR--RM--v--LM--BL--BM--BR
  \begin{scope}[shift={(0.00,0.00)},yscale=0.88]
    \GCoordsH
    \draw[combined path] (TL)--(TM);
    \draw[combined path] (TM)--(TR);
    \draw[combined path] (TR)--(RM);
    \draw[combined path] (RM)--(v);
    \draw[combined path] (v)--(LM);
    \draw[combined path] (LM)--(BL);
    \draw[combined path] (BL)--(BM);
    \draw[combined path] (BM)--(BR);
    \draw[combined path] (TL)--++(0,0.45);
    \draw[combined path] (BR)--++(0,-0.45);
    \GEdgesH
    \draw[grey edge] (TM)--(u);
    \draw[grey edge] (LM)--(u);
    \draw[grey edge] (RM)--(u);
    \GNodesH
    \node[grey vertex] at (u) {};
  \end{scope}
  % Row 1, column 2: H minus u, 1--4
  % Verified path: TL--TM--TR--RM--BR--BM--v--LM--BL
  \begin{scope}[shift={(4.10,0.00)},yscale=0.88]
    \GCoordsH
    \draw[combined path] (TL)--(TM);
    \draw[combined path] (TM)--(TR);
    \draw[combined path] (TR)--(RM);
    \draw[combined path] (RM)--(BR);
    \draw[combined path] (BR)--(BM);
    \draw[combined path] (BM)--(v);
    \draw[combined path] (v)--(LM);
    \draw[combined path] (LM)--(BL);
    \draw[combined path] (TL)--++(0,0.45);
    \draw[combined path] (BL)--++(-0.45,0);
    \GEdgesH
    \draw[grey edge] (TM)--(u);
    \draw[grey edge] (LM)--(u);
    \draw[grey edge] (RM)--(u);
    \GNodesH
    \node[grey vertex] at (u) {};
  \end{scope}
  % Row 1, column 3: H minus u, 2--3
  % Verified path: TR--TM--TL--LM--BL--BM--v--RM--BR
  \begin{scope}[shift={(8.20,0.00)},yscale=0.88]
    \GCoordsH
    \draw[combined path] (TR)--(TM);
    \draw[combined path] (TM)--(TL);
    \draw[combined path] (TL)--(LM);
    \draw[combined path] (LM)--(BL);
    \draw[combined path] (BL)--(BM);
    \draw[combined path] (BM)--(v);
    \draw[combined path] (v)--(RM);
    \draw[combined path] (RM)--(BR);
    \draw[combined path] (TR)--++(0.45,0);
    \draw[combined path] (BR)--++(0,-0.45);
    \GEdgesH
    \draw[grey edge] (TM)--(u);
    \draw[grey edge] (LM)--(u);
    \draw[grey edge] (RM)--(u);
    \GNodesH
    \node[grey vertex] at (u) {};
  \end{scope}
  % Row 1, column 4: H minus u, 2--4
  % Verified path: TR--TM--TL--LM--v--RM--BR--BM--BL
  \begin{scope}[shift={(12.30,0.00)},yscale=0.88]
    \GCoordsH
    \draw[combined path] (TR)--(TM);
    \draw[combined path] (TM)--(TL);
    \draw[combined path] (TL)--(LM);
    \draw[combined path] (LM)--(v);
    \draw[combined path] (v)--(RM);
    \draw[combined path] (RM)--(BR);
    \draw[combined path] (BR)--(BM);
    \draw[combined path] (BM)--(BL);
    \draw[combined path] (TR)--++(0.45,0);
    \draw[combined path] (BL)--++(-0.45,0);
    \GEdgesH
    \draw[grey edge] (TM)--(u);
    \draw[grey edge] (LM)--(u);
    \draw[grey edge] (RM)--(u);
    \GNodesH
    \node[grey vertex] at (u) {};
  \end{scope}
  % Row 1, column 5: H minus u, 3--4
  % Verified path: BR--RM--TR--TM--TL--LM--v--BM--BL
  \begin{scope}[shift={(16.40,0.00)},yscale=0.88]
    \GCoordsH
    \draw[combined path] (BR)--(RM);
    \draw[combined path] (RM)--(TR);
    \draw[combined path] (TR)--(TM);
    \draw[combined path] (TM)--(TL);
    \draw[combined path] (TL)--(LM);
    \draw[combined path] (LM)--(v);
    \draw[combined path] (v)--(BM);
    \draw[combined path] (BM)--(BL);
    \draw[combined path] (BR)--++(0,-0.45);
    \draw[combined path] (BL)--++(-0.45,0);
    \GEdgesH
    \draw[grey edge] (TM)--(u);
    \draw[grey edge] (LM)--(u);
    \draw[grey edge] (RM)--(u);
    \GNodesH
    \node[grey vertex] at (u) {};
  \end{scope}
  % Row 2, column 1: H1 minus l, 1--3
  % Verified path: TL--TM--TR--RM--r--v--u--LM--BL--BM--BR
  \begin{scope}[shift={(0.00,-5.15)},yscale=0.88]
    \GCoordsHone
    \draw[combined path] (TL)--(TM);
    \draw[combined path] (TM)--(TR);
    \draw[combined path] (TR)--(RM);
    \draw[combined path] (RM)--(r);
    \draw[combined path] (r)--(v);
    \draw[combined path] (v)--(u);
    \draw[combined path] (u)--(LM);
    \draw[combined path] (LM)--(BL);
    \draw[combined path] (BL)--(BM);
    \draw[combined path] (BM)--(BR);
    \draw[combined path] (TL)--++(0,0.45);
    \draw[combined path] (BR)--++(0,-0.45);
    \GEdgesHone
    \draw[grey edge] (LM)--(l);
    \draw[grey edge] (l)--(v);
    \draw[grey edge] (l)--(BM);
    \GNodesHone
    \node[grey vertex] at (l) {};
  \end{scope}
  % Row 2, column 2: H1 minus l, 1--4
  % Verified path: TL--TM--TR--RM--BR--BM--r--v--u--LM--BL
  \begin{scope}[shift={(4.10,-5.15)},yscale=0.88]
    \GCoordsHone
    \draw[combined path] (TL)--(TM);
    \draw[combined path] (TM)--(TR);
    \draw[combined path] (TR)--(RM);
    \draw[combined path] (RM)--(BR);
    \draw[combined path] (BR)--(BM);
    \draw[combined path] (BM)--(r);
    \draw[combined path] (r)--(v);
    \draw[combined path] (v)--(u);
    \draw[combined path] (u)--(LM);
    \draw[combined path] (LM)--(BL);
    \draw[combined path] (TL)--++(0,0.45);
    \draw[combined path] (BL)--++(-0.45,0);
    \GEdgesHone
    \draw[grey edge] (LM)--(l);
    \draw[grey edge] (l)--(v);
    \draw[grey edge] (l)--(BM);
    \GNodesHone
    \node[grey vertex] at (l) {};
  \end{scope}
  % Row 2, column 3: H1 minus l, 2--3
  % Verified path: TR--TM--TL--LM--BL--BM--r--v--u--RM--BR
  \begin{scope}[shift={(8.20,-5.15)},yscale=0.88]
    \GCoordsHone
    \draw[combined path] (TR)--(TM);
    \draw[combined path] (TM)--(TL);
    \draw[combined path] (TL)--(LM);
    \draw[combined path] (LM)--(BL);
    \draw[combined path] (BL)--(BM);
    \draw[combined path] (BM)--(r);
    \draw[combined path] (r)--(v);
    \draw[combined path] (v)--(u);
    \draw[combined path] (u)--(RM);
    \draw[combined path] (RM)--(BR);
    \draw[combined path] (TR)--++(0.45,0);
    \draw[combined path] (BR)--++(0,-0.45);
    \GEdgesHone
    \draw[grey edge] (LM)--(l);
    \draw[grey edge] (l)--(v);
    \draw[grey edge] (l)--(BM);
    \GNodesHone
    \node[grey vertex] at (l) {};
  \end{scope}
  % Row 2, column 4: H1 minus l, 2--4
  % Verified path: TR--TM--TL--LM--u--v--r--RM--BR--BM--BL
  \begin{scope}[shift={(12.30,-5.15)},yscale=0.88]
    \GCoordsHone
    \draw[combined path] (TR)--(TM);
    \draw[combined path] (TM)--(TL);
    \draw[combined path] (TL)--(LM);
    \draw[combined path] (LM)--(u);
    \draw[combined path] (u)--(v);
    \draw[combined path] (v)--(r);
    \draw[combined path] (r)--(RM);
    \draw[combined path] (RM)--(BR);
    \draw[combined path] (BR)--(BM);
    \draw[combined path] (BM)--(BL);
    \draw[combined path] (TR)--++(0.45,0);
    \draw[combined path] (BL)--++(-0.45,0);
    \GEdgesHone
    \draw[grey edge] (LM)--(l);
    \draw[grey edge] (l)--(v);
    \draw[grey edge] (l)--(BM);
    \GNodesHone
    \node[grey vertex] at (l) {};
  \end{scope}
  % Row 2, column 5: H1 minus l, 3--4
  % Verified path: BR--RM--TR--TM--TL--LM--u--v--r--BM--BL
  \begin{scope}[shift={(16.40,-5.15)},yscale=0.88]
    \GCoordsHone
    \draw[combined path] (BR)--(RM);
    \draw[combined path] (RM)--(TR);
    \draw[combined path] (TR)--(TM);
    \draw[combined path] (TM)--(TL);
    \draw[combined path] (TL)--(LM);
    \draw[combined path] (LM)--(u);
    \draw[combined path] (u)--(v);
    \draw[combined path] (v)--(r);
    \draw[combined path] (r)--(BM);
    \draw[combined path] (BM)--(BL);
    \draw[combined path] (BR)--++(0,-0.45);
    \draw[combined path] (BL)--++(-0.45,0);
    \GEdgesHone
    \draw[grey edge] (LM)--(l);
    \draw[grey edge] (l)--(v);
    \draw[grey edge] (l)--(BM);
    \GNodesHone
    \node[grey vertex] at (l) {};
  \end{scope}
  % Row 3, column 1: H1 minus u, 1--3
  % Verified path: TL--TM--TR--RM--r--v--l--LM--BL--BM--BR
  \begin{scope}[shift={(0.00,-9.95)},yscale=0.88]
    \GCoordsHone
    \draw[combined path] (TL)--(TM);
    \draw[combined path] (TM)--(TR);
    \draw[combined path] (TR)--(RM);
    \draw[combined path] (RM)--(r);
    \draw[combined path] (r)--(v);
    \draw[combined path] (v)--(l);
    \draw[combined path] (l)--(LM);
    \draw[combined path] (LM)--(BL);
    \draw[combined path] (BL)--(BM);
    \draw[combined path] (BM)--(BR);
    \draw[combined path] (TL)--++(0,0.45);
    \draw[combined path] (BR)--++(0,-0.45);
    \GEdgesHone
    \draw[grey edge] (TM)--(u);
    \draw[grey edge] (LM)--(u);
    \draw[grey edge] (RM)--(u);
    \draw[grey edge] (u)--(v);
    \GNodesHone
    \node[grey vertex] at (u) {};
  \end{scope}
  % Row 3, column 2: H1 minus u, 1--4
  % Verified path: TL--TM--TR--RM--BR--BM--r--v--l--LM--BL
  \begin{scope}[shift={(4.10,-9.95)},yscale=0.88]
    \GCoordsHone
    \draw[combined path] (TL)--(TM);
    \draw[combined path] (TM)--(TR);
    \draw[combined path] (TR)--(RM);
    \draw[combined path] (RM)--(BR);
    \draw[combined path] (BR)--(BM);
    \draw[combined path] (BM)--(r);
    \draw[combined path] (r)--(v);
    \draw[combined path] (v)--(l);
    \draw[combined path] (l)--(LM);
    \draw[combined path] (LM)--(BL);
    \draw[combined path] (TL)--++(0,0.45);
    \draw[combined path] (BL)--++(-0.45,0);
    \GEdgesHone
    \draw[grey edge] (TM)--(u);
    \draw[grey edge] (LM)--(u);
    \draw[grey edge] (RM)--(u);
    \draw[grey edge] (u)--(v);
    \GNodesHone
    \node[grey vertex] at (u) {};
  \end{scope}
  % Row 3, column 3: H1 minus u, 2--3
  % Verified path: TR--TM--TL--LM--BL--BM--l--v--r--RM--BR
  \begin{scope}[shift={(8.20,-9.95)},yscale=0.88]
    \GCoordsHone
    \draw[combined path] (TR)--(TM);
    \draw[combined path] (TM)--(TL);
    \draw[combined path] (TL)--(LM);
    \draw[combined path] (LM)--(BL);
    \draw[combined path] (BL)--(BM);
    \draw[combined path] (BM)--(l);
    \draw[combined path] (l)--(v);
    \draw[combined path] (v)--(r);
    \draw[combined path] (r)--(RM);
    \draw[combined path] (RM)--(BR);
    \draw[combined path] (TR)--++(0.45,0);
    \draw[combined path] (BR)--++(0,-0.45);
    \GEdgesHone
    \draw[grey edge] (TM)--(u);
    \draw[grey edge] (LM)--(u);
    \draw[grey edge] (RM)--(u);
    \draw[grey edge] (u)--(v);
    \GNodesHone
    \node[grey vertex] at (u) {};
  \end{scope}
  % Row 3, column 4: H1 minus u, 2--4
  % Verified path: TR--TM--TL--LM--l--v--r--RM--BR--BM--BL
  \begin{scope}[shift={(12.30,-9.95)},yscale=0.88]
    \GCoordsHone
    \draw[combined path] (TR)--(TM);
    \draw[combined path] (TM)--(TL);
    \draw[combined path] (TL)--(LM);
    \draw[combined path] (LM)--(l);
    \draw[combined path] (l)--(v);
    \draw[combined path] (v)--(r);
    \draw[combined path] (r)--(RM);
    \draw[combined path] (RM)--(BR);
    \draw[combined path] (BR)--(BM);
    \draw[combined path] (BM)--(BL);
    \draw[combined path] (TR)--++(0.45,0);
    \draw[combined path] (BL)--++(-0.45,0);
    \GEdgesHone
    \draw[grey edge] (TM)--(u);
    \draw[grey edge] (LM)--(u);
    \draw[grey edge] (RM)--(u);
    \draw[grey edge] (u)--(v);
    \GNodesHone
    \node[grey vertex] at (u) {};
  \end{scope}
  % Row 3, column 5: H1 minus u, 3--4
  % Verified path: BR--RM--TR--TM--TL--LM--l--v--r--BM--BL
  \begin{scope}[shift={(16.40,-9.95)},yscale=0.88]
    \GCoordsHone
    \draw[combined path] (BR)--(RM);
    \draw[combined path] (RM)--(TR);
    \draw[combined path] (TR)--(TM);
    \draw[combined path] (TM)--(TL);
    \draw[combined path] (TL)--(LM);
    \draw[combined path] (LM)--(l);
    \draw[combined path] (l)--(v);
    \draw[combined path] (v)--(r);
    \draw[combined path] (r)--(BM);
    \draw[combined path] (BM)--(BL);
    \draw[combined path] (BR)--++(0,-0.45);
    \draw[combined path] (BL)--++(-0.45,0);
    \GEdgesHone
    \draw[grey edge] (TM)--(u);
    \draw[grey edge] (LM)--(u);
    \draw[grey edge] (RM)--(u);
    \draw[grey edge] (u)--(v);
    \GNodesHone
    \node[grey vertex] at (u) {};
  \end{scope}
  % Row 4, column 1: H2 minus b, 1--3
  % Verified path: TL--TM--TR--RM--c--a--d--e--f--LM--BL--BM--BR
  \begin{scope}[shift={(0.00,-15.10)},yscale=0.88]
    \GCoordsHtwo
    \draw[combined path] (TL)--(TM);
    \draw[combined path] (TM)--(TR);
    \draw[combined path] (TR)--(RM);
    \draw[combined path] (RM)--(c);
    \draw[combined path] (c)--(a);
    \draw[combined path] (a)--(d);
    \draw[combined path] (d)--(e);
    \draw[combined path] (e)--(f);
    \draw[combined path] (f)--(LM);
    \draw[combined path] (LM)--(BL);
    \draw[combined path] (BL)--(BM);
    \draw[combined path] (BM)--(BR);
    \draw[combined path] (TL)--++(0,0.45);
    \draw[combined path] (BR)--++(0,-0.45);
    \GEdgesHtwo
    \draw[grey edge] (TM)--(b);
    \draw[grey edge] (b)--(RM);
    \draw[grey edge] (a)--(b);
    \draw[grey edge] (LM) .. controls (0.15,3.08) and (1.05,3.26) .. (b);
    \GNodesHtwo
    \node[grey vertex] at (b) {};
  \end{scope}
  % Row 4, column 2: H2 minus b, 1--4
  % Verified path: TL--TM--TR--RM--BR--BM--f--e--c--a--d--LM--BL
  \begin{scope}[shift={(4.10,-15.10)},yscale=0.88]
    \GCoordsHtwo
    \draw[combined path] (TL)--(TM);
    \draw[combined path] (TM)--(TR);
    \draw[combined path] (TR)--(RM);
    \draw[combined path] (RM)--(BR);
    \draw[combined path] (BR)--(BM);
    \draw[combined path] (BM)--(f);
    \draw[combined path] (f)--(e);
    \draw[combined path] (e)--(c);
    \draw[combined path] (c)--(a);
    \draw[combined path] (a)--(d);
    \draw[combined path] (d)--(LM);
    \draw[combined path] (LM)--(BL);
    \draw[combined path] (TL)--++(0,0.45);
    \draw[combined path] (BL)--++(-0.45,0);
    \GEdgesHtwo
    \draw[grey edge] (TM)--(b);
    \draw[grey edge] (b)--(RM);
    \draw[grey edge] (a)--(b);
    \draw[grey edge] (LM) .. controls (0.15,3.08) and (1.05,3.26) .. (b);
    \GNodesHtwo
    \node[grey vertex] at (b) {};
  \end{scope}
  % Row 4, column 3: H2 minus b, 2--3
  % Verified path: TR--TM--TL--LM--BL--BM--f--e--d--a--c--RM--BR
  \begin{scope}[shift={(8.20,-15.10)},yscale=0.88]
    \GCoordsHtwo
    \draw[combined path] (TR)--(TM);
    \draw[combined path] (TM)--(TL);
    \draw[combined path] (TL)--(LM);
    \draw[combined path] (LM)--(BL);
    \draw[combined path] (BL)--(BM);
    \draw[combined path] (BM)--(f);
    \draw[combined path] (f)--(e);
    \draw[combined path] (e)--(d);
    \draw[combined path] (d)--(a);
    \draw[combined path] (a)--(c);
    \draw[combined path] (c)--(RM);
    \draw[combined path] (RM)--(BR);
    \draw[combined path] (TR)--++(0.45,0);
    \draw[combined path] (BR)--++(0,-0.45);
    \GEdgesHtwo
    \draw[grey edge] (TM)--(b);
    \draw[grey edge] (b)--(RM);
    \draw[grey edge] (a)--(b);
    \draw[grey edge] (LM) .. controls (0.15,3.08) and (1.05,3.26) .. (b);
    \GNodesHtwo
    \node[grey vertex] at (b) {};
  \end{scope}
  % Row 4, column 4: H2 minus b, 2--4
  % Verified path: TR--TM--TL--LM--f--e--d--a--c--RM--BR--BM--BL
  \begin{scope}[shift={(12.30,-15.10)},yscale=0.88]
    \GCoordsHtwo
    \draw[combined path] (TR)--(TM);
    \draw[combined path] (TM)--(TL);
    \draw[combined path] (TL)--(LM);
    \draw[combined path] (LM)--(f);
    \draw[combined path] (f)--(e);
    \draw[combined path] (e)--(d);
    \draw[combined path] (d)--(a);
    \draw[combined path] (a)--(c);
    \draw[combined path] (c)--(RM);
    \draw[combined path] (RM)--(BR);
    \draw[combined path] (BR)--(BM);
    \draw[combined path] (BM)--(BL);
    \draw[combined path] (TR)--++(0.45,0);
    \draw[combined path] (BL)--++(-0.45,0);
    \GEdgesHtwo
    \draw[grey edge] (TM)--(b);
    \draw[grey edge] (b)--(RM);
    \draw[grey edge] (a)--(b);
    \draw[grey edge] (LM) .. controls (0.15,3.08) and (1.05,3.26) .. (b);
    \GNodesHtwo
    \node[grey vertex] at (b) {};
  \end{scope}
  % Row 4, column 5: H2 minus b, 3--4
  % Verified path: BR--RM--TR--TM--TL--LM--d--a--c--e--f--BM--BL
  \begin{scope}[shift={(16.40,-15.10)},yscale=0.88]
    \GCoordsHtwo
    \draw[combined path] (BR)--(RM);
    \draw[combined path] (RM)--(TR);
    \draw[combined path] (TR)--(TM);
    \draw[combined path] (TM)--(TL);
    \draw[combined path] (TL)--(LM);
    \draw[combined path] (LM)--(d);
    \draw[combined path] (d)--(a);
    \draw[combined path] (a)--(c);
    \draw[combined path] (c)--(e);
    \draw[combined path] (e)--(f);
    \draw[combined path] (f)--(BM);
    \draw[combined path] (BM)--(BL);
    \draw[combined path] (BR)--++(0,-0.45);
    \draw[combined path] (BL)--++(-0.45,0);
    \GEdgesHtwo
    \draw[grey edge] (TM)--(b);
    \draw[grey edge] (b)--(RM);
    \draw[grey edge] (a)--(b);
    \draw[grey edge] (LM) .. controls (0.15,3.08) and (1.05,3.26) .. (b);
    \GNodesHtwo
    \node[grey vertex] at (b) {};
  \end{scope}
  % Row 5, column 1: H2 minus c, 1--3
  % Verified path: TL--TM--TR--RM--b--a--d--e--f--LM--BL--BM--BR
  \begin{scope}[shift={(0.00,-19.90)},yscale=0.88]
    \GCoordsHtwo
    \draw[combined path] (TL)--(TM);
    \draw[combined path] (TM)--(TR);
    \draw[combined path] (TR)--(RM);
    \draw[combined path] (RM)--(b);
    \draw[combined path] (b)--(a);
    \draw[combined path] (a)--(d);
    \draw[combined path] (d)--(e);
    \draw[combined path] (e)--(f);
    \draw[combined path] (f)--(LM);
    \draw[combined path] (LM)--(BL);
    \draw[combined path] (BL)--(BM);
    \draw[combined path] (BM)--(BR);
    \draw[combined path] (TL)--++(0,0.45);
    \draw[combined path] (BR)--++(0,-0.45);
    \GEdgesHtwo
    \draw[grey edge] (a)--(c);
    \draw[grey edge] (c)--(RM);
    \draw[grey edge] (c)--(e);
    \GNodesHtwo
    \node[grey vertex] at (c) {};
  \end{scope}
  % Row 5, column 2: H2 minus c, 1--4
  % Verified path: TL--LM--f--e--d--a--b--TM--TR--RM--BR--BM--BL
  \begin{scope}[shift={(4.10,-19.90)},yscale=0.88]
    \GCoordsHtwo
    \draw[combined path] (TL)--(LM);
    \draw[combined path] (LM)--(f);
    \draw[combined path] (f)--(e);
    \draw[combined path] (e)--(d);
    \draw[combined path] (d)--(a);
    \draw[combined path] (a)--(b);
    \draw[combined path] (b)--(TM);
    \draw[combined path] (TM)--(TR);
    \draw[combined path] (TR)--(RM);
    \draw[combined path] (RM)--(BR);
    \draw[combined path] (BR)--(BM);
    \draw[combined path] (BM)--(BL);
    \draw[combined path] (TL)--++(0,0.45);
    \draw[combined path] (BL)--++(-0.45,0);
    \GEdgesHtwo
    \draw[grey edge] (a)--(c);
    \draw[grey edge] (c)--(RM);
    \draw[grey edge] (c)--(e);
    \GNodesHtwo
    \node[grey vertex] at (c) {};
  \end{scope}
  % Row 5, column 3: H2 minus c, 2--3
  % Verified path: TR--TM--TL--LM--BL--BM--f--e--d--a--b--RM--BR
  \begin{scope}[shift={(8.20,-19.90)},yscale=0.88]
    \GCoordsHtwo
    \draw[combined path] (TR)--(TM);
    \draw[combined path] (TM)--(TL);
    \draw[combined path] (TL)--(LM);
    \draw[combined path] (LM)--(BL);
    \draw[combined path] (BL)--(BM);
    \draw[combined path] (BM)--(f);
    \draw[combined path] (f)--(e);
    \draw[combined path] (e)--(d);
    \draw[combined path] (d)--(a);
    \draw[combined path] (a)--(b);
    \draw[combined path] (b)--(RM);
    \draw[combined path] (RM)--(BR);
    \draw[combined path] (TR)--++(0.45,0);
    \draw[combined path] (BR)--++(0,-0.45);
    \GEdgesHtwo
    \draw[grey edge] (a)--(c);
    \draw[grey edge] (c)--(RM);
    \draw[grey edge] (c)--(e);
    \GNodesHtwo
    \node[grey vertex] at (c) {};
  \end{scope}
  % Row 5, column 4: H2 minus c, 2--4
  % Verified path: TR--RM--BR--BM--f--e--d--a--b--TM--TL--LM--BL
  \begin{scope}[shift={(12.30,-19.90)},yscale=0.88]
    \GCoordsHtwo
    \draw[combined path] (TR)--(RM);
    \draw[combined path] (RM)--(BR);
    \draw[combined path] (BR)--(BM);
    \draw[combined path] (BM)--(f);
    \draw[combined path] (f)--(e);
    \draw[combined path] (e)--(d);
    \draw[combined path] (d)--(a);
    \draw[combined path] (a)--(b);
    \draw[combined path] (b)--(TM);
    \draw[combined path] (TM)--(TL);
    \draw[combined path] (TL)--(LM);
    \draw[combined path] (LM)--(BL);
    \draw[combined path] (TR)--++(0.45,0);
    \draw[combined path] (BL)--++(-0.45,0);
    \GEdgesHtwo
    \draw[grey edge] (a)--(c);
    \draw[grey edge] (c)--(RM);
    \draw[grey edge] (c)--(e);
    \GNodesHtwo
    \node[grey vertex] at (c) {};
  \end{scope}
  % Row 5, column 5: H2 minus c, 3--4
  % Verified path: BR--BM--f--e--d--a--b--RM--TR--TM--TL--LM--BL
  \begin{scope}[shift={(16.40,-19.90)},yscale=0.88]
    \GCoordsHtwo
    \draw[combined path] (BR)--(BM);
    \draw[combined path] (BM)--(f);
    \draw[combined path] (f)--(e);
    \draw[combined path] (e)--(d);
    \draw[combined path] (d)--(a);
    \draw[combined path] (a)--(b);
    \draw[combined path] (b)--(RM);
    \draw[combined path] (RM)--(TR);
    \draw[combined path] (TR)--(TM);
    \draw[combined path] (TM)--(TL);
    \draw[combined path] (TL)--(LM);
    \draw[combined path] (LM)--(BL);
    \draw[combined path] (BR)--++(0,-0.45);
    \draw[combined path] (BL)--++(-0.45,0);
    \GEdgesHtwo
    \draw[grey edge] (a)--(c);
    \draw[grey edge] (c)--(RM);
    \draw[grey edge] (c)--(e);
    \GNodesHtwo
    \node[grey vertex] at (c) {};
  \end{scope}
  % Row 6, column 1: H3 minus a, 1--3
  % Verified path: TL--TM--TR--RM--f--d--b--c--e--LM--BL--BM--BR
  \begin{scope}[shift={(0.00,-25.05)},yscale=0.88]
    \GCoordsHthree
    \draw[combined path] (TL)--(TM);
    \draw[combined path] (TM)--(TR);
    \draw[combined path] (TR)--(RM);
    \draw[combined path] (RM)--(f);
    \draw[combined path] (f)--(d);
    \draw[combined path] (d)--(b);
    \draw[combined path] (b)--(c);
    \draw[combined path] (c)--(e);
    \draw[combined path] (e)--(LM);
    \draw[combined path] (LM)--(BL);
    \draw[combined path] (BL)--(BM);
    \draw[combined path] (BM)--(BR);
    \draw[combined path] (TL)--++(0,0.45);
    \draw[combined path] (BR)--++(0,-0.45);
    \GEdgesHthree
    \draw[grey edge] (TM)--(a);
    \draw[grey edge] (LM)--(a);
    \draw[grey edge] (a)--(c);
    \GNodesHthree
    \node[grey vertex] at (a) {};
  \end{scope}
  % Row 6, column 2: H3 minus a, 1--4
  % Verified path: TL--TM--TR--RM--BR--BM--f--d--b--c--e--LM--BL
  \begin{scope}[shift={(4.10,-25.05)},yscale=0.88]
    \GCoordsHthree
    \draw[combined path] (TL)--(TM);
    \draw[combined path] (TM)--(TR);
    \draw[combined path] (TR)--(RM);
    \draw[combined path] (RM)--(BR);
    \draw[combined path] (BR)--(BM);
    \draw[combined path] (BM)--(f);
    \draw[combined path] (f)--(d);
    \draw[combined path] (d)--(b);
    \draw[combined path] (b)--(c);
    \draw[combined path] (c)--(e);
    \draw[combined path] (e)--(LM);
    \draw[combined path] (LM)--(BL);
    \draw[combined path] (TL)--++(0,0.45);
    \draw[combined path] (BL)--++(-0.45,0);
    \GEdgesHthree
    \draw[grey edge] (TM)--(a);
    \draw[grey edge] (LM)--(a);
    \draw[grey edge] (a)--(c);
    \GNodesHthree
    \node[grey vertex] at (a) {};
  \end{scope}
  % Row 6, column 3: H3 minus a, 2--3
  % Verified path: TR--TM--TL--LM--BL--BM--e--c--b--d--f--RM--BR
  \begin{scope}[shift={(8.20,-25.05)},yscale=0.88]
    \GCoordsHthree
    \draw[combined path] (TR)--(TM);
    \draw[combined path] (TM)--(TL);
    \draw[combined path] (TL)--(LM);
    \draw[combined path] (LM)--(BL);
    \draw[combined path] (BL)--(BM);
    \draw[combined path] (BM)--(e);
    \draw[combined path] (e)--(c);
    \draw[combined path] (c)--(b);
    \draw[combined path] (b)--(d);
    \draw[combined path] (d)--(f);
    \draw[combined path] (f)--(RM);
    \draw[combined path] (RM)--(BR);
    \draw[combined path] (TR)--++(0.45,0);
    \draw[combined path] (BR)--++(0,-0.45);
    \GEdgesHthree
    \draw[grey edge] (TM)--(a);
    \draw[grey edge] (LM)--(a);
    \draw[grey edge] (a)--(c);
    \GNodesHthree
    \node[grey vertex] at (a) {};
  \end{scope}
  % Row 6, column 4: H3 minus a, 2--4
  % Verified path: TR--TM--TL--LM--e--c--b--d--f--RM--BR--BM--BL
  \begin{scope}[shift={(12.30,-25.05)},yscale=0.88]
    \GCoordsHthree
    \draw[combined path] (TR)--(TM);
    \draw[combined path] (TM)--(TL);
    \draw[combined path] (TL)--(LM);
    \draw[combined path] (LM)--(e);
    \draw[combined path] (e)--(c);
    \draw[combined path] (c)--(b);
    \draw[combined path] (b)--(d);
    \draw[combined path] (d)--(f);
    \draw[combined path] (f)--(RM);
    \draw[combined path] (RM)--(BR);
    \draw[combined path] (BR)--(BM);
    \draw[combined path] (BM)--(BL);
    \draw[combined path] (TR)--++(0.45,0);
    \draw[combined path] (BL)--++(-0.45,0);
    \GEdgesHthree
    \draw[grey edge] (TM)--(a);
    \draw[grey edge] (LM)--(a);
    \draw[grey edge] (a)--(c);
    \GNodesHthree
    \node[grey vertex] at (a) {};
  \end{scope}
  % Row 6, column 5: H3 minus a, 3--4
  % Verified path: BR--RM--TR--TM--TL--LM--e--c--b--d--f--BM--BL
  \begin{scope}[shift={(16.40,-25.05)},yscale=0.88]
    \GCoordsHthree
    \draw[combined path] (BR)--(RM);
    \draw[combined path] (RM)--(TR);
    \draw[combined path] (TR)--(TM);
    \draw[combined path] (TM)--(TL);
    \draw[combined path] (TL)--(LM);
    \draw[combined path] (LM)--(e);
    \draw[combined path] (e)--(c);
    \draw[combined path] (c)--(b);
    \draw[combined path] (b)--(d);
    \draw[combined path] (d)--(f);
    \draw[combined path] (f)--(BM);
    \draw[combined path] (BM)--(BL);
    \draw[combined path] (BR)--++(0,-0.45);
    \draw[combined path] (BL)--++(-0.45,0);
    \GEdgesHthree
    \draw[grey edge] (TM)--(a);
    \draw[grey edge] (LM)--(a);
    \draw[grey edge] (a)--(c);
    \GNodesHthree
    \node[grey vertex] at (a) {};
  \end{scope}
  % Row 7, column 1: H3 minus b, 1--3
  % Verified path: TL--TM--TR--RM--f--d--e--c--a--LM--BL--BM--BR
  \begin{scope}[shift={(0.00,-29.85)},yscale=0.88]
    \GCoordsHthree
    \draw[combined path] (TL)--(TM);
    \draw[combined path] (TM)--(TR);
    \draw[combined path] (TR)--(RM);
    \draw[combined path] (RM)--(f);
    \draw[combined path] (f)--(d);
    \draw[combined path] (d)--(e);
    \draw[combined path] (e)--(c);
    \draw[combined path] (c)--(a);
    \draw[combined path] (a)--(LM);
    \draw[combined path] (LM)--(BL);
    \draw[combined path] (BL)--(BM);
    \draw[combined path] (BM)--(BR);
    \draw[combined path] (TL)--++(0,0.45);
    \draw[combined path] (BR)--++(0,-0.45);
    \GEdgesHthree
    \draw[grey edge] (TM)--(b);
    \draw[grey edge] (c)--(b);
    \draw[grey edge] (b)--(RM);
    \draw[grey edge] (b)--(d);
    \GNodesHthree
    \node[grey vertex] at (b) {};
  \end{scope}
  % Row 7, column 2: H3 minus b, 1--4
  % Verified path: TL--TM--TR--RM--BR--BM--f--d--e--c--a--LM--BL
  \begin{scope}[shift={(4.10,-29.85)},yscale=0.88]
    \GCoordsHthree
    \draw[combined path] (TL)--(TM);
    \draw[combined path] (TM)--(TR);
    \draw[combined path] (TR)--(RM);
    \draw[combined path] (RM)--(BR);
    \draw[combined path] (BR)--(BM);
    \draw[combined path] (BM)--(f);
    \draw[combined path] (f)--(d);
    \draw[combined path] (d)--(e);
    \draw[combined path] (e)--(c);
    \draw[combined path] (c)--(a);
    \draw[combined path] (a)--(LM);
    \draw[combined path] (LM)--(BL);
    \draw[combined path] (TL)--++(0,0.45);
    \draw[combined path] (BL)--++(-0.45,0);
    \GEdgesHthree
    \draw[grey edge] (TM)--(b);
    \draw[grey edge] (c)--(b);
    \draw[grey edge] (b)--(RM);
    \draw[grey edge] (b)--(d);
    \GNodesHthree
    \node[grey vertex] at (b) {};
  \end{scope}
  % Row 7, column 3: H3 minus b, 2--3
  % Verified path: TR--RM--f--d--e--c--a--TM--TL--LM--BL--BM--BR
  \begin{scope}[shift={(8.20,-29.85)},yscale=0.88]
    \GCoordsHthree
    \draw[combined path] (TR)--(RM);
    \draw[combined path] (RM)--(f);
    \draw[combined path] (f)--(d);
    \draw[combined path] (d)--(e);
    \draw[combined path] (e)--(c);
    \draw[combined path] (c)--(a);
    \draw[combined path] (a)--(TM);
    \draw[combined path] (TM)--(TL);
    \draw[combined path] (TL)--(LM);
    \draw[combined path] (LM)--(BL);
    \draw[combined path] (BL)--(BM);
    \draw[combined path] (BM)--(BR);
    \draw[combined path] (TR)--++(0.45,0);
    \draw[combined path] (BR)--++(0,-0.45);
    \GEdgesHthree
    \draw[grey edge] (TM)--(b);
    \draw[grey edge] (c)--(b);
    \draw[grey edge] (b)--(RM);
    \draw[grey edge] (b)--(d);
    \GNodesHthree
    \node[grey vertex] at (b) {};
  \end{scope}
  % Row 7, column 4: H3 minus b, 2--4
  % Verified path: TR--TM--TL--LM--a--c--e--d--f--RM--BR--BM--BL
  \begin{scope}[shift={(12.30,-29.85)},yscale=0.88]
    \GCoordsHthree
    \draw[combined path] (TR)--(TM);
    \draw[combined path] (TM)--(TL);
    \draw[combined path] (TL)--(LM);
    \draw[combined path] (LM)--(a);
    \draw[combined path] (a)--(c);
    \draw[combined path] (c)--(e);
    \draw[combined path] (e)--(d);
    \draw[combined path] (d)--(f);
    \draw[combined path] (f)--(RM);
    \draw[combined path] (RM)--(BR);
    \draw[combined path] (BR)--(BM);
    \draw[combined path] (BM)--(BL);
    \draw[combined path] (TR)--++(0.45,0);
    \draw[combined path] (BL)--++(-0.45,0);
    \GEdgesHthree
    \draw[grey edge] (TM)--(b);
    \draw[grey edge] (c)--(b);
    \draw[grey edge] (b)--(RM);
    \draw[grey edge] (b)--(d);
    \GNodesHthree
    \node[grey vertex] at (b) {};
  \end{scope}
  % Row 7, column 5: H3 minus b, 3--4
  % Verified path: BR--RM--TR--TM--TL--LM--a--c--e--d--f--BM--BL
  \begin{scope}[shift={(16.40,-29.85)},yscale=0.88]
    \GCoordsHthree
    \draw[combined path] (BR)--(RM);
    \draw[combined path] (RM)--(TR);
    \draw[combined path] (TR)--(TM);
    \draw[combined path] (TM)--(TL);
    \draw[combined path] (TL)--(LM);
    \draw[combined path] (LM)--(a);
    \draw[combined path] (a)--(c);
    \draw[combined path] (c)--(e);
    \draw[combined path] (e)--(d);
    \draw[combined path] (d)--(f);
    \draw[combined path] (f)--(BM);
    \draw[combined path] (BM)--(BL);
    \draw[combined path] (BR)--++(0,-0.45);
    \draw[combined path] (BL)--++(-0.45,0);
    \GEdgesHthree
    \draw[grey edge] (TM)--(b);
    \draw[grey edge] (c)--(b);
    \draw[grey edge] (b)--(RM);
    \draw[grey edge] (b)--(d);
    \GNodesHthree
    \node[grey vertex] at (b) {};
  \end{scope}
  % Column labels used for every row.
  \node[font=\normalsize,text=partpurple] at (1.30,4.67) {$ 1\! -\! 3 $};
  \node[font=\normalsize,text=partpurple] at (5.40,4.67) {$ 1\! -\! 4 $};
  \node[font=\normalsize,text=partpurple] at (9.50,4.67) {$ 2\! -\! 3 $};
  \node[font=\normalsize,text=partpurple] at (13.60,4.67) {$ 2\! -\! 4 $};
  \node[font=\normalsize,text=partpurple] at (17.70,4.67) {$ 3\! -\! 4 $};
  % A subtle rule separates the shared H reference row from the larger configurations.
  \draw[gray!55,line width=0.5pt] (-0.9,-0.78)--(19.9,-0.78);
  \node[font=\Large] at (-0.92,1.75) {$H$};
  \node[font=\Large] at (-0.95,-6.02) {$H_1$};
  \node[font=\Large] at (-0.95,-15.96) {$H_2$};
  \node[font=\Large] at (-0.95,-25.92) {$H_3$};
\end{tikzpicture}%
}
\begingroup\footnotesize\caption{The five column headings give the endpoint decomposition in the common boundary of $H$. Grey vertices and incident edges are deleted. The thick orange edges depict a cycle fragment. 
The first row records the common reference $H$,
followed by the two deletion cases for each of $H_1$, $H_2$, and $H_3$. %For each green $w_i \in H_i$, the graph $H_i - w_i$ has a collection of cycle fragments with endpoint decomposition matching $H- w$, for an internal green $w$. %Every endpoint decomposition in $H-v$ is realized by some cycle fragment in $H_i-w$
}
\label{fig:hypogreen}
\endgroup
\end{figure}
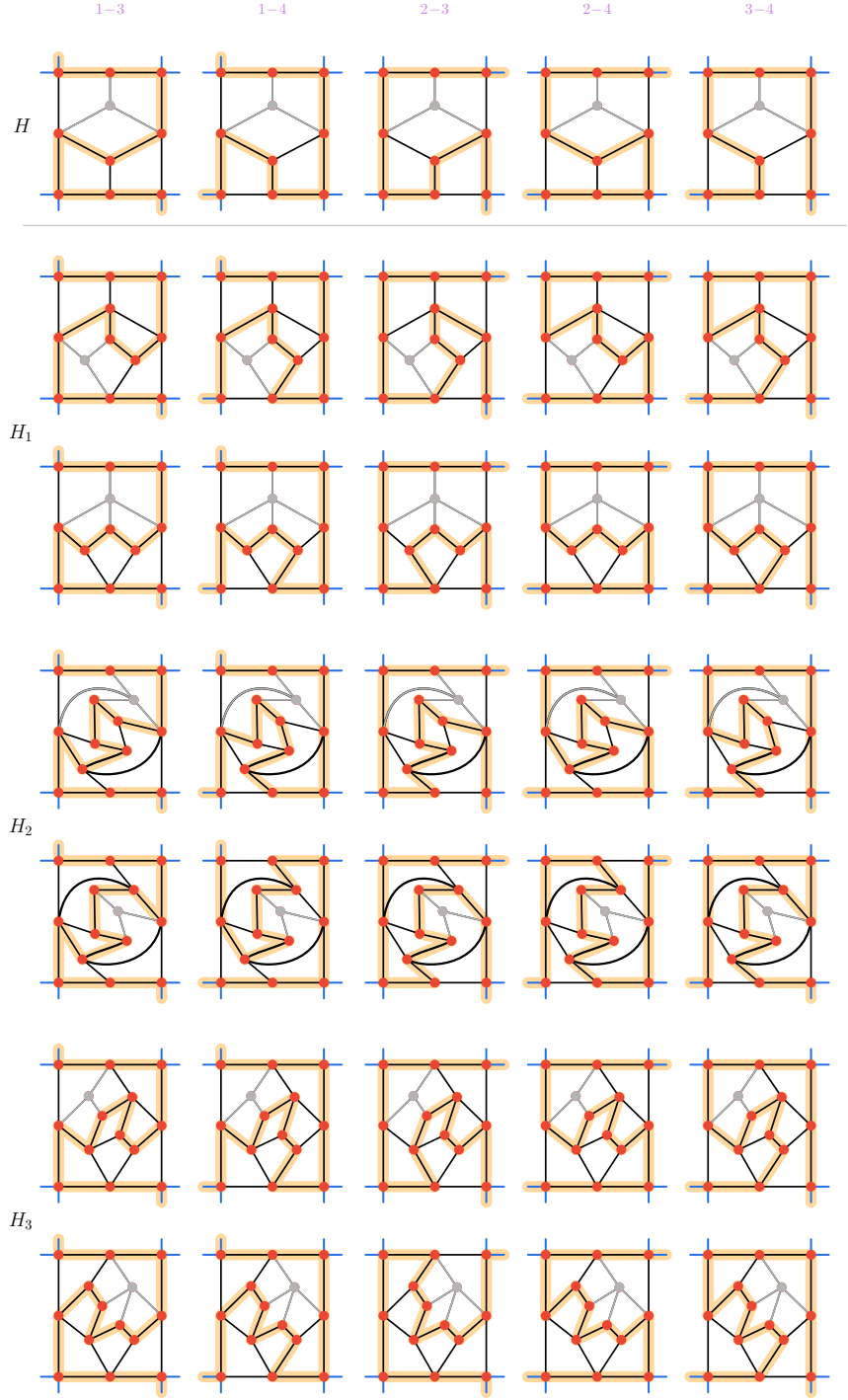

%\begin{figure}
%    \centering
%    \includegraphics[width=0.8\linewidth]{Fig11-H1-HypoGreen.jpg}
%    \caption{Each possible cycle fragment in $H_1 -v$ corresponds to a cycle fragment in $H - w$, for some green $w$. Up to symmetry, there are only two such vertices $v$.}
%    \label{fig:H1-HypoGreen}
%\end{figure}

%\begin{figure}
%    \centering
 %   \includegraphics[width=0.8\linewidth]{Fig12-H2-HypoGreen.jpg}
 %   \caption{Each possible cycle fragment in $H_2 -v$ corresponds to a cycle fragment in $H-w$, for some green $w$. Up to symmetry, there are only two such vertices $v$.}
 %   \label{fig:H2-HypoGreen}
%\end{figure}

%\begin{figure}
%    \centering
%    \includegraphics[width=0.8\linewidth]{Fig13-H3-HypoGreen.jpg}
%    \caption{Each possible cycle fragment in $H_3 -v$ corresponds to a cycle fragment in $H-w$, for some green $w$. Up to symmetry, there are only two such vertices $v$.}
%    \label{fig:H3-HypoGreen}
%\end{figure}

% Optional two-page alternative. Both parts belong to a single figure number.
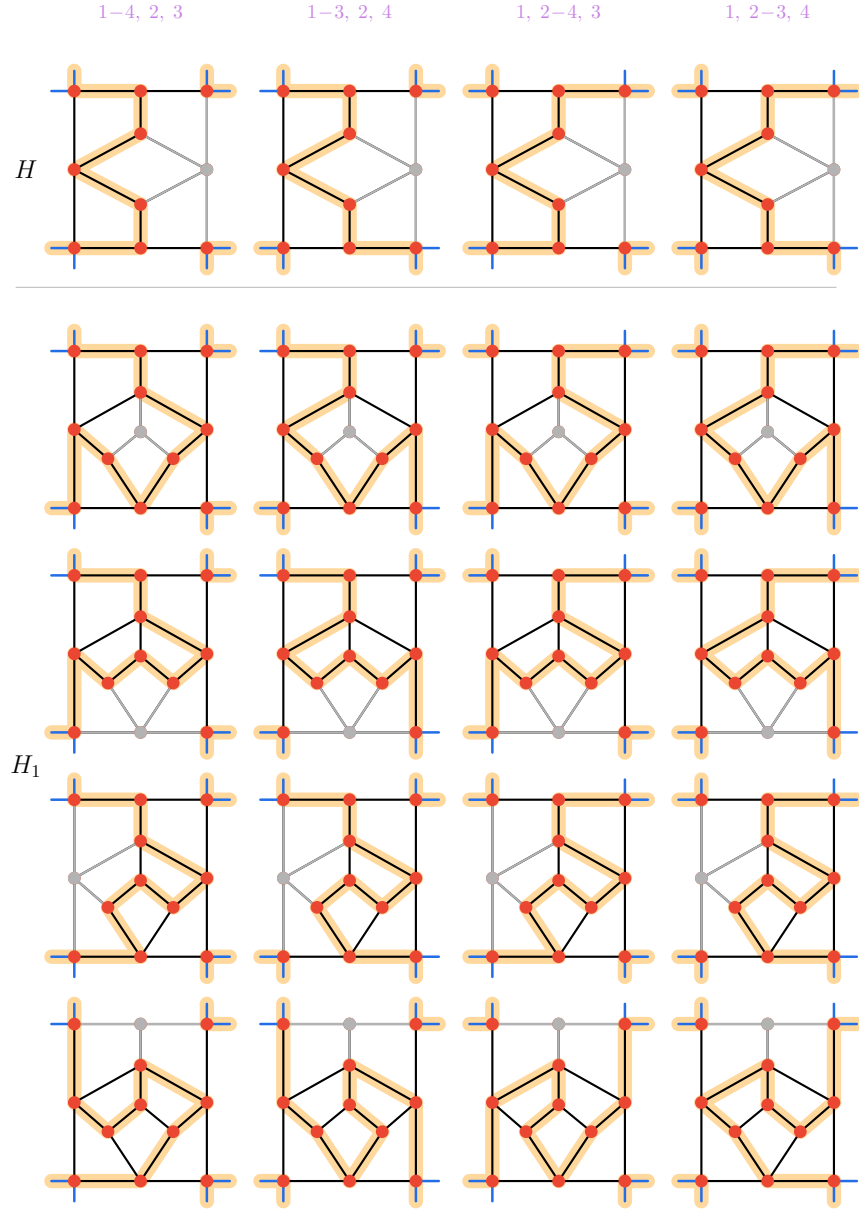
\begin{figure}[p]
\centering
\resizebox{!}{0.70\textheight}{%
\begin{tikzpicture}
  \tikzset{combined path/.style={draw=pathorange,line width=8pt,line cap=round,line join=round}}
  % Row 1, column 1: H minus RM; 1-4,2,3
  \begin{scope}[shift={(0.00,0.00)},yscale=0.88]
    \GCoordsH
    \draw[combined path] (TL)--(TM);
    \draw[combined path] (TM)--(u);
    \draw[combined path] (u)--(LM);
    \draw[combined path] (LM)--(v);
    \draw[combined path] (v)--(BM);
    \draw[combined path] (BM)--(BL);
    \draw[combined path] (TL)--++(0.00,0.45);
    \draw[combined path] (BL)--++(-0.45,0.00);
    \draw[combined path] (TR)--++(0.00,0.45);
    \draw[combined path] (TR)--++(0.45,0.00);
    \draw[combined path] (BR)--++(0.45,0.00);
    \draw[combined path] (BR)--++(0.00,-0.45);
    \GEdgesH
    \draw[grey edge] (TR)--(RM);
    \draw[grey edge] (RM)--(BR);
    \draw[grey edge] (RM)--(u);
    \draw[grey edge] (RM)--(v);
    \GNodesH
    \node[grey vertex] at (RM) {};
  \end{scope}
  % Row 1, column 2: H minus RM; 1-3,2,4
  \begin{scope}[shift={(4.10,0.00)},yscale=0.88]
    \GCoordsH
    \draw[combined path] (TL)--(TM);
    \draw[combined path] (TM)--(u);
    \draw[combined path] (u)--(LM);
    \draw[combined path] (LM)--(v);
    \draw[combined path] (v)--(BM);
    \draw[combined path] (BM)--(BR);
    \draw[combined path] (TL)--++(0.00,0.45);
    \draw[combined path] (BR)--++(0.00,-0.45);
    \draw[combined path] (TR)--++(0.00,0.45);
    \draw[combined path] (TR)--++(0.45,0.00);
    \draw[combined path] (BL)--++(-0.45,0.00);
    \draw[combined path] (BL)--++(0.00,-0.45);
    \GEdgesH
    \draw[grey edge] (TR)--(RM);
    \draw[grey edge] (RM)--(BR);
    \draw[grey edge] (RM)--(u);
    \draw[grey edge] (RM)--(v);
    \GNodesH
    \node[grey vertex] at (RM) {};
  \end{scope}
  % Row 1, column 3: H minus RM; 1,2-4,3
  \begin{scope}[shift={(8.20,0.00)},yscale=0.88]
    \GCoordsH
    \draw[combined path] (TR)--(TM);
    \draw[combined path] (TM)--(u);
    \draw[combined path] (u)--(LM);
    \draw[combined path] (LM)--(v);
    \draw[combined path] (v)--(BM);
    \draw[combined path] (BM)--(BL);
    \draw[combined path] (TR)--++(0.45,0.00);
    \draw[combined path] (BL)--++(-0.45,0.00);
    \draw[combined path] (TL)--++(-0.45,0.00);
    \draw[combined path] (TL)--++(0.00,0.45);
    \draw[combined path] (BR)--++(0.45,0.00);
    \draw[combined path] (BR)--++(0.00,-0.45);
    \GEdgesH
    \draw[grey edge] (TR)--(RM);
    \draw[grey edge] (RM)--(BR);
    \draw[grey edge] (RM)--(u);
    \draw[grey edge] (RM)--(v);
    \GNodesH
    \node[grey vertex] at (RM) {};
  \end{scope}
  % Row 1, column 4: H minus RM; 1,2-3,4
  \begin{scope}[shift={(12.30,0.00)},yscale=0.88]
    \GCoordsH
    \draw[combined path] (TR)--(TM);
    \draw[combined path] (TM)--(u);
    \draw[combined path] (u)--(LM);
    \draw[combined path] (LM)--(v);
    \draw[combined path] (v)--(BM);
    \draw[combined path] (BM)--(BR);
    \draw[combined path] (TR)--++(0.45,0.00);
    \draw[combined path] (BR)--++(0.00,-0.45);
    \draw[combined path] (TL)--++(-0.45,0.00);
    \draw[combined path] (TL)--++(0.00,0.45);
    \draw[combined path] (BL)--++(-0.45,0.00);
    \draw[combined path] (BL)--++(0.00,-0.45);
    \GEdgesH
    \draw[grey edge] (TR)--(RM);
    \draw[grey edge] (RM)--(BR);
    \draw[grey edge] (RM)--(u);
    \draw[grey edge] (RM)--(v);
    \GNodesH
    \node[grey vertex] at (RM) {};
  \end{scope}
  % Row 2, column 1: H1 minus v; 1-4,2,3
  \begin{scope}[shift={(0.00,-5.10)},yscale=0.88]
    \GCoordsHone
    \draw[combined path] (TL)--(TM);
    \draw[combined path] (TM)--(u);
    \draw[combined path] (u)--(RM);
    \draw[combined path] (RM)--(r);
    \draw[combined path] (r)--(BM);
    \draw[combined path] (BM)--(l);
    \draw[combined path] (l)--(LM);
    \draw[combined path] (LM)--(BL);
    \draw[combined path] (TL)--++(0.00,0.45);
    \draw[combined path] (BL)--++(-0.45,0.00);
    \draw[combined path] (TR)--++(0.00,0.45);
    \draw[combined path] (TR)--++(0.45,0.00);
    \draw[combined path] (BR)--++(0.45,0.00);
    \draw[combined path] (BR)--++(0.00,-0.45);
    \GEdgesHone
    \draw[grey edge] (u)--(v);
    \draw[grey edge] (l)--(v);
    \draw[grey edge] (v)--(r);
    \GNodesHone
    \node[grey vertex] at (v) {};
  \end{scope}
  % Row 2, column 2: H1 minus v; 1-3,2,4
  \begin{scope}[shift={(4.10,-5.10)},yscale=0.88]
    \GCoordsHone
    \draw[combined path] (TL)--(TM);
    \draw[combined path] (TM)--(u);
    \draw[combined path] (u)--(LM);
    \draw[combined path] (LM)--(l);
    \draw[combined path] (l)--(BM);
    \draw[combined path] (BM)--(r);
    \draw[combined path] (r)--(RM);
    \draw[combined path] (RM)--(BR);
    \draw[combined path] (TL)--++(0.00,0.45);
    \draw[combined path] (BR)--++(0.00,-0.45);
    \draw[combined path] (TR)--++(0.00,0.45);
    \draw[combined path] (TR)--++(0.45,0.00);
    \draw[combined path] (BL)--++(-0.45,0.00);
    \draw[combined path] (BL)--++(0.00,-0.45);
    \GEdgesHone
    \draw[grey edge] (u)--(v);
    \draw[grey edge] (l)--(v);
    \draw[grey edge] (v)--(r);
    \GNodesHone
    \node[grey vertex] at (v) {};
  \end{scope}
  % Row 2, column 3: H1 minus v; 1,2-4,3
  \begin{scope}[shift={(8.20,-5.10)},yscale=0.88]
    \GCoordsHone
    \draw[combined path] (TR)--(TM);
    \draw[combined path] (TM)--(u);
    \draw[combined path] (u)--(RM);
    \draw[combined path] (RM)--(r);
    \draw[combined path] (r)--(BM);
    \draw[combined path] (BM)--(l);
    \draw[combined path] (l)--(LM);
    \draw[combined path] (LM)--(BL);
    \draw[combined path] (TR)--++(0.45,0.00);
    \draw[combined path] (BL)--++(-0.45,0.00);
    \draw[combined path] (TL)--++(-0.45,0.00);
    \draw[combined path] (TL)--++(0.00,0.45);
    \draw[combined path] (BR)--++(0.45,0.00);
    \draw[combined path] (BR)--++(0.00,-0.45);
    \GEdgesHone
    \draw[grey edge] (u)--(v);
    \draw[grey edge] (l)--(v);
    \draw[grey edge] (v)--(r);
    \GNodesHone
    \node[grey vertex] at (v) {};
  \end{scope}
  % Row 2, column 4: H1 minus v; 1,2-3,4
  \begin{scope}[shift={(12.30,-5.10)},yscale=0.88]
    \GCoordsHone
    \draw[combined path] (TR)--(TM);
    \draw[combined path] (TM)--(u);
    \draw[combined path] (u)--(LM);
    \draw[combined path] (LM)--(l);
    \draw[combined path] (l)--(BM);
    \draw[combined path] (BM)--(r);
    \draw[combined path] (r)--(RM);
    \draw[combined path] (RM)--(BR);
    \draw[combined path] (TR)--++(0.45,0.00);
    \draw[combined path] (BR)--++(0.00,-0.45);
    \draw[combined path] (TL)--++(-0.45,0.00);
    \draw[combined path] (TL)--++(0.00,0.45);
    \draw[combined path] (BL)--++(-0.45,0.00);
    \draw[combined path] (BL)--++(0.00,-0.45);
    \GEdgesHone
    \draw[grey edge] (u)--(v);
    \draw[grey edge] (l)--(v);
    \draw[grey edge] (v)--(r);
    \GNodesHone
    \node[grey vertex] at (v) {};
  \end{scope}
  % Row 3, column 1: H1 minus BM; 1-4,2,3
  \begin{scope}[shift={(0.00,-9.50)},yscale=0.88]
    \GCoordsHone
    \draw[combined path] (TL)--(TM);
    \draw[combined path] (TM)--(u);
    \draw[combined path] (u)--(RM);
    \draw[combined path] (RM)--(r);
    \draw[combined path] (r)--(v);
    \draw[combined path] (v)--(l);
    \draw[combined path] (l)--(LM);
    \draw[combined path] (LM)--(BL);
    \draw[combined path] (TL)--++(0.00,0.45);
    \draw[combined path] (BL)--++(-0.45,0.00);
    \draw[combined path] (TR)--++(0.00,0.45);
    \draw[combined path] (TR)--++(0.45,0.00);
    \draw[combined path] (BR)--++(0.45,0.00);
    \draw[combined path] (BR)--++(0.00,-0.45);
    \GEdgesHone
    \draw[grey edge] (BR)--(BM);
    \draw[grey edge] (BM)--(BL);
    \draw[grey edge] (l)--(BM);
    \draw[grey edge] (r)--(BM);
    \GNodesHone
    \node[grey vertex] at (BM) {};
  \end{scope}
  % Row 3, column 2: H1 minus BM; 1-3,2,4
  \begin{scope}[shift={(4.10,-9.50)},yscale=0.88]
    \GCoordsHone
    \draw[combined path] (TL)--(TM);
    \draw[combined path] (TM)--(u);
    \draw[combined path] (u)--(LM);
    \draw[combined path] (LM)--(l);
    \draw[combined path] (l)--(v);
    \draw[combined path] (v)--(r);
    \draw[combined path] (r)--(RM);
    \draw[combined path] (RM)--(BR);
    \draw[combined path] (TL)--++(0.00,0.45);
    \draw[combined path] (BR)--++(0.00,-0.45);
    \draw[combined path] (TR)--++(0.00,0.45);
    \draw[combined path] (TR)--++(0.45,0.00);
    \draw[combined path] (BL)--++(-0.45,0.00);
    \draw[combined path] (BL)--++(0.00,-0.45);
    \GEdgesHone
    \draw[grey edge] (BR)--(BM);
    \draw[grey edge] (BM)--(BL);
    \draw[grey edge] (l)--(BM);
    \draw[grey edge] (r)--(BM);
    \GNodesHone
    \node[grey vertex] at (BM) {};
  \end{scope}
  % Row 3, column 3: H1 minus BM; 1,2-4,3
  \begin{scope}[shift={(8.20,-9.50)},yscale=0.88]
    \GCoordsHone
    \draw[combined path] (TR)--(TM);
    \draw[combined path] (TM)--(u);
    \draw[combined path] (u)--(RM);
    \draw[combined path] (RM)--(r);
    \draw[combined path] (r)--(v);
    \draw[combined path] (v)--(l);
    \draw[combined path] (l)--(LM);
    \draw[combined path] (LM)--(BL);
    \draw[combined path] (TR)--++(0.45,0.00);
    \draw[combined path] (BL)--++(-0.45,0.00);
    \draw[combined path] (TL)--++(-0.45,0.00);
    \draw[combined path] (TL)--++(0.00,0.45);
    \draw[combined path] (BR)--++(0.45,0.00);
    \draw[combined path] (BR)--++(0.00,-0.45);
    \GEdgesHone
    \draw[grey edge] (BR)--(BM);
    \draw[grey edge] (BM)--(BL);
    \draw[grey edge] (l)--(BM);
    \draw[grey edge] (r)--(BM);
    \GNodesHone
    \node[grey vertex] at (BM) {};
  \end{scope}
  % Row 3, column 4: H1 minus BM; 1,2-3,4
  \begin{scope}[shift={(12.30,-9.50)},yscale=0.88]
    \GCoordsHone
    \draw[combined path] (TR)--(TM);
    \draw[combined path] (TM)--(u);
    \draw[combined path] (u)--(LM);
    \draw[combined path] (LM)--(l);
    \draw[combined path] (l)--(v);
    \draw[combined path] (v)--(r);
    \draw[combined path] (r)--(RM);
    \draw[combined path] (RM)--(BR);
    \draw[combined path] (TR)--++(0.45,0.00);
    \draw[combined path] (BR)--++(0.00,-0.45);
    \draw[combined path] (TL)--++(-0.45,0.00);
    \draw[combined path] (TL)--++(0.00,0.45);
    \draw[combined path] (BL)--++(-0.45,0.00);
    \draw[combined path] (BL)--++(0.00,-0.45);
    \GEdgesHone
    \draw[grey edge] (BR)--(BM);
    \draw[grey edge] (BM)--(BL);
    \draw[grey edge] (l)--(BM);
    \draw[grey edge] (r)--(BM);
    \GNodesHone
    \node[grey vertex] at (BM) {};
  \end{scope}
  % Row 4, column 1: H1 minus LM; 1-4,2,3
  \begin{scope}[shift={(0.00,-13.90)},yscale=0.88]
    \GCoordsHone
    \draw[combined path] (TL)--(TM);
    \draw[combined path] (TM)--(u);
    \draw[combined path] (u)--(RM);
    \draw[combined path] (RM)--(r);
    \draw[combined path] (r)--(v);
    \draw[combined path] (v)--(l);
    \draw[combined path] (l)--(BM);
    \draw[combined path] (BM)--(BL);
    \draw[combined path] (TL)--++(0.00,0.45);
    \draw[combined path] (BL)--++(-0.45,0.00);
    \draw[combined path] (TR)--++(0.00,0.45);
    \draw[combined path] (TR)--++(0.45,0.00);
    \draw[combined path] (BR)--++(0.45,0.00);
    \draw[combined path] (BR)--++(0.00,-0.45);
    \GEdgesHone
    \draw[grey edge] (BL)--(LM);
    \draw[grey edge] (LM)--(TL);
    \draw[grey edge] (LM)--(u);
    \draw[grey edge] (LM)--(l);
    \GNodesHone
    \node[grey vertex] at (LM) {};
  \end{scope}
  % Row 4, column 2: H1 minus LM; 1-3,2,4
  \begin{scope}[shift={(4.10,-13.90)},yscale=0.88]
    \GCoordsHone
    \draw[combined path] (TL)--(TM);
    \draw[combined path] (TM)--(u);
    \draw[combined path] (u)--(RM);
    \draw[combined path] (RM)--(r);
    \draw[combined path] (r)--(v);
    \draw[combined path] (v)--(l);
    \draw[combined path] (l)--(BM);
    \draw[combined path] (BM)--(BR);
    \draw[combined path] (TL)--++(0.00,0.45);
    \draw[combined path] (BR)--++(0.00,-0.45);
    \draw[combined path] (TR)--++(0.00,0.45);
    \draw[combined path] (TR)--++(0.45,0.00);
    \draw[combined path] (BL)--++(-0.45,0.00);
    \draw[combined path] (BL)--++(0.00,-0.45);
    \GEdgesHone
    \draw[grey edge] (BL)--(LM);
    \draw[grey edge] (LM)--(TL);
    \draw[grey edge] (LM)--(u);
    \draw[grey edge] (LM)--(l);
    \GNodesHone
    \node[grey vertex] at (LM) {};
  \end{scope}
  % Row 4, column 3: H1 minus LM; 1,2-4,3
  \begin{scope}[shift={(8.20,-13.90)},yscale=0.88]
    \GCoordsHone
    \draw[combined path] (TR)--(TM);
    \draw[combined path] (TM)--(u);
    \draw[combined path] (u)--(RM);
    \draw[combined path] (RM)--(r);
    \draw[combined path] (r)--(v);
    \draw[combined path] (v)--(l);
    \draw[combined path] (l)--(BM);
    \draw[combined path] (BM)--(BL);
    \draw[combined path] (TR)--++(0.45,0.00);
    \draw[combined path] (BL)--++(-0.45,0.00);
    \draw[combined path] (TL)--++(-0.45,0.00);
    \draw[combined path] (TL)--++(0.00,0.45);
    \draw[combined path] (BR)--++(0.45,0.00);
    \draw[combined path] (BR)--++(0.00,-0.45);
    \GEdgesHone
    \draw[grey edge] (BL)--(LM);
    \draw[grey edge] (LM)--(TL);
    \draw[grey edge] (LM)--(u);
    \draw[grey edge] (LM)--(l);
    \GNodesHone
    \node[grey vertex] at (LM) {};
  \end{scope}
  % Row 4, column 4: H1 minus LM; 1,2-3,4
  \begin{scope}[shift={(12.30,-13.90)},yscale=0.88]
    \GCoordsHone
    \draw[combined path] (TR)--(TM);
    \draw[combined path] (TM)--(u);
    \draw[combined path] (u)--(RM);
    \draw[combined path] (RM)--(r);
    \draw[combined path] (r)--(v);
    \draw[combined path] (v)--(l);
    \draw[combined path] (l)--(BM);
    \draw[combined path] (BM)--(BR);
    \draw[combined path] (TR)--++(0.45,0.00);
    \draw[combined path] (BR)--++(0.00,-0.45);
    \draw[combined path] (TL)--++(-0.45,0.00);
    \draw[combined path] (TL)--++(0.00,0.45);
    \draw[combined path] (BL)--++(-0.45,0.00);
    \draw[combined path] (BL)--++(0.00,-0.45);
    \GEdgesHone
    \draw[grey edge] (BL)--(LM);
    \draw[grey edge] (LM)--(TL);
    \draw[grey edge] (LM)--(u);
    \draw[grey edge] (LM)--(l);
    \GNodesHone
    \node[grey vertex] at (LM) {};
  \end{scope}
  % Row 5, column 1: H1 minus TM; 1-4,2,3
  \begin{scope}[shift={(0.00,-18.30)},yscale=0.88]
    \GCoordsHone
    \draw[combined path] (TL)--(LM);
    \draw[combined path] (LM)--(l);
    \draw[combined path] (l)--(v);
    \draw[combined path] (v)--(u);
    \draw[combined path] (u)--(RM);
    \draw[combined path] (RM)--(r);
    \draw[combined path] (r)--(BM);
    \draw[combined path] (BM)--(BL);
    \draw[combined path] (TL)--++(0.00,0.45);
    \draw[combined path] (BL)--++(-0.45,0.00);
    \draw[combined path] (TR)--++(0.00,0.45);
    \draw[combined path] (TR)--++(0.45,0.00);
    \draw[combined path] (BR)--++(0.45,0.00);
    \draw[combined path] (BR)--++(0.00,-0.45);
    \GEdgesHone
    \draw[grey edge] (TL)--(TM);
    \draw[grey edge] (TM)--(TR);
    \draw[grey edge] (TM)--(u);
    \GNodesHone
    \node[grey vertex] at (TM) {};
  \end{scope}
  % Row 5, column 2: H1 minus TM; 1-3,2,4
  \begin{scope}[shift={(4.10,-18.30)},yscale=0.88]
    \GCoordsHone
    \draw[combined path] (TL)--(LM);
    \draw[combined path] (LM)--(l);
    \draw[combined path] (l)--(BM);
    \draw[combined path] (BM)--(r);
    \draw[combined path] (r)--(v);
    \draw[combined path] (v)--(u);
    \draw[combined path] (u)--(RM);
    \draw[combined path] (RM)--(BR);
    \draw[combined path] (TL)--++(0.00,0.45);
    \draw[combined path] (BR)--++(0.00,-0.45);
    \draw[combined path] (TR)--++(0.00,0.45);
    \draw[combined path] (TR)--++(0.45,0.00);
    \draw[combined path] (BL)--++(-0.45,0.00);
    \draw[combined path] (BL)--++(0.00,-0.45);
    \GEdgesHone
    \draw[grey edge] (TL)--(TM);
    \draw[grey edge] (TM)--(TR);
    \draw[grey edge] (TM)--(u);
    \GNodesHone
    \node[grey vertex] at (TM) {};
  \end{scope}
  % Row 5, column 3: H1 minus TM; 1,2-4,3
  \begin{scope}[shift={(8.20,-18.30)},yscale=0.88]
    \GCoordsHone
    \draw[combined path] (TR)--(RM);
    \draw[combined path] (RM)--(r);
    \draw[combined path] (r)--(BM);
    \draw[combined path] (BM)--(l);
    \draw[combined path] (l)--(v);
    \draw[combined path] (v)--(u);
    \draw[combined path] (u)--(LM);
    \draw[combined path] (LM)--(BL);
    \draw[combined path] (TR)--++(0.45,0.00);
    \draw[combined path] (BL)--++(-0.45,0.00);
    \draw[combined path] (TL)--++(-0.45,0.00);
    \draw[combined path] (TL)--++(0.00,0.45);
    \draw[combined path] (BR)--++(0.45,0.00);
    \draw[combined path] (BR)--++(0.00,-0.45);
    \GEdgesHone
    \draw[grey edge] (TL)--(TM);
    \draw[grey edge] (TM)--(TR);
    \draw[grey edge] (TM)--(u);
    \GNodesHone
    \node[grey vertex] at (TM) {};
  \end{scope}
  % Row 5, column 4: H1 minus TM; 1,2-3,4
  \begin{scope}[shift={(12.30,-18.30)},yscale=0.88]
    \GCoordsHone
    \draw[combined path] (TR)--(RM);
    \draw[combined path] (RM)--(r);
    \draw[combined path] (r)--(v);
    \draw[combined path] (v)--(u);
    \draw[combined path] (u)--(LM);
    \draw[combined path] (LM)--(l);
    \draw[combined path] (l)--(BM);
    \draw[combined path] (BM)--(BR);
    \draw[combined path] (TR)--++(0.45,0.00);
    \draw[combined path] (BR)--++(0.00,-0.45);
    \draw[combined path] (TL)--++(-0.45,0.00);
    \draw[combined path] (TL)--++(0.00,0.45);
    \draw[combined path] (BL)--++(-0.45,0.00);
    \draw[combined path] (BL)--++(0.00,-0.45);
    \GEdgesHone
    \draw[grey edge] (TL)--(TM);
    \draw[grey edge] (TM)--(TR);
    \draw[grey edge] (TM)--(u);
    \GNodesHone
    \node[grey vertex] at (TM) {};
  \end{scope}
  \node[font=\normalsize,text=partpurple] at (1.30,4.60) {$1\!-\!4,\;2,\;3$};
  \node[font=\normalsize,text=partpurple] at (5.40,4.60) {$1\!-\!3,\;2,\;4$};
  \node[font=\normalsize,text=partpurple] at (9.50,4.60) {$1,\;2\!-\!4,\;3$};
  \node[font=\normalsize,text=partpurple] at (13.60,4.60) {$1,\;2\!-\!3,\;4$};
  \draw[gray!55,line width=0.5pt] (-1.15,-0.77)--(14.95,-0.77);
  \node[font=\Large] at (-0.95,1.54) {$H$};
  \node[font=\Large] at (-0.95,-10.16) {$H_1$};
\end{tikzpicture}%
}
\begingroup\footnotesize\caption{{If $v$ is purple then each $H_i - v$ has cycle fragments with endpoint decompositions spanning the four possible endpoint decompositions in $H- w$. (Continued below)}}
\label{fig:hypopurple}
\endgroup
\end{figure}
\clearpage
\begin{figure}[p]
\centering
\resizebox{!}{0.70\textheight}{%
\begin{tikzpicture}
  \tikzset{combined path/.style={draw=pathorange,line width=8pt,line cap=round,line join=round}}
  % Row 6, column 1: H2 minus TM; 1-4,2,3
  \begin{scope}[shift={(0.00,-23.40)},yscale=0.88]
    \GCoordsHtwo
    \draw[combined path] (TL)--(LM);
    \draw[combined path] (LM)--(d);
    \draw[combined path] (d)--(a);
    \draw[combined path] (a)--(b);
    \draw[combined path] (b)--(RM);
    \draw[combined path] (RM)--(c);
    \draw[combined path] (c)--(e);
    \draw[combined path] (f) .. controls (0.84,0.88) and (1.3,0.96) .. (e);
    \draw[combined path] (f)--(BM);
    \draw[combined path] (BM)--(BL);
    \draw[combined path] (TL)--++(0.00,0.45);
    \draw[combined path] (BL)--++(-0.45,0.00);
    \draw[combined path] (TR)--++(0.00,0.45);
    \draw[combined path] (TR)--++(0.45,0.00);
    \draw[combined path] (BR)--++(0.45,0.00);
    \draw[combined path] (BR)--++(0.00,-0.45);
    \GEdgesHtwo
    \draw[grey edge] (TL)--(TM);
    \draw[grey edge] (TM)--(TR);
    \draw[grey edge] (TM)--(b);
    \GNodesHtwo
    \node[grey vertex] at (TM) {};
  \end{scope}
  % Row 6, column 2: H2 minus TM; 1-3,2,4
  \begin{scope}[shift={(4.10,-23.40)},yscale=0.88]
    \GCoordsHtwo
    \draw[combined path] (TL)--(LM);
    \draw[combined path] (LM)--(d);
    \draw[combined path] (d)--(a);
    \draw[combined path] (a)--(b);
    \draw[combined path] (b)--(RM);
    \draw[combined path] (RM)--(c);
    \draw[combined path] (c)--(e);
    \draw[combined path] (f) .. controls (0.84,0.88) and (1.3,0.96) .. (e);
    \draw[combined path] (f)--(BM);
    \draw[combined path] (BM)--(BR);
    \draw[combined path] (TL)--++(0.00,0.45);
    \draw[combined path] (BR)--++(0.00,-0.45);
    \draw[combined path] (TR)--++(0.00,0.45);
    \draw[combined path] (TR)--++(0.45,0.00);
    \draw[combined path] (BL)--++(-0.45,0.00);
    \draw[combined path] (BL)--++(0.00,-0.45);
    \GEdgesHtwo
    \draw[grey edge] (TL)--(TM);
    \draw[grey edge] (TM)--(TR);
    \draw[grey edge] (TM)--(b);
    \GNodesHtwo
    \node[grey vertex] at (TM) {};
  \end{scope}
  % Row 6, column 3: H2 minus TM; 1,2-4,3
  \begin{scope}[shift={(8.20,-23.40)},yscale=0.88]
    \GCoordsHtwo
    \draw[combined path] (TR)--(RM);
    \draw[combined path] (RM)--(b);
    \draw[combined path] (b)--(a);
    \draw[combined path] (a)--(c);
    \draw[combined path] (c)--(e);
    \draw[combined path] (e)--(d);
    \draw[combined path] (d)--(LM);
    \draw[combined path] (LM)--(f);
    \draw[combined path] (f)--(BM);
    \draw[combined path] (BM)--(BL);
    \draw[combined path] (TR)--++(0.45,0.00);
    \draw[combined path] (BL)--++(-0.45,0.00);
    \draw[combined path] (TL)--++(-0.45,0.00);
    \draw[combined path] (TL)--++(0.00,0.45);
    \draw[combined path] (BR)--++(0.45,0.00);
    \draw[combined path] (BR)--++(0.00,-0.45);
    \GEdgesHtwo
    \draw[grey edge] (TL)--(TM);
    \draw[grey edge] (TM)--(TR);
    \draw[grey edge] (TM)--(b);
    \GNodesHtwo
    \node[grey vertex] at (TM) {};
  \end{scope}
  % Row 6, column 4: H2 minus TM; 1,2-3,4
  \begin{scope}[shift={(12.30,-23.40)},yscale=0.88]
    \GCoordsHtwo
    \draw[combined path] (TR)--(RM);
    \draw[combined path] (RM)--(b);
    \draw[combined path] (b)--(a);
    \draw[combined path] (a)--(c);
    \draw[combined path] (c)--(e);
    \draw[combined path] (e)--(d);
    \draw[combined path] (d)--(LM);
    \draw[combined path] (LM)--(f);
    \draw[combined path] (f)--(BM);
    \draw[combined path] (BM)--(BR);
    \draw[combined path] (TR)--++(0.45,0.00);
    \draw[combined path] (BR)--++(0.00,-0.45);
    \draw[combined path] (TL)--++(-0.45,0.00);
    \draw[combined path] (TL)--++(0.00,0.45);
    \draw[combined path] (BL)--++(-0.45,0.00);
    \draw[combined path] (BL)--++(0.00,-0.45);
    \GEdgesHtwo
    \draw[grey edge] (TL)--(TM);
    \draw[grey edge] (TM)--(TR);
    \draw[grey edge] (TM)--(b);
    \GNodesHtwo
    \node[grey vertex] at (TM) {};
  \end{scope}
  % Row 7, column 1: H2 minus LM; 1-4,2,3
  \begin{scope}[shift={(0.00,-27.80)},yscale=0.88]
    \GCoordsHtwo
    \draw[combined path] (TL)--(TM);
    \draw[combined path] (TM)--(b);
    \draw[combined path] (b)--(RM);
    \draw[combined path] (RM)--(c);
    \draw[combined path] (c)--(a);
    \draw[combined path] (a)--(d);
    \draw[combined path] (d)--(e);
    \draw[combined path] (f) .. controls (0.84,0.88) and (1.3,0.96) .. (e);
    \draw[combined path] (f)--(BM);
    \draw[combined path] (BM)--(BL);
    \draw[combined path] (TL)--++(0.00,0.45);
    \draw[combined path] (BL)--++(-0.45,0.00);
    \draw[combined path] (TR)--++(0.00,0.45);
    \draw[combined path] (TR)--++(0.45,0.00);
    \draw[combined path] (BR)--++(0.45,0.00);
    \draw[combined path] (BR)--++(0.00,-0.45);
    \GEdgesHtwo
    \draw[grey edge] (BL)--(LM);
    \draw[grey edge] (LM)--(TL);
    \draw[grey edge] (LM)--(d);
    \draw[grey edge] (LM)--(f);
    \draw[grey edge] (LM) .. controls (0.15,3.08) and (1.05,3.26) .. (b);
    \GNodesHtwo
    \node[grey vertex] at (LM) {};
  \end{scope}
  % Row 7, column 2: H2 minus LM; 1-3,2,4
  \begin{scope}[shift={(4.10,-27.80)},yscale=0.88]
    \GCoordsHtwo
    \draw[combined path] (TL)--(TM);
    \draw[combined path] (TM)--(b);
    \draw[combined path] (b)--(RM);
    \draw[combined path] (RM)--(c);
    \draw[combined path] (c)--(a);
    \draw[combined path] (a)--(d);
    \draw[combined path] (d)--(e);
    \draw[combined path] (f) .. controls (0.84,0.88) and (1.3,0.96) .. (e);
    \draw[combined path] (f)--(BM);
    \draw[combined path] (BM)--(BR);
    \draw[combined path] (TL)--++(0.00,0.45);
    \draw[combined path] (BR)--++(0.00,-0.45);
    \draw[combined path] (TR)--++(0.00,0.45);
    \draw[combined path] (TR)--++(0.45,0.00);
    \draw[combined path] (BL)--++(-0.45,0.00);
    \draw[combined path] (BL)--++(0.00,-0.45);
    \GEdgesHtwo
    \draw[grey edge] (BL)--(LM);
    \draw[grey edge] (LM)--(TL);
    \draw[grey edge] (LM)--(d);
    \draw[grey edge] (LM)--(f);
    \draw[grey edge] (LM) .. controls (0.15,3.08) and (1.05,3.26) .. (b);
    \GNodesHtwo
    \node[grey vertex] at (LM) {};
  \end{scope}
  % Row 7, column 3: H2 minus LM; 1,2-4,3
  \begin{scope}[shift={(8.20,-27.80)},yscale=0.88]
    \GCoordsHtwo
    \draw[combined path] (TR)--(TM);
    \draw[combined path] (TM)--(b);
    \draw[combined path] (b)--(RM);
    \draw[combined path] (RM)--(c);
    \draw[combined path] (c)--(a);
    \draw[combined path] (a)--(d);
    \draw[combined path] (d)--(e);
    \draw[combined path] (f) .. controls (0.84,0.88) and (1.3,0.96) .. (e);
    \draw[combined path] (f)--(BM);
    \draw[combined path] (BM)--(BL);
    \draw[combined path] (TR)--++(0.45,0.00);
    \draw[combined path] (BL)--++(-0.45,0.00);
    \draw[combined path] (TL)--++(-0.45,0.00);
    \draw[combined path] (TL)--++(0.00,0.45);
    \draw[combined path] (BR)--++(0.45,0.00);
    \draw[combined path] (BR)--++(0.00,-0.45);
    \GEdgesHtwo
    \draw[grey edge] (BL)--(LM);
    \draw[grey edge] (LM)--(TL);
    \draw[grey edge] (LM)--(d);
    \draw[grey edge] (LM)--(f);
    \draw[grey edge] (LM) .. controls (0.15,3.08) and (1.05,3.26) .. (b);
    \GNodesHtwo
    \node[grey vertex] at (LM) {};
  \end{scope}
  % Row 7, column 4: H2 minus LM; 1,2-3,4
  \begin{scope}[shift={(12.30,-27.80)},yscale=0.88]
    \GCoordsHtwo
    \draw[combined path] (TR)--(TM);
    \draw[combined path] (TM)--(b);
    \draw[combined path] (b)--(RM);
    \draw[combined path] (RM)--(c);
    \draw[combined path] (c)--(a);
    \draw[combined path] (a)--(d);
    \draw[combined path] (d)--(e);
    \draw[combined path] (f) .. controls (0.84,0.88) and (1.3,0.96) .. (e);
    \draw[combined path] (f)--(BM);
    \draw[combined path] (BM)--(BR);
    \draw[combined path] (TR)--++(0.45,0.00);
    \draw[combined path] (BR)--++(0.00,-0.45);
    \draw[combined path] (TL)--++(-0.45,0.00);
    \draw[combined path] (TL)--++(0.00,0.45);
    \draw[combined path] (BL)--++(-0.45,0.00);
    \draw[combined path] (BL)--++(0.00,-0.45);
    \GEdgesHtwo
    \draw[grey edge] (BL)--(LM);
    \draw[grey edge] (LM)--(TL);
    \draw[grey edge] (LM)--(d);
    \draw[grey edge] (LM)--(f);
    \draw[grey edge] (LM) .. controls (0.15,3.08) and (1.05,3.26) .. (b);
    \GNodesHtwo
    \node[grey vertex] at (LM) {};
  \end{scope}
  % Row 8, column 1: H2 minus a; 1-4,2,3
  \begin{scope}[shift={(0.00,-32.20)},yscale=0.88]
    \GCoordsHtwo
    \draw[combined path] (TL)--(TM);
    \draw[combined path] (TM)--(b);
    \draw[combined path] (b)--(RM);
    \draw[combined path] (RM)--(c);
    \draw[combined path] (c)--(e);
    \draw[combined path] (e)--(d);
    \draw[combined path] (d)--(LM);
    \draw[combined path] (LM)--(f);
    \draw[combined path] (f)--(BM);
    \draw[combined path] (BM)--(BL);
    \draw[combined path] (TL)--++(0.00,0.45);
    \draw[combined path] (BL)--++(-0.45,0.00);
    \draw[combined path] (TR)--++(0.00,0.45);
    \draw[combined path] (TR)--++(0.45,0.00);
    \draw[combined path] (BR)--++(0.45,0.00);
    \draw[combined path] (BR)--++(0.00,-0.45);
    \GEdgesHtwo
    \draw[grey edge] (a)--(b);
    \draw[grey edge] (a)--(c);
    \draw[grey edge] (a)--(d);
    \GNodesHtwo
    \node[grey vertex] at (a) {};
  \end{scope}
  % Row 8, column 2: H2 minus a; 1-3,2,4
  \begin{scope}[shift={(4.10,-32.20)},yscale=0.88]
    \GCoordsHtwo
    \draw[combined path] (TL)--(TM);
    \draw[combined path] (TM)--(b);
    \draw[combined path] (b)--(RM);
    \draw[combined path] (RM)--(c);
    \draw[combined path] (c)--(e);
    \draw[combined path] (e)--(d);
    \draw[combined path] (d)--(LM);
    \draw[combined path] (LM)--(f);
    \draw[combined path] (f)--(BM);
    \draw[combined path] (BM)--(BR);
    \draw[combined path] (TL)--++(0.00,0.45);
    \draw[combined path] (BR)--++(0.00,-0.45);
    \draw[combined path] (TR)--++(0.00,0.45);
    \draw[combined path] (TR)--++(0.45,0.00);
    \draw[combined path] (BL)--++(-0.45,0.00);
    \draw[combined path] (BL)--++(0.00,-0.45);
    \GEdgesHtwo
    \draw[grey edge] (a)--(b);
    \draw[grey edge] (a)--(c);
    \draw[grey edge] (a)--(d);
    \GNodesHtwo
    \node[grey vertex] at (a) {};
  \end{scope}
  % Row 8, column 3: H2 minus a; 1,2-4,3
  \begin{scope}[shift={(8.20,-32.20)},yscale=0.88]
    \GCoordsHtwo
    \draw[combined path] (TR)--(TM);
    \draw[combined path] (TM)--(b);
    \draw[combined path] (b)--(RM);
    \draw[combined path] (RM)--(c);
    \draw[combined path] (c)--(e);
    \draw[combined path] (e)--(d);
    \draw[combined path] (d)--(LM);
    \draw[combined path] (LM)--(f);
    \draw[combined path] (f)--(BM);
    \draw[combined path] (BM)--(BL);
    \draw[combined path] (TR)--++(0.45,0.00);
    \draw[combined path] (BL)--++(-0.45,0.00);
    \draw[combined path] (TL)--++(-0.45,0.00);
    \draw[combined path] (TL)--++(0.00,0.45);
    \draw[combined path] (BR)--++(0.45,0.00);
    \draw[combined path] (BR)--++(0.00,-0.45);
    \GEdgesHtwo
    \draw[grey edge] (a)--(b);
    \draw[grey edge] (a)--(c);
    \draw[grey edge] (a)--(d);
    \GNodesHtwo
    \node[grey vertex] at (a) {};
  \end{scope}
  % Row 8, column 4: H2 minus a; 1,2-3,4
  \begin{scope}[shift={(12.30,-32.20)},yscale=0.88]
    \GCoordsHtwo
    \draw[combined path] (TR)--(TM);
    \draw[combined path] (TM)--(b);
    \draw[combined path] (b)--(RM);
    \draw[combined path] (RM)--(c);
    \draw[combined path] (c)--(e);
    \draw[combined path] (e)--(d);
    \draw[combined path] (d)--(LM);
    \draw[combined path] (LM)--(f);
    \draw[combined path] (f)--(BM);
    \draw[combined path] (BM)--(BR);
    \draw[combined path] (TR)--++(0.45,0.00);
    \draw[combined path] (BR)--++(0.00,-0.45);
    \draw[combined path] (TL)--++(-0.45,0.00);
    \draw[combined path] (TL)--++(0.00,0.45);
    \draw[combined path] (BL)--++(-0.45,0.00);
    \draw[combined path] (BL)--++(0.00,-0.45);
    \GEdgesHtwo
    \draw[grey edge] (a)--(b);
    \draw[grey edge] (a)--(c);
    \draw[grey edge] (a)--(d);
    \GNodesHtwo
    \node[grey vertex] at (a) {};
  \end{scope}
  % Row 9, column 1: H3 minus TM; 1-4,2,3
  \begin{scope}[shift={(0.00,-37.30)},yscale=0.88]
    \GCoordsHthree
    \draw[combined path] (TL)--(LM);
    \draw[combined path] (LM)--(a);
    \draw[combined path] (a)--(c);
    \draw[combined path] (c)--(b);
    \draw[combined path] (b)--(RM);
    \draw[combined path] (f) .. controls (1.05,0.38) and (2.4,0.42) .. (RM);
    \draw[combined path] (f)--(d);
    \draw[combined path] (d)--(e);
    \draw[combined path] (e)--(BM);
    \draw[combined path] (BM)--(BL);
    \draw[combined path] (TL)--++(0.00,0.45);
    \draw[combined path] (BL)--++(-0.45,0.00);
    \draw[combined path] (TR)--++(0.00,0.45);
    \draw[combined path] (TR)--++(0.45,0.00);
    \draw[combined path] (BR)--++(0.45,0.00);
    \draw[combined path] (BR)--++(0.00,-0.45);
    \GEdgesHthree
    \draw[grey edge] (TL)--(TM);
    \draw[grey edge] (TM)--(TR);
    \draw[grey edge] (TM)--(a);
    \draw[grey edge] (TM)--(b);
    \GNodesHthree
    \node[grey vertex] at (TM) {};
  \end{scope}
  % Row 9, column 2: H3 minus TM; 1-3,2,4
  \begin{scope}[shift={(4.10,-37.30)},yscale=0.88]
    \GCoordsHthree
    \draw[combined path] (TL)--(LM);
    \draw[combined path] (LM)--(a);
    \draw[combined path] (a)--(c);
    \draw[combined path] (c)--(e);
    \draw[combined path] (e)--(BM);
    \draw[combined path] (BM)--(f);
    \draw[combined path] (f)--(d);
    \draw[combined path] (d)--(b);
    \draw[combined path] (b)--(RM);
    \draw[combined path] (RM)--(BR);
    \draw[combined path] (TL)--++(0.00,0.45);
    \draw[combined path] (BR)--++(0.00,-0.45);
    \draw[combined path] (TR)--++(0.00,0.45);
    \draw[combined path] (TR)--++(0.45,0.00);
    \draw[combined path] (BL)--++(-0.45,0.00);
    \draw[combined path] (BL)--++(0.00,-0.45);
    \GEdgesHthree
    \draw[grey edge] (TL)--(TM);
    \draw[grey edge] (TM)--(TR);
    \draw[grey edge] (TM)--(a);
    \draw[grey edge] (TM)--(b);
    \GNodesHthree
    \node[grey vertex] at (TM) {};
  \end{scope}
  % Row 9, column 3: H3 minus TM; 1,2-4,3
  \begin{scope}[shift={(8.20,-37.30)},yscale=0.88]
    \GCoordsHthree
    \draw[combined path] (TR)--(RM);
    \draw[combined path] (RM)--(b);
    \draw[combined path] (b)--(c);
    \draw[combined path] (c)--(a);
    \draw[combined path] (a)--(LM);
    \draw[combined path] (LM)--(e);
    \draw[combined path] (e)--(d);
    \draw[combined path] (d)--(f);
    \draw[combined path] (f)--(BM);
    \draw[combined path] (BM)--(BL);
    \draw[combined path] (TR)--++(0.45,0.00);
    \draw[combined path] (BL)--++(-0.45,0.00);
    \draw[combined path] (TL)--++(-0.45,0.00);
    \draw[combined path] (TL)--++(0.00,0.45);
    \draw[combined path] (BR)--++(0.45,0.00);
    \draw[combined path] (BR)--++(0.00,-0.45);
    \GEdgesHthree
    \draw[grey edge] (TL)--(TM);
    \draw[grey edge] (TM)--(TR);
    \draw[grey edge] (TM)--(a);
    \draw[grey edge] (TM)--(b);
    \GNodesHthree
    \node[grey vertex] at (TM) {};
  \end{scope}
  % Row 9, column 4: H3 minus TM; 1,2-3,4
  \begin{scope}[shift={(12.30,-37.30)},yscale=0.88]
    \GCoordsHthree
    \draw[combined path] (TR)--(RM);
    \draw[combined path] (RM)--(b);
    \draw[combined path] (b)--(c);
    \draw[combined path] (c)--(a);
    \draw[combined path] (a)--(LM);
    \draw[combined path] (LM)--(e);
    \draw[combined path] (e)--(d);
    \draw[combined path] (d)--(f);
    \draw[combined path] (f)--(BM);
    \draw[combined path] (BM)--(BR);
    \draw[combined path] (TR)--++(0.45,0.00);
    \draw[combined path] (BR)--++(0.00,-0.45);
    \draw[combined path] (TL)--++(-0.45,0.00);
    \draw[combined path] (TL)--++(0.00,0.45);
    \draw[combined path] (BL)--++(-0.45,0.00);
    \draw[combined path] (BL)--++(0.00,-0.45);
    \GEdgesHthree
    \draw[grey edge] (TL)--(TM);
    \draw[grey edge] (TM)--(TR);
    \draw[grey edge] (TM)--(a);
    \draw[grey edge] (TM)--(b);
    \GNodesHthree
    \node[grey vertex] at (TM) {};
  \end{scope}
  % Row 10, column 1: H3 minus c; 1-4,2,3
  \begin{scope}[shift={(0.00,-41.70)},yscale=0.88]
    \GCoordsHthree
    \draw[combined path] (TL)--(LM);
    \draw[combined path] (LM)--(a);
    \draw[combined path] (a)--(TM);
    \draw[combined path] (TM)--(b);
    \draw[combined path] (b)--(RM);
    \draw[combined path] (f) .. controls (1.05,0.38) and (2.4,0.42) .. (RM);
    \draw[combined path] (f)--(d);
    \draw[combined path] (d)--(e);
    \draw[combined path] (e)--(BM);
    \draw[combined path] (BM)--(BL);
    \draw[combined path] (TL)--++(0.00,0.45);
    \draw[combined path] (BL)--++(-0.45,0.00);
    \draw[combined path] (TR)--++(0.00,0.45);
    \draw[combined path] (TR)--++(0.45,0.00);
    \draw[combined path] (BR)--++(0.45,0.00);
    \draw[combined path] (BR)--++(0.00,-0.45);
    \GEdgesHthree
    \draw[grey edge] (a)--(c);
    \draw[grey edge] (c)--(b);
    \draw[grey edge] (c)--(e);
    \GNodesHthree
    \node[grey vertex] at (c) {};
  \end{scope}
  % Row 10, column 2: H3 minus c; 1-3,2,4
  \begin{scope}[shift={(4.10,-41.70)},yscale=0.88]
    \GCoordsHthree
    \draw[combined path] (TL)--(TM);
    \draw[combined path] (TM)--(a);
    \draw[combined path] (a)--(LM);
    \draw[combined path] (LM)--(e);
    \draw[combined path] (e)--(BM);
    \draw[combined path] (BM)--(f);
    \draw[combined path] (f)--(d);
    \draw[combined path] (d)--(b);
    \draw[combined path] (b)--(RM);
    \draw[combined path] (RM)--(BR);
    \draw[combined path] (TL)--++(0.00,0.45);
    \draw[combined path] (BR)--++(0.00,-0.45);
    \draw[combined path] (TR)--++(0.00,0.45);
    \draw[combined path] (TR)--++(0.45,0.00);
    \draw[combined path] (BL)--++(-0.45,0.00);
    \draw[combined path] (BL)--++(0.00,-0.45);
    \GEdgesHthree
    \draw[grey edge] (a)--(c);
    \draw[grey edge] (c)--(b);
    \draw[grey edge] (c)--(e);
    \GNodesHthree
    \node[grey vertex] at (c) {};
  \end{scope}
  % Row 10, column 3: H3 minus c; 1,2-4,3
  \begin{scope}[shift={(8.20,-41.70)},yscale=0.88]
    \GCoordsHthree
    \draw[combined path] (TR)--(RM);
    \draw[combined path] (RM)--(b);
    \draw[combined path] (b)--(TM);
    \draw[combined path] (TM)--(a);
    \draw[combined path] (a)--(LM);
    \draw[combined path] (LM)--(e);
    \draw[combined path] (e)--(d);
    \draw[combined path] (d)--(f);
    \draw[combined path] (f)--(BM);
    \draw[combined path] (BM)--(BL);
    \draw[combined path] (TR)--++(0.45,0.00);
    \draw[combined path] (BL)--++(-0.45,0.00);
    \draw[combined path] (TL)--++(-0.45,0.00);
    \draw[combined path] (TL)--++(0.00,0.45);
    \draw[combined path] (BR)--++(0.45,0.00);
    \draw[combined path] (BR)--++(0.00,-0.45);
    \GEdgesHthree
    \draw[grey edge] (a)--(c);
    \draw[grey edge] (c)--(b);
    \draw[grey edge] (c)--(e);
    \GNodesHthree
    \node[grey vertex] at (c) {};
  \end{scope}
  % Row 10, column 4: H3 minus c; 1,2-3,4
  \begin{scope}[shift={(12.30,-41.70)},yscale=0.88]
    \GCoordsHthree
    \draw[combined path] (TR)--(RM);
    \draw[combined path] (RM)--(b);
    \draw[combined path] (b)--(TM);
    \draw[combined path] (TM)--(a);
    \draw[combined path] (a)--(LM);
    \draw[combined path] (LM)--(e);
    \draw[combined path] (e)--(d);
    \draw[combined path] (d)--(f);
    \draw[combined path] (f)--(BM);
    \draw[combined path] (BM)--(BR);
    \draw[combined path] (TR)--++(0.45,0.00);
    \draw[combined path] (BR)--++(0.00,-0.45);
    \draw[combined path] (TL)--++(-0.45,0.00);
    \draw[combined path] (TL)--++(0.00,0.45);
    \draw[combined path] (BL)--++(-0.45,0.00);
    \draw[combined path] (BL)--++(0.00,-0.45);
    \GEdgesHthree
    \draw[grey edge] (a)--(c);
    \draw[grey edge] (c)--(b);
    \draw[grey edge] (c)--(e);
    \GNodesHthree
    \node[grey vertex] at (c) {};
  \end{scope}
  \node[font=\normalsize,text=partpurple] at (1.30,-18.85) {$1\!-\!4,\;2,\;3$};
  \node[font=\normalsize,text=partpurple] at (5.40,-18.85) {$1\!-\!3,\;2,\;4$};
  \node[font=\normalsize,text=partpurple] at (9.50,-18.85) {$1,\;2\!-\!4,\;3$};
  \node[font=\normalsize,text=partpurple] at (13.60,-18.85) {$1,\;2\!-\!3,\;4$};
  \node[font=\Large] at (-0.95,-26.26) {$H_2$};
  \node[font=\Large] at (-0.95,-37.96) {$H_3$};
\end{tikzpicture}%
}
\addtocounter{figure}{-1}
\begingroup\footnotesize\caption{If $v$ is purple then each $H_i - v$ has cycle fragments with endpoint decompositions spanning the four possible endpoint decompositions in $H- w$. (Continued from above)}\endgroup
\end{figure}

%\begin{figure}
%    \centering
%    \includegraphics[width=0.8\linewidth]{Fig14-H1-HypoPurple.jpg}
%    \caption{Each possible cycle fragment in $H_1 -v$ corresponds to a cycle fragment in $H - w$, for some purple $w$. Up to symmetry, there are four such vertices $v$.}
%    \label{fig:H1-HypoPurple}
%\end{figure}

%\begin{figure}
%    \centering
%    \includegraphics[width=0.8\linewidth]{Fig15-H2-HypoPurple.jpg}
%    \caption{Each possible cycle fragment in $H_2 -v$ corresponds to a cycle fragment in $H-w$, for some purple $w$. Up to symmetry, there are three such vertices $v$.}
%    \label{fig:H2-HypoPurple}
%\end{figure}

%\begin{figure}
%    \centering
%    \includegraphics[width=0.8\linewidth]{Fig16-H3-HypoPurple.jpg}
%    \caption{Each possible cycle fragment in $H_3 -v$ corresponds to a cycle fragment in $H-w$, for some purple $w$. Up to symmetry, there are two such vertices $v$.}
%    \label{fig:H3-HypoPurple}
%\end{figure}

% Figure 19: three green-deletion and five purple-deletion cases.
% Each orange component is a simple path; gray marks are deleted.
\begin{figure}[p]
\centering
\resizebox{\linewidth}{!}{%
\begin{tikzpicture}
\def\PrimeBase{%
 \draw[graph edge] (TL)--(TM)--(TR)--(RM)--(BR)--(BM)--(BL)--(LM)--cycle;
 \draw[graph edge] (TM)--(u) (LM)--(u) (u)--(v)
    (LM)--(l) (l)--(v) (v)--(r) (r)--(RM) (l)--(BM) (r)--(BM);
 \draw[external stub] (TL)--++(-.42,0) (TL)--++(0,.42)
     (TR)--++(.42,0) (TR)--++(0,.42)
     (BR)--++(.42,0) (BR)--++(0,-.42)
     (BL)--++(-.42,0) (BL)--++(0,-.42);
}
  % Case 1: deleted u; component boundaries (('TL', 'BR'),)
  \begin{scope}[shift={(3.35,4.85)},scale=.74]
    \GCoordsHone
    \draw[path mark] (TL)--(TM);
    \draw[path mark] (TM)--(TR);
    \draw[path mark] (TR)--(RM);
    \draw[path mark] (RM)--(r);
    \draw[path mark] (r)--(v);
    \draw[path mark] (v)--(l);
    \draw[path mark] (l)--(LM);
    \draw[path mark] (LM)--(BL);
    \draw[path mark] (BL)--(BM);
    \draw[path mark] (BM)--(BR);
    \draw[path mark] (TL)--++(-0.42,0);
    \draw[path mark] (BR)--++(0.42,0);
    \PrimeBase
    \draw[grey edge] (u)--(TM);
    \draw[grey edge] (u)--(LM);
    \draw[grey edge] (u)--(v);
    \GNodesHone
    \node[grey vertex] at (u) {};
    \node[font=\scriptsize,text=partpurple!70!black] at (1.3,4.13) {$ 1-3 $};
  \end{scope}
  % Case 2: deleted l; component boundaries (('TL', 'BR'),)
  \begin{scope}[shift={(6.70,4.85)},scale=.74]
    \GCoordsHone
    \draw[path mark] (TL)--(LM);
    \draw[path mark] (LM)--(BL);
    \draw[path mark] (BL)--(BM);
    \draw[path mark] (BM)--(r);
    \draw[path mark] (r)--(v);
    \draw[path mark] (v)--(u);
    \draw[path mark] (u)--(TM);
    \draw[path mark] (TM)--(TR);
    \draw[path mark] (TR)--(RM);
    \draw[path mark] (RM)--(BR);
    \draw[path mark] (TL)--++(-0.42,0);
    \draw[path mark] (BR)--++(0.42,0);
    \PrimeBase
    \draw[grey edge] (l)--(LM);
    \draw[grey edge] (l)--(v);
    \draw[grey edge] (l)--(BM);
    \GNodesHone
    \node[grey vertex] at (l) {};
    \node[font=\scriptsize,text=partpurple!70!black] at (1.3,4.13) {$ 1-3 $};
  \end{scope}
  % Case 3: deleted r; component boundaries (('TL', 'BR'),)
  \begin{scope}[shift={(10.05,4.85)},scale=.74]
    \GCoordsHone
    \draw[path mark] (TL)--(LM);
    \draw[path mark] (LM)--(BL);
    \draw[path mark] (BL)--(BM);
    \draw[path mark] (BM)--(l);
    \draw[path mark] (l)--(v);
    \draw[path mark] (v)--(u);
    \draw[path mark] (u)--(TM);
    \draw[path mark] (TM)--(TR);
    \draw[path mark] (TR)--(RM);
    \draw[path mark] (RM)--(BR);
    \draw[path mark] (TL)--++(-0.42,0);
    \draw[path mark] (BR)--++(0.42,0);
    \PrimeBase
    \draw[grey edge] (r)--(v);
    \draw[grey edge] (r)--(RM);
    \draw[grey edge] (r)--(BM);
    \GNodesHone
    \node[grey vertex] at (r) {};
    \node[font=\scriptsize,text=partpurple!70!black] at (1.3,4.13) {$ 1-3 $};
  \end{scope}
  % Case 4: deleted TM; component boundaries (('TL', 'BR'), ('TR',), ('BL',))
  \begin{scope}[shift={(0.00,0.00)},scale=.74]
    \GCoordsHone
    \draw[path mark] (TL)--(LM);
    \draw[path mark] (LM)--(u);
    \draw[path mark] (u)--(v);
    \draw[path mark] (v)--(l);
    \draw[path mark] (l)--(BM);
    \draw[path mark] (BM)--(r);
    \draw[path mark] (r)--(RM);
    \draw[path mark] (RM)--(BR);
    \draw[path mark] (TL)--++(-0.42,0);
    \draw[path mark] (BR)--++(0.42,0);
    \draw[path mark] (TR)--++(0.42,0);
    \draw[path mark] (TR)--++(0,0.42);
    \draw[path mark] (BL)--++(-0.42,0);
    \draw[path mark] (BL)--++(0,-0.42);
    \PrimeBase
    \draw[grey edge] (TM)--(TL);
    \draw[grey edge] (TM)--(TR);
    \draw[grey edge] (TM)--(u);
    \GNodesHone
    \node[grey vertex] at (TM) {};
    \node[font=\scriptsize,text=partpurple!70!black] at (1.3,4.13) {$ 1-3,\;2,\;4 $};
  \end{scope}
  % Case 5: deleted LM; component boundaries (('TL', 'BR'), ('TR',), ('BL',))
  \begin{scope}[shift={(3.35,0.00)},scale=.74]
    \GCoordsHone
    \draw[path mark] (TL)--(TM);
    \draw[path mark] (TM)--(u);
    \draw[path mark] (u)--(v);
    \draw[path mark] (v)--(l);
    \draw[path mark] (l)--(BM);
    \draw[path mark] (BM)--(r);
    \draw[path mark] (r)--(RM);
    \draw[path mark] (RM)--(BR);
    \draw[path mark] (TL)--++(-0.42,0);
    \draw[path mark] (BR)--++(0.42,0);
    \draw[path mark] (TR)--++(0.42,0);
    \draw[path mark] (TR)--++(0,0.42);
    \draw[path mark] (BL)--++(-0.42,0);
    \draw[path mark] (BL)--++(0,-0.42);
    \PrimeBase
    \draw[grey edge] (LM)--(BL);
    \draw[grey edge] (LM)--(TL);
    \draw[grey edge] (LM)--(u);
    \draw[grey edge] (LM)--(l);
    \GNodesHone
    \node[grey vertex] at (LM) {};
    \node[font=\scriptsize,text=partpurple!70!black] at (1.3,4.13) {$ 1-3,\;2,\;4 $};
  \end{scope}
  % Case 6: deleted RM; component boundaries (('TL', 'BR'), ('TR',), ('BL',))
  \begin{scope}[shift={(6.70,0.00)},scale=.74]
    \GCoordsHone
    \draw[path mark] (TL)--(TM);
    \draw[path mark] (TM)--(u);
    \draw[path mark] (u)--(LM);
    \draw[path mark] (LM)--(l);
    \draw[path mark] (l)--(v);
    \draw[path mark] (v)--(r);
    \draw[path mark] (r)--(BM);
    \draw[path mark] (BM)--(BR);
    \draw[path mark] (TL)--++(-0.42,0);
    \draw[path mark] (BR)--++(0.42,0);
    \draw[path mark] (TR)--++(0.42,0);
    \draw[path mark] (TR)--++(0,0.42);
    \draw[path mark] (BL)--++(-0.42,0);
    \draw[path mark] (BL)--++(0,-0.42);
    \PrimeBase
    \draw[grey edge] (RM)--(TR);
    \draw[grey edge] (RM)--(BR);
    \draw[grey edge] (RM)--(r);
    \GNodesHone
    \node[grey vertex] at (RM) {};
    \node[font=\scriptsize,text=partpurple!70!black] at (1.3,4.13) {$ 1-3,\;2,\;4 $};
  \end{scope}
  % Case 7: deleted v; component boundaries (('TL', 'BR'), ('TR',), ('BL',))
  \begin{scope}[shift={(10.05,0.00)},scale=.74]
    \GCoordsHone
    \draw[path mark] (TL)--(TM);
    \draw[path mark] (TM)--(u);
    \draw[path mark] (u)--(LM);
    \draw[path mark] (LM)--(l);
    \draw[path mark] (l)--(BM);
    \draw[path mark] (BM)--(r);
    \draw[path mark] (r)--(RM);
    \draw[path mark] (RM)--(BR);
    \draw[path mark] (TL)--++(-0.42,0);
    \draw[path mark] (BR)--++(0.42,0);
    \draw[path mark] (TR)--++(0.42,0);
    \draw[path mark] (TR)--++(0,0.42);
    \draw[path mark] (BL)--++(-0.42,0);
    \draw[path mark] (BL)--++(0,-0.42);
    \PrimeBase
    \draw[grey edge] (v)--(u);
    \draw[grey edge] (v)--(l);
    \draw[grey edge] (v)--(r);
    \GNodesHone
    \node[grey vertex] at (v) {};
    \node[font=\scriptsize,text=partpurple!70!black] at (1.3,4.13) {$ 1-3,\;2,\;4 $};
  \end{scope}
  % Case 8: deleted BM; component boundaries (('TL', 'BR'), ('TR',), ('BL',))
  \begin{scope}[shift={(13.40,0.00)},scale=.74]
    \GCoordsHone
    \draw[path mark] (TL)--(TM);
    \draw[path mark] (TM)--(u);
    \draw[path mark] (u)--(LM);
    \draw[path mark] (LM)--(l);
    \draw[path mark] (l)--(v);
    \draw[path mark] (v)--(r);
    \draw[path mark] (r)--(RM);
    \draw[path mark] (RM)--(BR);
    \draw[path mark] (TL)--++(-0.42,0);
    \draw[path mark] (BR)--++(0.42,0);
    \draw[path mark] (TR)--++(0.42,0);
    \draw[path mark] (TR)--++(0,0.42);
    \draw[path mark] (BL)--++(-0.42,0);
    \draw[path mark] (BL)--++(0,-0.42);
    \PrimeBase
    \draw[grey edge] (BM)--(BR);
    \draw[grey edge] (BM)--(BL);
    \draw[grey edge] (BM)--(l);
    \draw[grey edge] (BM)--(r);
    \GNodesHone
    \node[grey vertex] at (BM) {};
    \node[font=\scriptsize,text=partpurple!70!black] at (1.3,4.13) {$ 1-3,\;2,\;4 $};
  \end{scope}
\end{tikzpicture}
}
 \caption{Among the possible cycle fragments for each $H_1' -v$, one can find a fragment with endpoint decomposition $1-3$ or $1-3, \ 2, \ 4$ depending on the color of the vertex.}
    \label{fig:H1Prime-Hypo}
\end{figure}
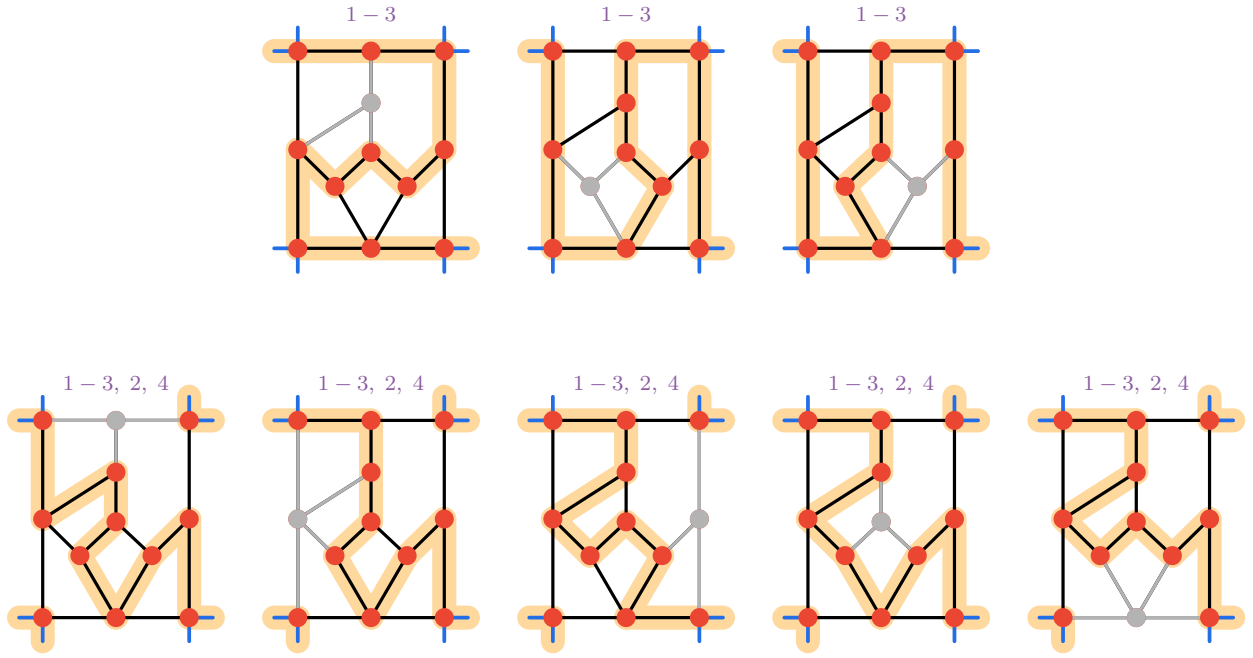

 %\begin{figure}
 %   \centering
%    \includegraphics[width=0.8\linewidth]{Fig19-H1Prime-fragments.jpg}
%    \caption{Among the possible cycle fragments for each $H_1' -v$, one can find a fragment with endpoint decomposition $1-3$ or $1-3 \ 2 \ 4$ depending on the color of the vertex.}
%    \label{fig:H1Prime-Hypo}
%\end{figure}

\end{document}